\documentclass[11pt]{article}
\usepackage{graphicx}
\usepackage{amsmath}
\usepackage{amssymb}
\usepackage{theorem}
\usepackage{fancyhdr}

\numberwithin{equation}{section}

\newtheorem{thm}{Theorem}[section]
\newtheorem{lemma}[thm]{Lemma}
\newtheorem{prop}[thm]{Proposition}
\newtheorem{cor}[thm]{Corollary}
{\theorembodyfont{\rmfamily}
\newtheorem{defn}[thm]{Definition}

\newtheorem{rmk}[thm]{Remark}
}

\newcommand{\qed}{\hfill \mbox{\raggedright \rule{.07in}{.1in}}}

\newenvironment{proof}{\vspace{1ex}\noindent{\bf
Proof}\hspace{0.5em}}{\hfill\qed\vspace{1ex}}
\newenvironment{pfof}[1]{\vspace{1ex}\noindent{\bf Proof of
#1}\hspace{0.5em}}{\hfill\qed\vspace{1ex}}

\newcommand{\R}{{\mathbb R}}

\newcommand{\C}{{\mathbb C}}
\newcommand{\Z}{{\mathbb Z}}

\newcommand{\N}{{\mathbb N}}
\newcommand{\T}{{\mathbb T}}
\newcommand{\bS}{{\mathbb S}}

\renewcommand{\H}{{\mathbb H}}

\newcommand{\overbar}[1]{\mkern 1.5mu\overline{\mkern-1.5mu#1\mkern-1.5mu}\mkern 1.5mu}

\newcommand{\overbarr}[1]{\mkern 3.5mu\overline{\mkern-3.5mu#1\mkern-0.5mu}\mkern 0.5mu}

\newcommand{\barH}{{\overbar{\H}}}

\newcommand{\intphi}{|\varphi|_1}
\newcommand{\bphi}{{\bar\varphi}}
\newcommand{\bmu}{\bar\mu}
\newcommand{\bY}{{\overbar Y}}
\newcommand{\bZ}{{\overbar Z}}
\newcommand{\bV}{{\overbar V}}
\newcommand{\bW}{{\overbar W}}
\newcommand{\tY}{{\widetilde Y}}
\newcommand{\hV}{{\widehat V}}
\newcommand{\hW}{{\widehat W}}
\newcommand{\hw}{{\widehat w}}
\newcommand{\hA}{{\widehat A}}
\newcommand{\hB}{{\widehat B}}
\newcommand{\hC}{{\widehat C}}
\newcommand{\hR}{{\widehat R}}
\newcommand{\hT}{{\widehat T}}
\newcommand{\hJ}{{\widehat J}}
\newcommand{\bF}{\overbarr F}
\newcommand{\tF}{\widetilde F}
\newcommand{\tP}{\widetilde P}
\newcommand{\tlambda}{\tilde\lambda}
\newcommand{\sumd}{{\textstyle\sum_d}}

\newcommand{\tX}{\widetilde{X}}
\newcommand{\tf}{\tilde{f}}
\newcommand{\tr}{\tilde{r}}
\renewcommand{\th}{\tilde{h}}
\newcommand{\tx}{\tilde{x}}
\newcommand{\bd}{\bar{d}(y,z)}
\newcommand{\bdy}{\bar{d}(y,y')}
\newcommand{\hy}{\hat{y}}
\newcommand{\hz}{\hat{z}}
\newcommand{\ttau}{\tilde{\tau}}

\newcommand{\cF}{{\mathcal F}}
\newcommand{\cH}{{\mathcal H}}
\newcommand{\cR}{{\mathcal R}}
\newcommand{\cS}{{\mathcal S}}
\newcommand{\cW}{{\mathcal W}}

\newcommand{\period}{{\mathcal L}}

\newcommand{\infYj}{{\SMALL\inf_{Y_j}}}
\newcommand{\infYk}{{\SMALL\inf_{\bY\!_k}}}
\newcommand{\supYk}{{\SMALL\sup_{\bY\!_k}}}
\newcommand{\infbYj}{{\SMALL\inf_{\bY\!_j}}}
\newcommand{\supbYj}{{\SMALL\sup_{\bY\!_j}}}
\newcommand{\infbFjd}{{\SMALL\inf_{\bF^jd}}}

\newcommand{\infd}{{\SMALL\inf_{d}}}

\newcommand{\eps}{\epsilon}

\newcommand{\spec}{\operatorname{spec}}
\newcommand{\supp}{\operatorname{supp}}
\newcommand{\Int}{\operatorname{Int}}
\newcommand{\diam}{\operatorname{diam}}
\newcommand{\dist}{\operatorname{dist}}
\renewcommand{\Re}{\operatorname{Re}}

\newcommand{\SMALL}{\textstyle}

\title{Polynomial decay of correlations for flows, \\ including Lorentz gas examples}

\author{
P\'eter B\'alint \thanks{
MTA-BME Stochastics Research Group,
Budapest University of Technology and Economics,
Egry J\'ozsef u. 1, H-1111 Budapest, Hungary and
 Department of Stochastics,
Budapest University of Technology and Economics
Egry J\'ozsef u. 1, H-1111, Budapest , Hungary
pet@math.bme.hu}
\and
Oliver Butterley
\thanks{ ICTP - Strada Costiera, 11 I-34151 Trieste, Italy.
oliver.butterley@ictp.it}
\and
Ian Melbourne\thanks{Mathematics Institute, University of Warwick, Coventry, CV4 7AL, UK.
i.melbourne@warwick.ac.uk}
}

\date{6 October 2017; revised 10 November 2018}

\begin{document}

 \maketitle

\begin{abstract}
We prove sharp results on polynomial decay of correlations for nonuniformly hyperbolic flows.
Applications include intermittent solenoidal flows and various Lorentz gas models including the infinite horizon Lorentz gas.
\end{abstract}

 \tableofcontents

 \section{Introduction}
 \label{sec-intro}

Let $(\Lambda,\mu_\Lambda)$ be a probability space.   Given a measure-preserving
flow $T_t:\Lambda\to \Lambda$ and observables $v,w\in L^2(\Lambda)$,
we define the correlation function
$\rho_{v,w}(t)=\int_\Lambda v\;w\circ T_t\,d\mu_\Lambda-\int_\Lambda v\,d\mu_\Lambda\int_\Lambda w\,d\mu_\Lambda$.
The flow is {\em mixing} if
$\lim_{t\to\infty}\rho_{v,w}(t)=0$ for all $v,w\in L^2(\Lambda)$.

Of interest is the rate of decay of correlations, or rate of mixing, namely the rate at
which $\rho_{v,w}$ converges to zero.
Dolgopyat~\cite{Dolgopyat98a} showed that geodesic flows on
compact surfaces of negative curvature with volume measure $\mu_\Lambda$ are exponentially mixing for H\"older observables $v, w$.
Liverani~\cite{Liverani04} extended this
result to arbitrary dimensional geodesic flows in negative curvature and
more generally to contact Anosov flows.
However, exponential mixing remains poorly understood in general.

Dolgopyat~\cite{Dolgopyat98b} considered the weaker notion of {\em rapid mixing}
(superpolynomial decay of correlations) where
$\rho_{v,w}(t)=O(t^{-q})$ for sufficiently regular observables for any fixed $q\ge1$,
and showed that rapid mixing is `prevalent' for Axiom~A flows:
it suffices that the flow contains two periodic solutions
with periods whose ratio is Diophantine.
Field~{\em et al.}~\cite{FMT07} introduced the notion of
{\em good asymptotics} and used this to
prove that amongst $C^r$ Axiom~A flows, $r\ge2$, an
open and dense set of flows is rapid mixing.

In~\cite{M07}, results on rapid mixing were obtained for nonuniformly hyperbolic semiflows, combining the rapid mixing method of Dolgopyat~\cite{Dolgopyat98b} with advances by Young~\cite{Young98,Young99} in the discrete time setting.
First results on {\em polynomial mixing} for nonuniformly hyperbolic semiflows ($\rho_{v,w}(t)=O(t^{-q})$ for some fixed $q>0$) were obtained in~\cite{M09}. Under certain assumptions  the results in~\cite{M07,M09} were established also for nonuniformly hyperbolic flows.   However, for polynomially mixing flows, the assumptions in~\cite{M09} are overly restrictive and exclude many examples including infinite horizon Lorentz gases.

In this paper, we develop the tools required to cover systematically large classes of nonuniformly hyperbolic flows.
The recent review article~\cite{M18} describes the current state of the art for rapid and polynomial decay of correlations for nonuniformly hyperbolic semiflows and flows and gives a complete self-contained proof in the case of semiflows.   Here we provide the arguments required to deal with flows.
Our results cover all of the examples in~\cite{M18}.

By~\cite{M07}, rapid mixing holds (at least typically)
for nonuniformly hyperbolic flows that are modelled as suspensions over Young towers with exponential tails~\cite{Young98}.
See also Remark~\ref{rmk:exp}.
Here we give a different proof that has a number of advantages as discussed in the introduction to~\cite{M18}.  Flows are modelled as suspensions over a uniformly hyperbolic map with an unbounded roof function
(rather than as suspensions over a nonuniformly hyperbolic map with a bounded roof function).  It then suffices to consider twisted transfer operators with one complex parameter rather than two as in~\cite{M07}, reducing from four to three the number of periodic orbits that need to be considered in Proposition~\ref{prop:tau}.  Also, the proof of rapid mixing only uses superpolynomial tails for the roof function, whereas~\cite{M07} requires exponential tails.

Examples covered by our results on rapid mixing include
finite Lorentz gases (including those with cusps, corner points, and external forcing), Lorenz attractors, and H\'enon-like attractors.
We refer to~\cite{M18} for references and further details.

Examples discussed in~\cite{M09,M18} for which polynomial mixing holds
include nonuniformly hyperbolic flows that are modelled as suspensions over Young towers with polynomial tails~\cite{Young99}.
This includes
intermittent solenoidal flows, see also Remark~\ref{rmk:Wss}.

The key example of continuous time planar periodic infinite horizon Lorentz gases is considered at length in Section~\ref{sec:Lorentz}.
In the finite horizon case,
exponential decay of correlations for the flow was proved in \cite{BaladiDemersLiverani18}.
In the infinite horizon case it has been conjectured~\cite{FriedmanMartin88,MatsuokaMartin97} that the decay rate for the flow is $O(t^{-1})$.
(An elementary argument in~\cite{BalintGouezel06} shows that this rate is optimal; the argument is reproduced in the current context in Proposition~\ref{prop:lower}.)
We obtain the conjectured decay rate $O(t^{-1})$ for planar infinite horizon Lorentz flows in Theorem~\ref{thm:Lorentz}.

\begin{rmk} \label{sec:obs}
(a) In~\cite{M09}, the decay rate $O(t^{-1})$ was proved for infinite horizon Lorentz gases at the semiflow level (after passing to a suspension over a Markov extension and quotienting out stable leaves as in Sections~\ref{sec:skew} and~\ref{sec:nonskew}).
 It was claimed in~\cite{M09} that this result held also in certain special cases for the Lorentz flow, and that the decay rate $O(t^{-(1-\eps)})$ held for all $\eps>0$ in complete generality.  The spurious factor of $t^\eps$ was then removed in an unpublished preprint ``Decay of correlations for flows with unbounded roof function, including the
infinite horizon planar periodic Lorentz gas'' by the first and third authors.
Unfortunately these results for flows do not apply to Lorentz gases since hypothesis (P1) in~\cite{M09} is not satisfied.
The situation is rectified in the current paper.  (The unpublished preprint also contained correct results on statistical limit laws such as the central limit theorem for flows with unbounded roof functions.  These aspects are completed and extended in~\cite{BalintMsub}.)
\\
(b)
A drawback of the method in this paper, already present in~\cite{Dolgopyat98b} and inherited by~\cite{M07,M09,M18}, is that at least one of the observables $v$ or $w$ is required to be $C^m$ in the flow direction.  Here $m$ can be estimated, with difficulty, but is likely to be quite large.  In the case of the infinite horizon Lorentz gas, this excludes certain physically important observables such as velocity.  A reasonable project is to attempt to combine methods in this paper with the methods for (stretched) exponential decay in~\cite{BaladiDemersLiverani18,Chernov07} to obtain the decay rate $O(t^{-1})$ for H\"older observables $v$ and $w$ (cf.\ the second open question in~\cite[Section~9]{M18}).
\end{rmk}

In Part~I of this paper, we consider results on rapid mixing and polynomial mixing for a class of suspension flows over infinite branch uniformly hyperbolic transformations~\cite{Young98}.
In Part~II, we show how these results apply to important classes of nonuniformly hyperbolic flows including those mentioned in this introduction.
The methods of proof in this paper, especially those in Part~I, are fairly straightforward adaptations of those in~\cite{M18}.
The main new contribution of the paper (Section~\ref{sec:nonskew} together with Part~II) is to develop a general framework whereby large classes of nonuniformly hyperbolic flows, including fundamental examples such as the infinite horizon Lorentz gas, are covered by these methods.

\begin{rmk} \label{rmk:self}
The paper has been structured to be as self-contained as possible.
It does not seem possible to reduce the results on flows in Part I of this paper to the results on semiflows in~\cite{M18}.  Instead, it is necessary to start from scratch and to emulate, rather than apply directly, the methods in~\cite{M18}.  Some of the more basic estimates in~\cite{M18} are applicable and
are collected together at the beginning of Sections~\ref{sec:rapid} (Lemma~\ref{lem:strat} to Proposition~\ref{prop:lambda}) and
Section~\ref{sec:poly} (Propositions~\ref{prop:fourier} to~\ref{prop:fb}),  as well as in Section~\ref{sec:trunc} (Propositions~\ref{prop:trunc},~\ref{prop:Tunif} and~\ref{prop:tP}).
Also, results on nonexistence of approximate eigenfunctions in~\cite{M18} are recalled in Sections~\ref{sec:period} and Section~\ref{sec:D}.
\end{rmk}

\paragraph{Notation}
We use the ``big $O$'' and $\ll$ notation interchangeably, writing $a_n=O(b_n)$ or $a_n\ll b_n$ if there is a constant $C>0$ such that
$a_n\le Cb_n$ for all $n\ge1$.
There are various ``universal'' constants $C_1,\dots,C_5\ge1$ depending only on the flow that do not change throughout.

\part{Mixing rates for Gibbs-Markov flows}

In this part of the paper, we state and prove results on rapid and polynomial mixing for a class of suspension flows that we call Gibbs-Markov flows.  These are suspensions over
infinite branch uniformly hyperbolic transformations~\cite{Young98}.
In Section~\ref{sec:semiflow}, we recall material on the noninvertible version, Gibbs-Markov semiflows (suspensions over infinite branch uniformly expanding maps).
In Section~\ref{sec:skew}, we consider skew product Gibbs-Markov flows
where the roof function is constant along stable leaves and state our main theorems for such flows, namely Theorem~\ref{thm:rapidflow} (rapid mixing) and
Theorem~\ref{thm:polyflow} (polynomial mixing).
These are proved in Sections~\ref{sec:rapid} and~\ref{sec:poly} respectively.
In Section~\ref{sec:nonskew}, we consider an enlarged class of Gibbs-Markov flows that can be reduced to skew products and for which
Theorems~\ref{thm:rapidflow} and~\ref{thm:polyflow} remain valid.

We quickly review
notation associated with suspension semiflows and suspension flows.
Let $(Y,\mu)$ be a probability space and let $F:Y\to Y$ be a measure-preserving transformation.
Let $\varphi:Y\to\R^+$ be an integrable roof function.
Define the suspension semiflow/flow
\begin{equation} \label{eq:susp}
F_t:Y^\varphi\to Y^\varphi, \qquad
Y^\varphi=\{(y,u)\in Y\times[0,\infty): u\in[0,\varphi(y)]\}/\sim,
\end{equation}
where
$(y,\varphi(y))\sim(Fy,0)$ and
$F_t(y,u)=(y,u+t)$ computed modulo identifications.
An $F_t$-invariant probability measure on
$Y^\varphi$ is given by $\mu^\varphi=\mu\times{\rm Lebesgue}/\int_Y\varphi\,d\mu$.

\section{Gibbs-Markov maps and semiflows}
\label{sec:semiflow}

In this section, we
review definitions and notation from~\cite[Section~3.1]{M18} for
a class of Gibbs-Markov semiflows built as suspensions over Gibbs-Markov maps.
Standard references for background material on Gibbs-Markov maps
are~\cite[Chapter~4]{Aaronson} and~\cite{AaronsonDenker01}.

Suppose that $(\bY,\bmu)$ is a probability space with
an at most countable measurable partition $\{\bY\!_j,\,j\ge1\}$
and let $\bF:\bY\to\bY$ be a measure-preserving transformation.
For $\theta\in(0,1)$, define $d_\theta(y,y')=\theta^{s(y,y')}$ where the {\em separation time} $s(y,y')$ is the least integer $n\ge0$ such that $\bF^ny$ and $\bF^ny'$ lie in distinct partition elements in $\{\bY\!_j\}$.
It is assumed that the partition $\{\bY\!_j\}$ separates trajectories, so $s(y,y')=\infty$ if and only if $y=y'$.  Then $d_\theta$ is a metric, called a {\em symbolic metric}.

A function $v:\bY\to\R$ is {\em $d_\theta$-Lipschitz} if
$|v|_\theta=\sup_{y\neq y'}|v(y)-v(y')|/d_\theta(y,y')$ is finite.
Let $\cF_\theta(\bY)$ be the Banach space of Lipschitz functions with norm $\|v\|_\theta=|v|_\infty+|v|_\theta$.

More generally (and with a slight abuse of notation), we say that a function
$v:\bY\to\R$ is {\em piecewise $d_\theta$-Lipschitz} if
$|1_{\bY\!_j}v|_\theta=\sup_{y,y'\in \bY\!_j,\,y\neq y'}|v(y)-v(y')|/d_\theta(y,y')$
is finite for all $j$.  If in addition, $\sup_j|1_{\bY\!_j}v|_\theta<\infty$ then we say that $v$ is {\em uniformly piecewise $d_\theta$-Lipschitz}.
Note that such a function $v$ is bounded on partition elements but need not be bounded on~$\bY$.

\begin{defn} \label{def:GM}
The map $\bF:\bY\to \bY$ is called a {\em (full branch) Gibbs-Markov map}
if
\begin{itemize}

\parskip=-2pt
\item $\bF|_{\bY\!_j}:\bY\!_j\to \bY$ is a measurable bijection for each $j\ge1$, and
\item The potential function $\log (d\bmu/d\bmu\circ \bF):\bY\to\R$ is uniformly piecewise $d_\theta$-Lipschitz for some $\theta\in(0,1)$.
\end{itemize}
\end{defn}

\begin{defn} \label{def:GMsemi}
A suspension semiflow $\bF_t:\bY^\varphi\to \bY^\varphi$ as in~\eqref{eq:susp} is called a {\em Gibbs-Markov semiflow} if
there exist constants $C_1\ge1$, $\theta\in(0,1)$ such that
$\bF:\bY\to \bY$ is a Gibbs-Markov map, $\varphi:\bY\to\R^+$ is an integrable roof function with $\inf\varphi>0$,
and
\begin{equation} \label{eq:infsemi}
|1_{\bY\!_j}\varphi|_{\theta}\le C_1\infbYj\varphi\quad\text{for all $j\ge1$}.
\end{equation}
\end{defn}
(Equivalently, $\log\varphi$ is uniformly piecewise $d_\theta$-Lipschitz.)
It follows that $\supbYj\varphi\le 2C_1\infbYj\varphi$ for all $j\ge1$.

For $b\in\R$, we define the operators
\[
M_b:L^\infty(\bY)\to L^\infty(\bY), \qquad M_bv=e^{ib\varphi}v\circ \bF.
\]

\begin{defn} \label{def:approx}
A subset $Z_0\subset \bY$ is a {\em finite subsystem} of $\bY$
if $Z_0=\bigcap_{n\ge0} \bF^{-n}Z$ where $Z$ is the union of finitely many
elements from the partition $\{\bY\!_j\}$.
(Note that $\bF|_{Z_0}:Z_0\to Z_0$ is a full one-sided
shift on finitely many symbols.)

We say that $M_b$ has {\em approximate eigenfunctions} on $Z_0$ if
for any $\alpha_0>0$,
there exist constants $\alpha$, $\xi>\alpha_0$ and $C>0$,
and sequences $|b_k|\to\infty$,
$\psi_k\in [0,2\pi)$, $u_k\in \cF_\theta(\bY)$ with $|u_k|\equiv1$
and $|u_k|_\theta\le C|b_k|$, such that
setting $n_k=[\xi\ln |b_k|]$,
\begin{equation}
\label{eq:approx}
|(M_{b_k}^{n_k}u_k)(y)-e^{i\psi_k}u_k(y)|\le C|b_k|^{-\alpha}
\quad\text{for all $y\in Z_0$, $k\ge1$.}
\end{equation}
\end{defn}

\begin{rmk} \label{rmk:approx}
For brevity, the statement ``Assume absence of approximate eigenfunctions''
is the assumption that there exists at least one finite subsystem $Z_0$ such that $M_b$ does not have approximate eigenfunctions on $Z_0$.
\end{rmk}

\section{Skew product Gibbs-Markov flows}
\label{sec:skew}

In this section, we recall the notion of skew product Gibbs-Markov flow~\cite[Section~4.1]{M18} and state our main results on mixing for such flows.

Let $(Y,d)$ be a metric space with $\diam Y\le1$,
and let $F:Y\to Y$ be a piecewise continuous map with
ergodic $F$-invariant probability measure~$\mu$.
Let $\cW^s$ be a cover of $Y$ by disjoint measurable subsets of $Y$ called {\em stable leaves}.
For each $y\in Y$, let $W^s(y)$ denote the stable leaf containing~$y$.
We require that $F(W^s(y))\subset W^s(Fy)$ for all $y\in Y$.

Let $\bY$ denote the space obtained from $Y$ after quotienting by $\cW^s$,
with natural projection $\bar\pi:Y\to\bY$.
We assume that the quotient map $\bF:\bY\to\bY$ is a Gibbs-Markov map as in Definition~\ref{def:GM}, with
partition $\{\bY\!_j\}$, separation time $s(y,y')$, and
 ergodic invariant probability measure $\bmu =\bar\pi_*\mu$.

Let $Y_j=\bar\pi^{-1}\bY\!_j$; these form a partition of $Y$ and each $Y_j$ is a union of stable leaves.
The separation time extends to $Y$, setting
$s(y,y')=s(\bar\pi y,\bar\pi y')$ for $y,y'\in Y$.

Next, we require that there is a
measurable subset $\tY\subset Y$ such that
for every $y\in Y$ there is a unique $\tilde y\in\tY\cap
W^s(y)$.  Let $\pi:Y\to\tY$ define the associated projection $\pi y=\tilde y$.  (Note that $\tY$ can be identified with $\bY$, but in general $\pi_*\mu\neq\bmu$.)

We assume that there are constants $C_2\ge1$, $\gamma\in(0,1)$ such that
for all $n\ge0$,
 \begin{alignat}{2} \label{eq:WsF}
 d(F^ny,F^ny') & \le C_2\gamma^n && \quad\text{for all
$y,y'\in Y$ with $y'\in W^s(y)$}, \\ \label{eq:WuF}
 d(F^ny,F^ny') & \le C_2\gamma^{s(y,y')-n} && \quad\text{for all $y,y'\in \tY$.}
\end{alignat}

Let $\varphi:Y\to\R^+$ be an integrable roof function with $\inf\varphi>0$, and define the suspension flow\footnote{Strictly speaking, $F_t$ is not always a flow since $F$ need not be invertible.  However, $F_t$ is used as a model for various flows, and it is then a flow when $\varphi$ is the first return to $Y$, so it is convenient to call it a flow.}
 $F_t:Y^\varphi\to Y^\varphi$ as in~\eqref{eq:susp}
with ergodic invariant probability measure $\mu^\varphi$.

In this subsection, we suppose that $\varphi$ is constant along stable leaves and hence projects to a well-defined roof function $\varphi:\bY\to\R^+$.
It follows that the suspension flow $F_t$ projects to
a quotient suspension semiflow $\bF_t:\bY^\varphi\to\bY^\varphi$.
We assume that
$\bF_t$ is a Gibbs-Markov semiflow (Definition~\ref{def:GMsemi}).
In particular, increasing $\gamma\in(0,1)$ if necessary,~\eqref{eq:infsemi} is satisfied in the form
\begin{equation} \label{eq:inf}
|\varphi(y)-\varphi(y')|\le C_1\infYj\varphi\,\gamma^{s(y,y')}
\quad\text{for all $y,y'\in\bY\!_j$, $j\ge1$.}
\end{equation}
We call $F_t$ a {\em skew product Gibbs-Markov flow}, and we say
that $F_t$ has {\em approximate eigenfunctions} if
$\bF_t$ has approximate eigenfunctions (Definition~\ref{def:approx}).

Fix $\eta\in(0,1]$.
For $v:Y^{\varphi}\to \R$, define
\begin{alignat*}{2}
|v|_\gamma & =\sup_{(y,u),(y',u)\in Y^\varphi,\, y\neq y'}\frac{|v(y,u)-v(y',u)|}{\varphi(y)\{d(y,y')+\gamma^{s(y,y')}\}},
& \qquad
\|v\|_\gamma & =|v|_\infty+|v|_\gamma, \\
|v|_{\infty,\eta} & =\sup_{(y,u),(y,u')\in Y^\varphi,\, u\neq u'}\frac{|v(y,u)-v(y,u')|}{|u-u'|^\eta},
& \qquad
\|v\|_{\gamma,\eta} & =\|v\|_\gamma+|v|_{\infty,\eta}.
\end{alignat*}
(Here $|u-u'|$ denotes absolute value, with $u,u'$ regarded as elements of $[0,\infty)$.)
Let $\cH_{\gamma}(Y^\varphi)$ and
$\cH_{\gamma,\eta}(Y^\varphi)$ be the spaces of observables
$v:Y^\varphi\to\R$ with $\|v\|_\gamma<\infty$ and
$\|v\|_{\gamma,\eta}<\infty$ respectively.

We say that $w:Y^\varphi\to\R$ is {\em differentiable in the flow direction} if the limit
$\partial_tw=\lim_{t\to0}(w\circ F_t-w)/t$ exists pointwise.  Note
that $\partial_tw=\frac{\partial w}{\partial u}$
on the set $\{(y,u):y\in Y,\,0< u< \varphi(y)\}$.
Define $\cH_{\gamma,0,m}(Y^\varphi)$ to consist of
observables $w:Y^\varphi\to\R$ that are $m$-times differentiable in the flow direction with derivatives in $\cH_\gamma(Y^\varphi)$,
with norm $\|w\|_{\gamma,0,m}=\sum_{j=0}^m \|\partial_t^jw\|_\gamma$.

We can now state the main theoretical results for skew product Gibbs-Markov flows.

\begin{thm} \label{thm:rapidflow}
Suppose that $F_t:Y^\varphi\to Y^\varphi$ is a skew product Gibbs-Markov flow
such that $\varphi\in L^q(Y)$ for all $q\in\N$.
Assume absence of approximate eigenfunctions.

Then for any $q\in\N$, there exists $m\ge1$ and $C>0$ such that
\[
|\rho_{v,w}(t)|\le C\|v\|_\gamma\|w\|_{\gamma,0,m} \,t^{-q}
\quad\text{for all $v\in \cH_\gamma(Y^\varphi)$, $w\in \cH_{\gamma,0,m}(Y^\varphi)$, $t>1$}.
\]
\end{thm}

\begin{thm} \label{thm:polyflow} Suppose that $F_t:Y^\varphi\to Y^\varphi$ is a skew product Gibbs-Markov flow
such that $\mu(\varphi>t)=O(t^{-\beta})$ for some $\beta>1$.
Assume absence of approximate eigenfunctions.
Then there exists $m\ge1$ and $C>0$ such that
\[
|\rho_{v,w}(t)|\le C\|v\|_{\gamma,\eta}\|w\|_{\gamma,0,m} \,t^{-(\beta-1)}
\quad\text{for all $v\in \cH_{\gamma,\eta}(Y^\varphi)$, $w\in \cH_{\gamma,0,m}(Y^\varphi)$, $t>1$}.
\]
\end{thm}

\begin{rmk}
Our result on polynomial mixing, Theorem~\ref{thm:polyflow}, implies the result on rapid mixing, Theorem~\ref{thm:rapidflow} (for a slightly more restricted class of observables).  However, the proof of Theorem~\ref{thm:rapidflow} plays a crucial role in the proof of Theorem~\ref{thm:polyflow}, justifying the movement of certain contours of integration to the imaginary axis after the truncation step in Section~\ref{sec:trunc}.
Hence, it is not possible to bypass Theorem~\ref{thm:rapidflow} even when only polynomial mixing is of interest.
\end{rmk}

These results are proved in Sections~\ref{sec:rapid} and~\ref{sec:poly} respectively.  For future reference, we mention the following estimates.
Define $\varphi_n=\sum_{j=0}^{n-1}\varphi\circ F^j$.

\begin{prop} \label{prop:varphi}
Let $\eta\in(0,\beta)$.  Then

\vspace{1ex} \noindent
(a) $\int_Y \varphi^\eta\circ F^i 1_{\{\varphi_n>t\}}\,d\mu\le (n+1)\int_Y \varphi^\eta 1_{\{\varphi>t/n\}}\,d\mu$
for all $i\ge0$, $n\ge1$, $t>0$.

\vspace{1ex} \noindent
(b)
If $\mu(\varphi>t)=O(t^{-\beta})$ for some $\beta>1$,
then
$\int_Y \varphi^\eta 1_{\{\varphi>t\}}\,d\mu = O(t^{-(\beta-\eta)})$.
\end{prop}

\begin{proof}
Writing $\varphi^\eta\circ F^i
=\varphi^\eta\circ F^i 1_{\{\varphi\circ F^i>t/n\}}+\varphi^\eta\circ F^i 1_{\{\varphi\circ F^i\le t/n\}}$, we compute that
\begin{align*}
\int_Y  \varphi^\eta & \circ F^i  1_{\{\varphi_n>t\}}\,d\mu  \\ &
=
\int_Y  \varphi^\eta\circ F^i  1_{\{\varphi\circ F^i> t/n\}}1_{\{\varphi_n>t\}} \,d\mu +
\int_Y  \varphi^\eta\circ F^i  1_{\{\varphi\circ F^i\le t/n\}}1_{\{\varphi_n>t\}} \,d\mu \\&
 \le
 \int_Y  \varphi^\eta\circ F^i  1_{\{\varphi\circ F^i> t/n\}} \,d\mu +
\sum_{j=0}^{n-1}\int_Y \Big(\frac{t}{n}\Big)^\eta 1_{\{\varphi\circ F^j>t/n\}}\,d\mu
\\ & = \int_Y \varphi^\eta 1_{\{\varphi>t/n\}}\,d\mu + n\int_Y \Big(\frac{t}{n}\Big)^\eta 1_{\{\varphi>t/n\}}\,d\mu
\le (n+1) \int_Y \varphi^\eta 1_{\{\varphi>t/n\}}\,d\mu,
\end{align*}
proving part~(a).
Part~(b) is standard (see for example~\cite[Proposition 8.5]{M18}).
\end{proof}

\section{Rapid mixing for skew product Gibbs-Markov flows}
\label{sec:rapid}

In this section, we consider skew product Gibbs-Markov flows $F_t:Y^\varphi\to Y^\varphi$ for which
the roof function $\varphi:Y\to\R^+$ lies in $L^q(Y)$ for all $q\ge1$.
For such flows, we prove Theorem~\ref{thm:rapidflow}, namely that absence of approximate eigenfunctions is a sufficient condition for rapid mixing.
For notational convenience, we suppose that $\inf\varphi\ge1$.

\subsection{Some notation and results from~\cite{M18}}
\label{sec:M18}

Let $\H=\{s\in\C:\Re s>0\}$ and $\barH=\{s\in\C:\Re s\ge0\}$.
The Laplace transform
$\hat\rho_{v,w}(s)=\int_0^\infty e^{-st}\rho_{v,w}(t)\,dt$ of the correlation function $\rho_{v,w}$ is analytic on~$\H$.

\begin{lemma}[ {\cite[Lemma~6.2]{M18}} ] \label{lem:strat}
Let $v\in L^1(Y^\varphi)$, $\eps>0$, $r\ge1$.  Suppose that
\begin{itemize}
\item[(i)]
$s\mapsto \hat\rho_{v,w}(s)$ is continuous on $\{\Re s\in[0,\eps]\}$ and
$b\mapsto \hat\rho_{v,w}(ib)$ is $C^r$ on $\R$
for all $w\in \cH_\gamma(Y^\varphi)$.
\item[(ii)]
There exist constants
$C,\alpha>0$ such that
\[
|\hat\rho_{v,w}(s)|\le C(|b|+1)^{\alpha}\|w\|_\gamma \quad\text{and}\quad
|\hat\rho_{v,w}^{(j)}(ib)|\le C(|b|+1)^{\alpha}\|w\|_\gamma,
\]
for all $w\in \cH_\gamma(Y^\varphi)$, $j\le r$,
and all $s=a+ib\in\C$ with $a\in[0,\eps]$.
\end{itemize}
Let $m=\lceil\alpha\rceil+2$.  Then
there exists a constant $C'>0$ depending only on $r$ and $\alpha$, such that
\[
|\rho_{v,w}(t)|\le CC'\|w\|_{\gamma,0,m}\, t^{-r}\quad\text{for all $w\in \cH_{\gamma,0,m}(Y^\varphi)$, $t>1$}.
\]

\vspace{-5ex}
\qed
\end{lemma}

\begin{rmk}  \label{rmk:strat}
Since $\hat\rho_{v,w}$ is not {\em a priori} well-defined on $\barH$, the conditions in this lemma should be interpreted in the usual way, namely that $\hat\rho_{v,w}:\H\to\C$ extends to a function $g:\barH\to\C$ satisfying
the desired conditions (i) and (ii).
The conclusion for $\rho_{v,w}$ then
follows from a standard uniqueness argument.

For completeness, we provide the uniqueness argument.
By~\cite[Corollary~6.1]{M18}, the inverse Laplace transform of $\hat\rho_{v,w}$
can be computed by integrating along a contour in $\H$.
Since $g\equiv \hat\rho_{v,w}$ on $\H$, we can compute
the inverse Laplace transform $f$ of $g$ using the same contour, and we obtain
$\rho_{v,w}\equiv f$.  Hence $\hat\rho_{v,w}\equiv g$ is well-defined on $\barH$ and satisfies conditions (i) and (ii), so the conclusion follows from~\cite[Lemma~6.2]{M18}.
\end{rmk}

Define $v_s(y)=\int_0^{\varphi(y)}e^{su}v(y,u)\,du$ and
$\hw(s)(y)=\int_0^{\varphi(y)}e^{-su}w(y,u)\,du$.

\begin{prop}[ {\cite[Proposition~6.3 and Corollary~8.6]{M18}} ] \label{prop:poll}
Let $v,\,w\in L^\infty(Y^\varphi)$ with $\int_{Y^\varphi}v\,d\mu^\varphi=0$.
Then $\hat\rho_{v,w}=\sum_{n=0}^\infty \hJ_n$ on $\H$ where
$\hJ_n$ is the Laplace transform of an $L^\infty$ function $J_n:[0,\infty)\to\R$ for $n\ge0$, and
\[
\hJ_n(s)  =
 \intphi^{-1}{\SMALL\int}_Y e^{-s\varphi_n}v_s \,\hw(s)\circ F^n\,d\mu
\qquad \text{for all $s\in\barH$, $n\ge1$.}
\]
Moreover, $|J_0(t)|=O(|v|_\infty|w|_\infty\,t^{-(\beta-1)})$.\footnote{
All series that we consider on $\H$ are absolutely convergent for elementary reasons.  Details are given in Lemma~\ref{lem:ABC} but are generally omitted.}  \qed
\end{prop}

Let $R:L^1(\bY)\to L^1(\bY)$ denote the transfer operator corresponding to
the Gibbs-Markov quotient map $\bF:\bY\to \bY$.
So $\int_\bY v\,w\circ \bF\,d\bmu=\int_\bY Rv\,w\,d\bmu$ for all $v\in L^1(\bY)$
and $w\in L^\infty(\bY)$.
Also, for $s\in\barH$, define the twisted transfer operators
\[
\hR(s):L^1(\bY)\to L^1(\bY),\qquad \hR(s)v=R(e^{-s\varphi}v).
\]

\begin{prop} \label{prop:GM}
Let $\theta\in(0,1)$ be as in Definition~\ref{def:GM}.
There is a constant $C>0$ such that
\[
\SMALL \|R^nv\|_\theta\le C\sum_{d}\bmu(d)\|1_dv\|_\theta \quad\text{for $v\in\cF_\theta(\bY)$, $n\ge1$},
\]
where the sum is over $n$-cylinders
$d=\bigcap_{i=0,\dots,n-1} \bF^{-i}\bY\!_{j_i}$, $j_0,\dots,j_{n-1}\ge1$.
\end{prop}

\begin{proof} This follows from~\cite[Corollary~7.2]{M18}.
\end{proof}

For the remainder of this subsection, we
suppose that $\mu(\varphi>t)=O(t^{-\beta})$ where $\beta>1$.
Fix $q>0$ with
\[
\max\{1,\beta-1\}<q<\beta.
\]
Let $\eta\in(0,1]$, $\gamma\in(0,1)$ are as in Section~\ref{sec:skew}.
Shrinking $\eta$ if needed, we may suppose without loss that
\[
q+2\eta<\beta,
\]
Let $\gamma_1=\gamma^\eta$ and increase $\theta$ if needed so that
$\theta\in[\gamma_1^{1/3},1)$.

A function $f:\R\to\R$ is said to be $C^q$ if $f$ is $C^{[q]}$ and $f^{([q])}$ is $(q-[q])$-H\"older.
Moreover, given $g:\R\to[0,\infty)$ and $E\subset \R$, we write $|f^{(q)}(b)|\le g(b)$ for $b\in E$ if for all $b,b'\in E$,
\[
|f^{(k)}(b)|\le g(b), \;k=0,1,\dots,[q],\;\text{and}\;
|f^{([q])}(b)- f^{([q])}(b')|\le (g(b)+g(b'))|b-b'|^{q-[q]}.
\]
For $f:\barH\to\R$ and $E\subset\barH$, we write $|f^{(q)}(s)|\le g(s)$ for $s\in E$ if
$|f^{(q)}(ib)|\le g(b)$ in the sense just given for $ib\in E$ and
$|f^{(k)}(s)|\le g(s)$ for $s\in E$, $k=0,\dots,[q]$.
The same conventions apply to operator-valued functions on $\barH$.

\begin{rmk}  Restricting to $q$ as above enables us to obtain estimates for the rapid mixing and polynomially mixing situations simultaneously hence avoiding a certain amount of repetition.  The trade off is that the proof of Theorem~\ref{thm:rapidflow} is considerably more difficult.  The reader interested only in the rapid mixing case can restrict to integer values of $q$ with greatly simplified arguments~\cite[Section~7]{M18} (also see version~3 of our preprint on arxiv).
\end{rmk}

Following~\cite[Section~7.4]{M18},
there exist constants $M_0,\,M_1$ and a scale of equivalent norms
\[
\|v\|_b=\max\Big\{|v|_\infty,\,\frac{|v|_\theta}{M_0(|b|+1)}\Big\}, \quad b\in\R,
\]
on $\cF_\theta(\bY)$ such that
\begin{equation} \label{eq:M1}
\|\hR(s)^n\|_b \le M_1\quad\text{for all $s=a+ib\in\C$ with $a\in[0,1]$ and all $n\ge1$}.
\end{equation}

\begin{prop} \label{prop:R}
There is a constant $C>0$ such that
\[
\SMALL \|\hR^{(q)}(s)\|_b \le
C \quad\text{for all $s=a+ib\in\C$ with $0\le a\le 1$.}
\]
\end{prop}

\begin{proof}  It is shown in \cite[Proposition~8.7]{M18} that
$\|\hR^{(q)}(s)\|_\theta \le C(|b|+1)$.  Using the definition of $\|\;\|_b$,
the desired estimate follows by exactly the same argument.
\end{proof}

\begin{rmk} Estimates such as those for $\hR^{(q)}$ in Proposition~\ref{prop:R} hold equally for $\hR^{(q')}$ for all $q'<q$.  We use this observation without comment throughout.
\end{rmk}

Define $\H_\delta=\barH\cap B_\delta(0)$ for $\delta>0$.
Let $\hT=(I-\hR)^{-1}$.  We have the key Dolgopyat estimate:
\begin{prop}  \label{prop:MT}
Assume absence of approximate eigenfunctions.
Then $\hT(s):\cF_\theta(\bY)\to\cF_\theta(\bY)$ is a well-defined bounded operator for $s\in\barH\setminus\{0\}$.
Moreover, for any $\delta>0$, there exists $\alpha,\,C>0$ such that
\[
\|\hT^{(q)}(s)\|_\theta\le C|b|^\alpha
\quad\text{for all $s=a+ib\in\C\setminus \H_\delta$ with $0\le a\le1$.}
\]

\vspace{-5ex}
\qed
\end{prop}

\begin{proof}
For the region $0\le a\le1$, $|b|\ge\delta$,
this is explicit in~\cite[Corollary~8.10]{M18}.
The remaining region $A=([0,1]\times[-\delta,\delta])\setminus \H_\delta$
is bounded.
Also, $1\not\in\spec\hR(s)$ for $s\in\barH\setminus\{0\}$ by~\cite[Proposition~7.8(b) and Theorem~7.10(a)]{M18}.
Hence $\|\hT^{(q)}\|_\theta$ is bounded on $A$ by Proposition~\ref{prop:R}.~
\end{proof}

\begin{prop}[ {\cite[Proposition~7.8 and Corollary~7.9]{M18}} ]
\label{prop:lambda}
There exists $\delta>0$ such that $\hR(s):\cF_\theta(\bY)\to\cF_\theta(\bY)$ has a $C^q$ family of simple eigenvalues $\lambda(s)$, $s\in\H_\delta$, isolated in $\spec \hR(s)$, with $\lambda(0)=1$,
$\lambda'(0)=-\intphi$, $|\lambda(s)|\le1$.
The corresponding spectral projections $P(s)$ form a $C^q$ family of operators on $\cF_\theta(\bY)$ with $P(0)v=\int_\bY v\,d\bmu$. \qed
\end{prop}

\subsection{Approximation of $v_s$ and $\hw(s)$}

The first step is to approximate $v_s,\,\hw(s):Y\to\C$ by functions that are constant on stable leaves and hence well-defined on $\bY$.

For $k\ge0$, define $\Delta_k:L^\infty(Y)\to L^\infty(Y)$,
\[
\Delta_kw=w\circ F^k\circ\pi-w\circ F^{k-1}\circ\pi\circ F, \;k\ge1,\quad \Delta_0w=w\circ\pi.
\]

\newpage
\begin{prop}   \label{prop:Dk}
Let $w\in L^\infty(Y)$.  Then
\begin{itemize}
\item[(a)]
$\Delta_kw$ is constant along stable leaves.
\item[(b)] $\sum_{k=0}^n (\Delta_kw)\circ F^{n-k}=w\circ F^n\circ \pi$.
\end{itemize}
\end{prop}

\begin{proof}
Part (a) is immediate from the definition and part~(b) follows
by induction.~
\end{proof}

Define
\[
\hV_j(s)=e^{-s\varphi\circ F^j} \Delta_jv_s,
\qquad \hW_k(s)=e^{-s\varphi_k} \Delta_k\hw(s).
\]
By Proposition~\ref{prop:Dk}(a), these can be regarded as functions
$\bV\!_j$, $\bW\!_k$ on $\bY$.
Similarly we write $\overline{\Delta_kw}\in L^\infty(\bY)$.

Also, for $k\ge0$, we define $E_k:L^\infty(Y)\to L^\infty(Y)$,
\[
E_kw=w\circ F^k-w\circ F^k\circ\pi.
\]

\begin{lemma} \label{lem:ABC}
Let $v,w\in L^\infty(Y^\varphi)$.  Then
\[
\hat\rho_{v,w} = \hJ_0+ \intphi^{-1}\Big(\sum_{n=1}^\infty \hA_n +\sum_{n=1}^\infty\sum_{k=0}^{n-1}\hB_{n,k}
 + \sum_{j=0}^\infty \sum_{k=0}^\infty \hC_{j,k}\Big),
\]
 on $\H$, where
\[
\begin{aligned}
\hA_n(s) & =\int_Y e^{-s\varphi_n}v_s\,
(E_{n-1}\hw(s))\circ F\, d\mu,
\\ \hB_{n,k}(s) & =
\int_Y e^{-s\varphi_n\circ F^n}E_nv_s
\;(\Delta_k\hw(s))\circ F^{2n-k}\,d\mu,
\\ \hC_{j,k}(s) & =\int_\bY \hR(s)^{\max\{j-k-1,0\}}\hT(s)R^{j+1}\bV\!_j(s)\;\bW\!_k(s)\,d\bmu.
\end{aligned}
\]
All of these series are absolutely convergent exponentially quickly, pointwise on $\H$.
\end{lemma}

\begin{proof}
Since this result is set in the right-half complex plane,
the final statement is elementary.  We sketch the arguments.
Let $s\in\C$ with $a=\Re s>0$.
It is clear that $|v_s|\le a^{-1}|v|_\infty e^{a\varphi}$ and $|\hw(s)|_\infty\le a^{-1}|w|_\infty$.
Hence $|\hA_n(s)|\le 2a^{-2}|v|_\infty |w|_\infty e^{-a(n-1)}$
and $|\hB_{n,k}(s)|\le 4a^{-2}|v|_\infty |w|_\infty e^{-a(n-1)}$.
Similarly, $|\hV_j(s)|_\infty\le 2a^{-1}|v|_\infty$ and $|\hW_k(s)|_\infty\le 2a^{-1}|w|_\infty e^{-ak}$.
As an operator on $L^\infty(Y)$, we have $|\hR(s)|_\infty\le e^{-a}$.
Hence $|\hC_{j,k}(s)|\le 4a^{-2}(1-e^{-a})^{-1}|v|_\infty|w|_\infty
e^{-a\max(j-1,k)}$.

By Proposition~\ref{prop:poll}, $\hat\rho_{v,w}(s)=\hJ_0(s)+
\intphi^{-1}\sum_{n=1}^\infty
\int_Y  e^{-s\varphi_n} v_s\; \hw(s)\circ F^n\,d\mu$ for $s\in\H$.
By Proposition~\ref{prop:Dk}(b), for each $n\ge1$,
\begin{align*}
\int_Y  e^{-s\varphi_n} v_s\; \hw(s)\circ F^n\,d\mu
 &=
\int_Y e^{-s\varphi_n}v_s\;\hw(s)\circ F^{n-1}\circ\pi\circ F\,d\mu
\\ & \qquad +\int_Y e^{-s\varphi_n}v_s\;(\hw(s)\circ F^{n-1}-\hw(s)\circ F^{n-1}\circ\pi)\circ F\,d\mu
\\ & =
\sum_{k=0}^{n-1}\int_Y e^{-s\varphi_n}v_s\;(\Delta_k\hw(s))\circ F^{n-k-1}\circ F\,d\mu  + \hA_n(s).
\end{align*}

Also, by Proposition~\ref{prop:Dk}(b), for each $n\ge1$, $0\le k\le n-1$,
\begin{align*}
\int_Y  e^{-s\varphi_n} v_s & \;(\Delta_k\hw(s))\circ F^{n-k-1}\circ F\,d\mu
  =
\int_Y e^{-s\varphi_n\circ F^n}v_s\circ F^n\;(\Delta_k\hw(s))\circ F^{2n-k}\,d\mu
\\ &  =
\sum_{j=0}^n \int_Y e^{-s\varphi_n\circ F^n}(\Delta_jv_s)\circ F^{n-j}\;\Delta_k\hw(s)\circ F^{2n-k}\,d\mu +\hB_{n,k}(s)
\\ & =
\sum_{j=0}^n \int_\bY e^{-s\varphi_n\circ \bF^j}\overline{\Delta_jv_s}\;\overline{\Delta_k\hw(s)}\circ \bF^{n-k+j}\,d\bmu +\hB_{n,k}(s).
\end{align*}

Next,
\begin{align*}
& \int_\bY   e^{-s\varphi_n\circ \bF^j}\overline{\Delta_jv_s}\;\overline{\Delta_k\hw(s)}\circ \bF^{n-k+j}\,d\bmu
 =
\int_\bY e^{-s\varphi_n}R^j\overline{\Delta_jv_s}\;\overline{\Delta_k\hw(s)}\circ \bF^{n-k}\,d\bmu
\\ & \quad  =
\int_\bY e^{-s\varphi_{n-k}}R^j\overline{\Delta_jv_s}\;(e^{-s\varphi_k}\overline{\Delta_k\hw(s)})\circ \bF^{n-k}\,d\bmu
 =
\int_\bY \hR(s)^{n-k}R^j\overline{\Delta_jv_s}\;\bW\!_k(s)\,d\bmu
 \\ &  \quad =
\int_\bY \hR(s)^{n-k-1}R^{j+1}(e^{-s\varphi\circ\bF^j}\overline{\Delta_jv_s})\;\bW\!_k(s)\,d\bmu
 =
\int_\bY \hR(s)^{n-k-1}R^{j+1}\bV\!_j(s)\;\bW\!_k(s)\,d\bmu.
\end{align*}

Altogether,
\[
\sum_{n=1}^\infty \int_Y  e^{-s\varphi_n}v_s\,\hw(s)\circ F^n\,d\mu
   = \sum_{n=1}^\infty \hA_n(s) +\sum_{n=1}^\infty\sum_{k=0}^{n-1}\hB_{n,k}(s)
 + C(s)
\]
where
\[
C(s)  =
\sum_{n=1}^\infty\sum_{k=0}^{n-1}\sum_{j=0}^n\int_\bY \hR(s)^{n-k-1}R^{j+1}\bV\!_j(s)\;\bW\!_k(s)\,d\bmu.
\]
Now
\begin{align*}
 & \sum_{n=1}^\infty\sum_{k=0}^{n-1}\sum_{j=0}^n\hR(s)^{n-k-1}a_j\;b_k
   =
\sum_{0\le j\le k}\sum_{n=k+1}^\infty \hR(s)^{n-k-1}a_j\;b_k
+
\sum_{j>k\ge0} \sum_{n=j}^\infty\hR(s)^{n-k-1}a_j\;b_k
\\ &  =
\sum_{k=0}^\infty \sum_{j=0}^k \hT(s)a_j\;b_k
+\sum_{j=1}^\infty \sum_{k=0}^{j-1} \hR(s)^{j-k-1}\hT(s)a_j\;b_k
=\sum_{j,k=0}^\infty \hR(s)^{\max\{j-k-1,0\}}\hT(s)a_j\;b_k.
\end{align*}
This completes the proof.
\end{proof}

For $w\in L^\infty(Y^\varphi)$, we define the approximation operators
\begin{align*}
\widetilde{\Delta}_k w(y,u) & =\begin{cases} w(F^k\pi y,u)-w(F^{k-1}\pi Fy,u) & k\ge1 \\ w(\pi y,u) & k=0 \end{cases},
\\
\widetilde{E}_k w(y,u) & = w(F^k y,u)-w(F^k\pi y,u),\, k\ge0,
\end{align*}
for $y\in Y$, $u\in[0,\varphi(F^ky)]$.

\begin{prop}   \label{prop:DeltaE}
(a) Let  $w\in \cH_\gamma(Y^\varphi)$, $k\ge0$.  Then
for all $y\in Y$, $u\in[0,\varphi(F^ky)]$,
\[
|\widetilde{\Delta}_k w(y,u)|\le 2C_2 \gamma_1^{k-1}\|w\|_\gamma \varphi(F^ky)^\eta \quad\text{and}\quad
|\widetilde{E}_k w(y,u)|\le 2C_2 \gamma_1^k|w|_\gamma \varphi(F^ky)^\eta.
\]

\noindent
(b) Let  $w\in \cH_\gamma(Y^\varphi)$, $k\ge0$.  Then
 for all $y,y'\in Y$, $u\in[0,\varphi(F^ky)]\cap[0,\varphi(F^ky')]$,
\[
|\widetilde\Delta_k w(y,u)- \widetilde\Delta_k w(y',u)|\le 4C_2\gamma_1^{s(y,y')-k}|w|_\gamma \varphi(F^ky)^\eta.
\]

\noindent
(c) Let $w\in \cH_{\gamma,\eta}(Y^\varphi)$, $k\ge0$.  Then
 for all $y\in Y$, $u,u'\in[0,\varphi(F^ky)]$,
\[
|\widetilde\Delta_k w(y,u)- \widetilde\Delta_k w(y,u')|\le 2|w|_{\infty,\eta} |u-u'|^\eta.
\]
\end{prop}

\begin{proof}
(a)
Clearly $|\widetilde{\Delta}_0 w(y,u)| \le |w|_\infty$.
By~\eqref{eq:WsF}, for $k\ge1$,
\begin{align*}
|\widetilde{\Delta}_k w(y,u)| &
\le |w|_\gamma \varphi(F^ky)
(d(F^k\pi y,F^{k-1}\pi Fy)+\gamma^{s(F^k\pi y,F^{k-1}\pi Fy)})
\\ & = |w|_\gamma \varphi(F^ky) d(F^k\pi y,F^{k-1}\pi Fy)
 \le C_2 \gamma^{k-1} |w|_\gamma\varphi(F^ky).
\end{align*}
Also, $|\widetilde{\Delta}_k w|\le 2|w|_\infty$, so
\[
|\widetilde\Delta_kw(y,u)|\le
2C_2\|w\|_\gamma\min\{1,\gamma^{k-1}\varphi(F^ky)\}
\le 2C_2\gamma_1^{k-1}\|w\|_\gamma\varphi(F^ky)^\eta.
\]
This proves the estimate for $\widetilde{\Delta}_k w$, and the estimate for $\widetilde{E}_kw$ is similar.
\\[.75ex]
(b) First suppose that $k\ge1$ and note by~\eqref{eq:WuF} that
\begin{align*}
& d(F^k\pi y,F^k\pi y')\le C_2\gamma^{s(y,y')-k},
\quad
d(F^{k-1}\pi F y,F^{k-1}\pi Fy')\le C_2\gamma^{s(y,y')-k}.
\end{align*}
It follows that
\begin{align*}
|w(F^k\pi y,u) & - w(F^k\pi y',u)|  \le |w|_\gamma
\varphi(F^ky)(d(F^k\pi y,F^k\pi y')+\gamma^{s(F^k\pi y,F^k\pi y')})
\\ & \le |w|_\gamma \varphi(F^ky)(C_2 \gamma^{s(y,y')-k}+\gamma^{s(y,y')-k})
 \le 2C_2\gamma^{s(y,y')-k}|w|_\gamma \varphi(F^ky) .
\end{align*}
Similarly,
$|w(F^{k-1}\pi F y,u)-w(F^{k-1}\pi Fy',u)|
 \le 2C_2 \gamma^{s(y,y')-k}|w|_\gamma \varphi(F^ky)$.
Hence
\begin{align*}
|\widetilde\Delta_k w(y,u)-   \widetilde\Delta_k w(y',u)|
 & \le |w(F^k\pi y,u)-w(F^k\pi y',u)|
\\ & \qquad  +
|w(F^{k-1}\pi Fy,u)-w(F^{k-1}\pi Fy',u)| \\
 & \le 4C_2 \gamma^{s(y,y')-k}|w|_\gamma\varphi(F^ky) .
\end{align*}
Also, $|\widetilde\Delta_k w(y,u)-   \widetilde\Delta_k w(y',u)|\le 4|w|_\infty$, so
\[
|\widetilde\Delta_k w(y,u)-   \widetilde\Delta_k w(y',u)|
  \le 4C_2 \gamma_1^{s(y,y')-k}|w|_\gamma\varphi(F^ky)^\eta .
\]
The case $k=0$ is the same with one term omitted.
\\[.75ex]
(c) For $k\ge1$,
\begin{align*}
|\widetilde\Delta_k w(y,u)- &  \widetilde\Delta_k w(y,u')|
 \le |w(F^k\pi y,u)-w(F^k\pi y,u')|
\\ & +
|w(F^{k-1}\pi Fy,u)-w(F^{k-1}\pi Fy,u')|\le 2|w|_{\infty,\eta}|u-u'|^\eta.
\end{align*}
The case $k=0$ is the same with one term omitted.
\end{proof}

We end this subsection by noting for all $k\ge0$ the identities
\begin{alignat*}{2}
\Delta_k v_s(y) & =\int_0^{\varphi(F^ky)}e^{su}\widetilde{\Delta}_kv(y,u)\,du, & \qquad
\Delta_k \hw(s)(y)& = \int_0^{\varphi(F^ky)}e^{-su}\widetilde{\Delta}_kw(y,u)\,du, \\
E_k v_s(y)& =\int_0^{\varphi(F^ky)}e^{su}\widetilde{E}_kv(y,u)\,du, & \qquad
E_k \hw(s)(y)& =\int_0^{\varphi(F^ky)}e^{-su}\widetilde{E}_kw(y,u)\,du.
\end{alignat*}

\subsection{Estimates for $A_n$ and $B_{n,k}$}

We continue to suppose that $\mu(\varphi>t)=O(t^{-\beta})$ where $\beta>1$,
and that $q$, $\eta$, $\gamma_1$, $\theta$ are as in Subsection~\ref{sec:M18}.
Let $c'=1/(2C_1)$.
As shown in the proofs of Propositions~\ref{prop:An} and~\ref{prop:Bnk} below,
$\hA_n$ and $\hB_{n,k}$ are Laplace transforms of $L^\infty$ functions
$A_n,\,B_{n,k}:[0,\infty)\to\R$.  In this subsection, we obtain estimates for
these functions $A_n,\,B_{n,k}$.

\begin{prop} \label{prop:Rvarphi2}
There is a constant $C>0$ such that
\[
\SMALL \int_Y \varphi\, \varphi\circ F^n 1_{\{\varphi_{n+1}>t\}}\,d\mu\le C
n\int_Y\varphi 1_{\{\varphi>c't/n\}}\,d\mu
\quad\text{for all $n\ge1$, $t>0$.}
\]
\end{prop}

\begin{proof}
Since $F$ is Gibbs-Markov, there is a constant $C_0$ (called $C_2$ in~\cite{M18}) such that
\[
\begin{aligned}
|R(\varphi1_{\{\varphi>c\}})|_\infty
& \le C_0 \sum \mu(Y_j)|1_{Y_j}\varphi|_\infty
1_{\{|1_{Y_j}\varphi|_\infty>c\}}
\\ & \le 2C_0C_1\sum\mu(Y_j)\infYj\varphi 1_{\{\infYj\varphi>c'c\}}
 \le K{\SMALL\int}_Y\varphi 1_{\{\varphi>c'c\}}\,d\mu,
\end{aligned}
\]
where $K=2C_0C_1$.  Similarly,
$|R\varphi|_\infty\le K|\varphi|_1$ and
$|R1_{\{\varphi>c\}}|_\infty\le K \mu(\varphi>c'c)$.

Now
\[
\begin{aligned}
\int_Y \varphi\, & \varphi\circ F^n 1_{\{\varphi_{n+1}>t\}}\,d\mu
 \le
\sum_{j=0}^n\int_Y \varphi\circ F^n\, \varphi 1_{\{\varphi\circ F^j>t/n\}}\,d\mu
\\ &  = \sum_{j=0}^n\int_Y \varphi\, R^n(\varphi 1_{\{\varphi\circ F^j>t/n\}})\,d\mu
 = \sum_{j=0}^n\int_Y \varphi\, R^{n-j}(1_{\{\varphi>t/n\}}R^j\varphi)\,d\mu.
\end{aligned}
\]
For $1\le j\le n-1$,
\begin{align*}
\SMALL \Big|\int_Y  \varphi\, & R^{n-j}(1_{\{\varphi>t/n\}}R^j\varphi)\,d\mu \Big|
 \le |\varphi|_1|R^{n-j}(1_{\{\varphi>t/n\}}R^j\varphi)|_\infty
\\ &\SMALL  \le |\varphi|_1|R^j\varphi|_\infty|R^{n-j}1_{\{\varphi>t/n\}}|_\infty
\le |\varphi|_1|R\varphi|_\infty|R1_{\{\varphi>t/n\}}|_\infty
\le K^2 |\varphi|_1^2 \mu(\varphi>c't/n).
\end{align*}
For $j=n$,
\[
\SMALL |\int_Y  \varphi\, R^{n-j} (1_{\{\varphi>t/n\}}R^j\varphi)\,d\mu |
\le |R\varphi|_\infty \int_Y  \varphi\, 1_{\{\varphi>t/n\}}\,d\mu
\le K|\varphi|_1  \int_Y\varphi 1_{\{\varphi>c't/n\}}\,d\mu.
\]
Finally for $j=0$,
\[
\SMALL
|\int_Y  \varphi\, R^{n-j} (1_{\{\varphi>t/n\}}R^j\varphi)\,d\mu |
\le |\varphi|_1|R(\varphi\, 1_{\{\varphi>t/n\}})|_\infty
\le K|\varphi|_1  \int_Y\varphi 1_{\{\varphi>c't/n\}}\,d\mu,
\]
completing the proof.
\end{proof}

\begin{prop} \label{prop:An}
There is a constant $C>0$ such that
\[
|A_n(t)|\le C n^\beta\gamma_1^n|v|_\infty|w|_\gamma\,(t+1)^{-(\beta-1)}
\;\text{for all $v\in L^\infty(Y^\varphi)$,
$w\in \cH_\gamma(Y^\varphi)$, $n\ge1$, $t>0$.}
\]
\end{prop}

\begin{proof}
We compute that
\begin{align*}
 \hA_n(s) & = \int_Y e^{-s\varphi_n} v_s\;(E_{n-1}\hw(s))\circ F\,d\mu
\\ & =\int_Y\int_0^{\varphi(y)}v(y,u)
\int_0^{\varphi(F^ny)}
e^{-s(\varphi_n(y)-u+u')}\widetilde{E}_{n-1}w(Fy,u')\,du'\, du\,d\mu
\\ & =\int_Y\int_0^{\varphi(y)}v(y,u)
\int_{\varphi_n(y)-u}^{\varphi_{n+1}(y)-u}
e^{-st}\widetilde{E}_{n-1}w(Fy,t-\varphi_n(y)+u)\,dt\, du\,d\mu.
\end{align*}
Hence
\begin{align*}
A_n(t)
=\int_Y\int_0^{\varphi(y)}v(y,u)
1_{\{\varphi_n(y)-u<t<\varphi_{n+1}(y)-u\}} \widetilde{E}_{n-1}w(Fy,t-\varphi_n(y)+u) \,du\,d\mu.
\end{align*}
By Proposition~\ref{prop:DeltaE}(a),
$|\widetilde{E}_{n-1}w(Fy,t-\varphi_n(y)+u)|\le 2C_2 \gamma_1^{n-1}|w|_\gamma \varphi(F^ny)^\eta$ and so
\[
\SMALL |A_n(t)|\le 2C_2 \gamma_1^{n-1}|v|_\infty|w|_\gamma
\int_Y \varphi\,\varphi\circ F^n  1_{\{\varphi_{n+1}>t\}}\,d\mu.
\]
The result follows from Propositions~\ref{prop:varphi}(b) (with $\eta=1$) and~\ref{prop:Rvarphi2}.
\end{proof}

\begin{prop} \label{prop:Bnk}
There is a constant $C>0$ such that
\[
|B_{n,k}(t)|\le C n^\beta\gamma_1^n|v|_\gamma|w|_\infty\,(t+1)^{-(\beta-1)}
\;\text{for all $v\in \cH_\gamma(Y^\varphi)$,
$w\in L^\infty(Y^\varphi)$, $n\ge1$, $k\ge0$, $t>0$.}
\]
\end{prop}

\begin{proof}
We compute that
\begin{align*}
& \hB_{n,k}(s)  =\int_Ye^{-s\varphi_n\circ F^n}
E_nv_s\; (\Delta_k\hw(s))\circ F^{2n-k} \,d\mu
\\
& \quad =\int_Y\int_0^{\varphi(F^{2n}y)}\int_0^{\varphi(F^ny)}e^{-s(\varphi_n(F^ny)-u'+u)}\widetilde{E}_nv(y,u')\widetilde{\Delta}_k w(F^{2n-k}y,u)
\,du'\,du\,d\mu \\
& \quad =\int_Y\int_0^{\varphi(F^{2n}y)}\int_{\varphi_{n-1}(F^{n+1}y)+u}^{\varphi_n(F^ny)+u}e^{-st} \widetilde{E}_nv(y,\varphi_n(F^ny)-t+u)\widetilde{\Delta}_k w(F^{2n-k}y,u) \,dt\,du\,d\mu.
\end{align*}
Hence
\begin{align*}
B_{n,k}(t) & =\int_Y\int_0^{\varphi(F^{2n}y)}
1_{\{\varphi_{n-1}(F^{n+1}y)+u<t<\varphi_n(F^ny)+u\}}
\\ &\qquad \qquad\qquad\qquad  \times
\widetilde{E}_nv(y,\varphi_n(F^ny)-t+u)\widetilde{\Delta}_k w(F^{2n-k}y,u)\,du\,d\mu.
\end{align*}
By Proposition~\ref{prop:DeltaE}(a),
$|\widetilde{E}_nv(y,\varphi_n(F^ny)-t+u)| \le 2C_2 \gamma_1^n|v|_\gamma\varphi(F^ny)$.
Also $|\widetilde{\Delta}_k w(F^{2n-k}y,u)|\le 2 |w|_\infty$.
 Hence
\begin{align*}
|B_{n,k}(t)|
& \le 2C_2 \gamma_1^n   |v|_\gamma|w|_\infty
\int_Y \varphi\circ F^{2n}\,\varphi\circ F^n1_{\{\varphi_{n+1}\circ F^n>t\}}\,d\mu
 \\ & = 2C_2 \gamma_1^n   |v|_\gamma|w|_\infty
\int_Y \varphi\, \varphi\circ F^n 1_{\{\varphi_{n+1}>t\}}\,d\mu.
\end{align*}
The result follows from Propositions~\ref{prop:varphi}(b) and~\ref{prop:Rvarphi2}.
\end{proof}

\subsection{Estimates for $\hC_{j,k}$}

For the moment, we suppose that $\mu(\varphi>t)=O(t^{-\beta})$ where $\beta>1$,
and that $q$, $\eta$, $\gamma_1$, $\theta$ are as in Subsection~\ref{sec:M18}.
First, we estimate the inverse Laplace transform $\bW\!_k(t):\bY\to\R$ associated to $\hW_k(s):Y\to\C$.

\begin{prop} \label{prop:Wk}
There is a constant $C>0$ such that
\[
|\bW\!_k(t)|_1\le C
(k+1)^{\beta+1}\gamma_1^k \|w\|_\gamma\,(t+1)^{-q}
\quad\text{for all $w\in \cH_\gamma(Y^\varphi)$, $k\ge0$, $t>0$.}
\]
\end{prop}

\begin{proof}
For all $k\ge0$,
\[
\begin{aligned}
\hW_k(s)(y)
= e^{-s\varphi_k(y)}\Delta_k\hw(s)(y)
  & =\int_0^{\varphi(F^ky)} e^{-s(\varphi_k(y)+u)}\widetilde{\Delta}_kw(y,u) \,du
\\ & =\int_{\varphi_k(y)}^{\varphi_{k+1}(y)} e^{-st}\widetilde{\Delta}_kw(y,t-\varphi_k(y)) \,dt .
\end{aligned}
\]
Hence
\[
W_k(t)(y)= 1_{\{\varphi_k(y)<t<\varphi_{k+1}(y)\}}\widetilde{\Delta}_kw(y,t-\varphi_k(y)),
\]
and
$|W_k(t)|\le 2C_2\gamma_1^{k-1} \|w\|_\gamma (\varphi\circ F^k)^\eta 1_{\{\varphi_{k+1}>t\}}$
by Proposition~\ref{prop:DeltaE}(a).
It follows that
\begin{align*}
|\bW\!_k(t)|_1 & =|W_k(t)|_1 \le \SMALL 2C_2(k+1)\gamma_1^{k-1}\|w\|_\gamma\int_Y \varphi^\eta 1_{\{\varphi>t/(k+1)\}}\,d\mu
\\ & \ll (k+1)^{\beta+1-\eta}\gamma_1^{k-1}\|w\|_\gamma \,(t+1)^{-(\beta-\eta)}
\le (k+1)^{\beta+1}\gamma_1^{k-1}\|w\|_\gamma \,(t+1)^{-q},
\end{align*}
by Proposition~\ref{prop:varphi}.
\end{proof}

\begin{prop}  \label{prop:Rn}
There exists $C>0$ such that
\[
\|(\hR^\ell)^{(q)}(s)\|_\theta \le C \ell^q(|s|+1)
\quad\text{for all $s=a+ib\in\C$ with $a\in[0,1]$ and all $\ell\ge1$}.
\]
\end{prop}

\begin{proof}
By Proposition~\ref{prop:R},
there exists a constant $M>0$ such that
$\|\hR^{(p)}(s)\|_b\le M$ for all $p\le q$.
Also $\|\hR(s)^n\|_b\le M_1$ by~\eqref{eq:M1}.

For $q\ge1$, note that $(\hR^\ell)^{(q)}$ consists of $\ell^q$ terms (counting repetitions) of the form
\[
\hR^{n_1}\hR^{(p_1)}\cdots\hR^{n_k}\hR^{(p_k)}\hR^{n_{k+1}},
\]
where $n_i\ge0$, $1\le p_i\le q$, $n_1+\dots+n_{k+1}+k=\ell$, $p_1+\dots+p_k=q$.
Since $k\le q$,
\[
\|\hR^{n_1}\hR^{(p_1)}\cdots\hR^{n_k}\hR^{(p_k)}\hR^{n_{k+1}}\|_b
\le M_1^{q+1} M^q.
\]
Hence
$\|(\hR^\ell)^{(q)}(s)\|_\theta
\le (M_0+1)(|s|+1)\|(\hR^\ell)^{(q)}(s)\|_b \ll \ell^q(|s|+1)$.
\end{proof}

\begin{prop} \label{prop:sumV}
Let $v\in \cH_{\gamma}(Y^\varphi)$.
Define $I_0(s)=\int_Y\int_0^{\varphi(y)}e^{-s(\varphi(y)-u)}v(y,u)\,du\,d\mu$.
Then
\[
\SMALL \sum_{j=0}^\infty \int_Y \hV_j\,d\mu=I_0
\quad\text{on $\barH$.}
\]
\end{prop}

\begin{proof}
For $j\ge1$,
\[
\begin{aligned}
\int_Y\hV_j(s)\,d\mu & =\int_Y\int_0^{\varphi(F^jy)}e^{-s(\varphi(F^jy)-u)}(v(F^j\pi y,u)-v(F^{j-1}\pi Fy,u)\,du\,d\mu \\
& =\int_Y\int_0^{\varphi(F^jy)}e^{-s(\varphi(F^jy)-u)}v(F^j\pi y,u)\,du\,d\mu
\\ & \qquad \qquad  -\int_Y\int_0^{\varphi(F^{j-1}y)}e^{-s(\varphi(F^{j-1}y)-u)}v(F^{j-1}\pi y,u)\,du\,d\mu,
\end{aligned}
\]
while $\int_Y\hV_0(s)\,d\mu =\int_Y\int_0^{\varphi(y)}e^{-s(\varphi(y)-u)}v(\pi y,u)\,du\,d\mu$.
Hence
\begin{align*}
\sum_{j=0}^J\int_Y\hV_j(s)\,d\mu & =\int_Y\int_0^{\varphi(F^Jy)}e^{-s(\varphi(F^Jy)-u)}v(F^J\pi y,u)\,du\,d\mu
\\ & = Z_J(s)
+\int_Y\int_0^{\varphi(F^Jy)}e^{-s(\varphi(F^Jy)-u)}v(F^Jy,u)\,du\,d\mu
 = Z_J(s)+I_0(s),
\end{align*}
where
\[
Z_J(s)=
\int_Y\int_0^{\varphi(F^Jy)}e^{-s(\varphi(F^Jy)-u)}(v(F^J\pi y,u)-v(F^Jy,u))\,du\,d\mu.
\]
By~\eqref{eq:WsF},
\[
|v(F^J\pi y,u)-v(F^Jy,u)|\le |v|_\gamma\,\varphi(F^Jy)d(F^J\pi y,F^Jy)
\le C_2\gamma^J|v|_\gamma\,\varphi(F^Jy).
\]
Also, $|v(F^J\pi y,u)-v(F^Jy,u)|\le 2|v|_\infty$, so
\[
|v(F^J\pi y,u)-v(F^Jy,u)| \le 2C_2\gamma_1^J\|v\|_\gamma\,\varphi(F^Jy)^\eta.
\]
Hence $|Z_J(s)|\le 2C_2 \gamma_1^{J}\|v\|_\gamma\int_Y (\varphi\circ F^J)^{1+\eta} \,d\mu=
2C_2 \gamma_1^{J}\|v\|_\gamma\int_Y \varphi^{1+\eta} \,d\mu\to0$ as $J\to\infty$.
\end{proof}

From now on, we specialize to the rapid mixing case, so $q$ and $\beta$ are arbitrarily large and all functions previously regarded as $C^q$ are now $C^\infty$.
Note that
\begin{equation} \label{eq:min}
\min\{\gamma_1^j,\gamma_1^{s(y,y')-j}\}\le
\gamma_1^{\frac13 j}\gamma_1^{\frac13 s(y,y')}
\le \gamma_1^{\frac13 j}\theta^{s(y,y')}.
\end{equation}

\begin{prop} \label{prop:Vj}
For each  $r\in\N$  there exists $C>0$ such that
\[
\|R^{j+1}\bV_j^{(r)}(s)\|_\theta\le C(|s|+1) \gamma_1^{j/3} \|v\|_\gamma
\quad\text{
for all $v\in \cH_\gamma(Y^\varphi)$, $s\in\barH$, $j\ge0$.}
\]
\end{prop}

\begin{proof}
For $j\ge0$,
\[
\hV_j(s)(y)=e^{-s\varphi(F^jy)}\Delta_j v_s(y)
=\int_0^{\varphi(F^jy)}e^{-s(\varphi(F^jy)-u)}\widetilde{\Delta}_jv(y,u)\,du.
\]
Hence
\[
\hV_j^{(r)}(s)(y)
=(-1)^r\int_0^{\varphi(F^jy)}e^{-s(\varphi(F^jy)-u)}(\varphi(F^jy)-u)^r\widetilde{\Delta}_jv(y,u)\,du.
\]
By Proposition~\ref{prop:DeltaE}(a),
$|\widetilde{\Delta}_jv(y,u)|\le 2C_2\gamma_1^{j-1}\|v\|_\gamma \varphi(F^jy)^\eta$.  Hence
$|\hV_j^{(r)}(s)|\le 2C_2 \gamma_1^{j-1} \|v\|_\gamma\, \varphi^{r+2}\circ F^j$.

Fix a $(j+1)$-cylinder $d$ for the Gibbs-Markov map $\bF:\bY\to\bY$.
Since $\bF^jd$ is a partition element,
\begin{equation} \label{eq:sup}
|1_d\bV_j^{(r)}(s)|_\infty\le C_2
\gamma_1^{j-1} \|v\|_\gamma\, |1_{\bF^jd}\,\varphi|_\infty^{r+2}
\le (2C_1)^{r+2}C_2 \gamma_1^{j-1} \|v\|_\gamma\, \infbFjd\,\varphi^{r+2}.
\end{equation}

Let $y,y'\in d$ with $\varphi(\bF^jy)\ge \varphi(\bF^jy')$.
Then
\[
\bV_j^{(r)}(s)(y)-\bV_j^{(r)}(s)(y')=(-1)^r(I_1+I_2+I_3+I_4),
\]
where
\begin{align*}
I_1 & = \int_{\varphi(F^jy')}^{\varphi(F^jy)}e^{-s(\varphi(F^jy)-u)}(\varphi(F^jy)-u)^r\widetilde{\Delta}_jv(y,u)\,du, \\
I_2 & = \int_0^{\varphi(F^jy')}\{e^{-s(\varphi(F^jy)-u)}-e^{-s(\varphi(F^jy')-u)}\}(\varphi(F^jy)-u)^r\widetilde{\Delta}_jv(y,u)\,du, \\
I_3 & = \int_0^{\varphi(F^jy')}e^{-s(\varphi(F^jy')-u)}\{(\varphi(F^jy)-u)^r-(\varphi(F^jy')-u)^r\}\widetilde{\Delta}_jv(y,u)\,du, \\
I_4 & = \int_0^{\varphi(F^jy')}e^{-s(\varphi(F^jy')-u)}(\varphi(F^jy')-u)^r\{\widetilde{\Delta}_jv(y,u)-\widetilde{\Delta}_jv(y',u)\}\,du.
\end{align*}
 By~\eqref{eq:inf},
\[
|\varphi(F^jy)-\varphi(F^jy')|\le
C_1\infbFjd\,\varphi\, \gamma^{s(F^jy,F^jy')}=
C_1\gamma^{s(y,y')-j}\infbFjd\,\varphi.
\]
Hence by Proposition~\ref{prop:DeltaE}(a,b),
\begin{align} \nonumber
& |\bV_j^{(r)}(s)(y)-\bV_j^{(r)}(s)(y')|\ll
(|s|+1)\gamma_1^{s(y,y')-j}\|v\|_\gamma \infbFjd\,\varphi^{r+3}.
\intertext{At the same time, the supnorm estimate~\eqref{eq:sup} yields}
\nonumber
& |\bV_j^{(r)}(s)(y)-\bV_j^{(r)}(s)(y')|\ll
\gamma_1^j \|v\|_\gamma\,  \infbFjd\,\varphi^{r+3}.
\intertext{Combining these estimates and using~\eqref{eq:min}
we obtain that}
\nonumber & |\bV_j^{(r)}(s)(y)-\bV_j^{(r)}(s)(y')|
  \ll
(|s|+1)\gamma_1^{j/3}\theta^{s(y,y')}\|v\|_\gamma\, \infbFjd\,\varphi^{r+3}.
\end{align}
In other words,
\[
|1_d\bV_j^{(r)}(s)|_\theta\ll (|s|+1)
\gamma_1^{j/3}\|v\|_\gamma\, \infbFjd\,\varphi^{r+3}.
\]

Using this and~\eqref{eq:sup}, it follows by Proposition~\ref{prop:GM} that
\begin{align*}
\|R^{j+1}\bV_j^{(r)}(s)\|_\theta
& \ll (|s|+1)\gamma_1^{j/3} \|v\|_\gamma
\sumd\,\bmu(d) \infd\,\varphi^{r+3}\circ \bF^j
 \\ & \le (|s|+1)\gamma_1^{j/3} \|v\|_\gamma
{\SMALL\int_\bY}\varphi^{r+3}\circ \bF^j\,d\bmu
 = (|s|+1)\gamma_1^{j/3} \|v\|_\gamma
{\SMALL\int_Y}\varphi^{r+3}\,d\mu,
\end{align*}
completing the proof.
\end{proof}

Define $D_{j,\ell}=\hR^\ell\hT R^{j+1}\bV\!_j$, $j,\ell\ge0$.
Let $\delta$ and $\lambda$ be as in Proposition~\ref{prop:lambda}, and
recall that $\H_\delta=\barH\cap B_\delta(0)$.

\begin{prop} \label{prop:Dj}
For each $r\in\N$, there exists $\alpha,\,C>0$ such that
for all $v\in\cH_\gamma(Y^\varphi)$, $j,\ell\ge0$, and all $s=a+ib\in\C$ with $a\in[0,1]$,
\begin{itemize}
\item[(a)] $|D_{j,\ell}^{(r)}(s)|_\infty\le C(\ell+1)^r\gamma_1^{j/3}(|b|+1)^{\alpha}\|v\|_\gamma$ for $s\not\in \H_\delta$,
\item[(b)]
$|\frac{d^r}{ds^r}\{D_{j,\ell}(s)-(1-\lambda(s))^{-1}\int_Y\hV_j(s)\,d\mu\}|_\infty\le C(\ell+1)^{r+1}\gamma_1^{j/3}\|v\|_\gamma$ for $s\in \H_\delta$.
\end{itemize}
\end{prop}

\begin{proof}
Let $p\in\N$, $p\le r$.
By Propositions~\ref{prop:Rn} and~\ref{prop:Vj},
$\|R^{j+1}\bV\!_j^{(p)}(s)\|_\theta\ll \gamma_1^{j/3}(|b|+1)\|v\|_\gamma$,
and $\|(\hR^\ell)^{(p)}(s)\|_\theta\ll (\ell+1)^r(|b|+1)$.

For $s\not\in\H_\delta$, it follows
from Proposition~\ref{prop:MT} that
$\|\hT^{(p)}(s)\|_\theta\ll (|b|+1)^\alpha$ for some $\alpha>0$.
Combining these estimates,
\begin{align*}
|(\hR^\ell\hT R^{j+1}\bV\!_j)^{(r)}(s)|_\infty
& \le \|(\hR^\ell\hT R^{j+1}\bV\!_j)^{(r)}(s)\|_\theta
 \ll (\ell+1)^r\gamma_1^{j/3}(|b|+1)^{\alpha+2} \|v\|_\gamma,
\end{align*}
completing the proof of~(a).

Next, suppose that $s\in \H_\delta$.  By Proposition~\ref{prop:lambda},
$\hR=\lambda P+\hR Q$ where
$P(s)$ is the spectral projection corresponding to $\lambda(s)$
and $Q(s)=I-P(s)$.  By Proposition~\ref{prop:lambda},
$\lambda(s)$ is a $C^\infty$ family of isolated eigenvalues with $\lambda(0)=1$, $\lambda'(0)\neq0$ and $|\lambda(s)|\le1$, and
$P(s)$ is a $C^\infty$ family of
operators on $\cF_\theta(\bY)$ with $P(0)v=\int_\bY v\,d\bmu$.
Also
\[
\hT=(1-\lambda)^{-1}P+Q_1
\quad\text{on $\H_\delta\setminus\{0\}$},
\]
where $Q_1=\hT Q$ is $C^\infty$ on $\H_\delta$.
Hence
\begin{align*}
\hR^\ell\hT & =(1-\lambda)^{-1}\lambda^\ell P+\hR^\ell Q_1
 =
(1-\lambda)^{-1}\lambda^\ell P(0)+\lambda^\ell Q_2+\hR^\ell Q_1
\quad\text{on $\H_\delta\setminus\{0\}$,}
\end{align*}
where $Q_2=(1-\lambda)^{-1}(P-P(0))$ is $C^\infty$ on $\H_\delta$.
Also, $(1-\lambda)^{-1}\lambda^\ell = (1-\lambda)^{-1}-(\lambda^{\ell-1}+\dots+1)$, so
\[
D_{j,\ell}-(1-\lambda)^{-1}P(0)R^{j+1}\bV\!_j =Q_{j,\ell}
\quad\text{on $\H_\delta$,}
\]
where
\[
 Q_{j,\ell}=
\big(-(\lambda^{\ell-1}+\dots+1)P(0)+\lambda^\ell Q_2+\hR^\ell Q_1\big)R^{j+1}\bV\!_j.
\]
It follows from the estimates for $R^{j+1}\bV\!_j$ and $\hR^\ell$ that
 $|(\hR^\ell Q_1R^{j+1}\bV\!_j)^{(r)}(s)|_\infty \ll (\ell+1)^r\gamma_1^{j/3}\|v\|_\gamma$
for $s\in \H_\delta$.  Since $|\lambda(s)|\le1$, the proof of Proposition~\ref{prop:Rn} applies equally to $\lambda^\ell$,  so
 $|Q_{j,\ell}^{(r)}(s)|_\infty \ll (\ell+1)^{r+1}\gamma_1^{j/3}\|v\|_\gamma$
for $s\in \H_\delta$.

Finally $P(0)R^{j+1}\bV\!_j=
\int_\bY \bV\!_j\,d\bmu
=\int_Y \hV_j\,d\mu$ completing the proof of part~(b).
\end{proof}

By Lemma~\ref{lem:ABC}, $\hC=\sum_{j,k=0}^\infty\hC_{j,k}$
is analytic on $\H$.  As shown in the next result,
$\hC$ extends smoothly to $\barH$.

\begin{cor} \label{cor:Cjk}
Assume absence of eigenfunctions, and let $r\in\N$.
There exists $\alpha,\,C>0$ such that
\[
|\hC^{(r)}(s)|\le C (|b|+1)^\alpha\|v\|_\gamma\|w\|_\gamma,
\]
for all $s=a+ib\in\barH$ with $a\in[0,1]$, and all
$v,w\in\cH_\gamma(Y^\varphi)$ with $\int_{Y^\varphi}v\,d\mu^\varphi=0$.
\end{cor}

\begin{proof}
Let $\ell=\max\{j-k-1,0\}$.
Recall from Lemma~\ref{lem:ABC}
that $\hC_{j,k}=\int_\bY D_{j,\ell}\bW\!_k\,d\bmu$.
Let $p\in\N$, $p\le r$.
By Proposition~\ref{prop:Wk}, $|\bW\!_k(t)|_1\ll (k+1)^{p+3}\gamma_1^k\|w\|_\gamma\, (t+1)^{-(p+2)}$, so
$|\bW\!_k^{(p)}(s)|_1\ll (k+1)^{r+3}\gamma_1^k\|w\|_\gamma$.
Combining this with Proposition~\ref{prop:Dj}(a),
\[
|\hC_{j,k}^{(r)}(s)|\ll
|b|^{\alpha} (j+1)^r\gamma_1^{j/3}(k+1)^{r+3}\gamma_1^k\|v\|_\gamma\|w\|_\gamma\quad\text{for $|b|\ge \delta$,}
\]
and the proof for $|b|\ge\delta$ is complete.

For $|b|\le\delta$, we use Proposition~\ref{prop:sumV} to write
\[
\SMALL \hC=\sum_{j,k}\int_Y \big\{D_{j,\ell}-(1-\lambda)^{-1}\int_Y\hV_j\,d\mu\big\}\bW\!_k\,d\mu+(1-\lambda)^{-1}I_0\sum_k\int_Y\bW\!_k\,d\mu.
\]
Proposition~\ref{prop:Dj}(b) takes care of the first term on the right-hand side, and it remains to estimate
$g=(1-\lambda)^{-1}I_0$.
Now
\begin{equation} \label{eq:I0}
I_0(0)=\int_Y\int_0^{\varphi(y)}v(y,u)\,du\,d\mu=\intphi\int_{Y^\varphi}v\,d\mu^\varphi=0,
\end{equation}
 so it follows from
Proposition~\ref{prop:lambda} that  $g$ is $C^\infty$ with
$|g^{(r)}(s)|\ll |v|_\infty$ on $\H_\delta$.
\end{proof}

\begin{pfof}{Theorem~\ref{thm:rapidflow}}
Recall that $\beta$ and $q$ can be taken arbitrarily large.
Hence it follows from Proposition~\ref{prop:poll} that
$\sup_\barH|\hJ_0^{(r)}|\ll |v|_\infty|w|_\infty$ for all $r\in\N$.
Similarly, by Propositions~\ref{prop:An} and~\ref{prop:Bnk},
$\sup_\barH|\hA_n^{(r)}|\ll n^{r+3}\gamma_1^n|v|_\infty|w|_\gamma$ and
$\sup_\barH|\hB_{n,k}^{(r)}|\ll n^{r+3}\gamma_1^n|v|_\gamma|w|_\infty$.
Combining these with Corollary~\ref{cor:Cjk}
and substituting into Lemma~\ref{lem:ABC},
we have shown that $\hat\rho_{v,w}:\H\to\C$ extends to
$\hat\rho_{v,w}:\barH\to\C$.
Moreover, we have shown that
for every $r\in\N$ there exists $C,\,\alpha>0$ such that
\[
|\hat\rho_{v,w}^{(r)}(s)|\le C(|b|+1)^\alpha\|v\|_{\gamma}\|w\|_{\gamma}
\quad\text{for $s=a+ib\in\C$ with $a\in[0,1]$,}
\]
for all $v,w\in\cH_\gamma(Y^\varphi)$ with $\int_{Y^\varphi}v\,d\mu^\varphi=0$.
The result now follows from Lemma~\ref{lem:strat}
and Remark~\ref{rmk:strat}.
\end{pfof}

\section{Polynomial mixing for skew product Gibbs-Markov flows}
\label{sec:poly}

In this section, we consider skew product Gibbs-Markov flows $F_t:Y^\varphi\to Y^\varphi$ for which
the roof function $\varphi:Y\to\R^+$
satisfies $\mu(\varphi>t)=O(t^{-\beta})$ for some $\beta>1$.
For such flows, we prove Theorem~\ref{thm:polyflow}, namely that absence of approximate eigenfunctions is a sufficient condition to obtain the mixing rate
$O(t^{-(\beta-1)})$.

If $f:\R\to\R$ is integrable, we write $f\in\cR(a(t))$ if the inverse Fourier transform of $f$ is $O(a(t))$.
We also write $\cR(t^{-p})$ instead of $\cR((t+1)^{-p})$ for $p>0$.

\begin{prop}[ {\cite[Proposition~8.2]{M18}} ] \label{prop:fourier}
Let $g:\R\to\R$ be an integrable function such that $g(b)\to0$ as $b\to\pm\infty$.  If $|f^{(q)}|\le g$, then $f\in\cR(|g|_1\,t^{-q})$. \qed
\end{prop}

The convolution $f\star g$ of two integrable functions $f,g:[0,\infty)\to\R$ is defined to be $(f\star g)(t)=\int_0^t f(x)g(t-x)\,dx$.

\begin{prop}[ {\cite[Proposition~8.4]{M18}} ] \label{prop:conv}
Fix $b>a>0$ with $b>1$.
Suppose that $f,g:[0,\infty)\to\R$ are integrable and there exist constants
$C,D>0$ such that $|f(t)|\le C(t+1)^{-a}$ and $|g(t)|\le D(t+1)^{-b}$ for $t\ge0$.
Then there exists a constant $K>0$ depending only on $a$ and $b$ such that
$|(f\star g)(t)|\le CDK(t+1)^{-a}$ for $t\ge0$.  \qed
\end{prop}

\begin{prop} \label{prop:fb}
Define $f(b)=b^{-1}(e^{-ib\varphi}-1)$ for $b\in\R\setminus\{0\}$.
Then there exists $C>0$ such that
$\|1_{\bY\!_k}f^{(q)}(b)\|_\theta \le C \infYk \varphi^{q+\eta} |b|^{-(1-\eta)}$
for all $b\in\R\setminus\{0\}$.
\end{prop}

\begin{proof}
This is contained in the proof of~\cite[Proposition~8.13]{M18}.
\end{proof}

\subsection{Modified estimate for $R^{j+1}\bV\!_j$}

\begin{prop} \label{prop:VV0}
There exists $C>0$ such that
\[
\SMALL \|R^{j+1}\bV\!_j^{(q)}(ib)\|_{\theta}\le C\gamma_1^{j/3}\|v\|_\gamma\,|b|^{-(1-\eta)},
\]
for all $v\in \cH_\gamma(Y^\varphi)$ such that $v$ is independent of $u$,
and all $b\neq0$, $j\ge0$.
\end{prop}

\begin{proof}
Recall that
\[
\hV_j(s)=e^{-s\varphi\circ F^j}\Delta_j v_s
=\int_0^{\varphi\circ F^j}e^{-s(\varphi\circ F^j-u)}\,du\;\Delta_jv
=\int_0^{\varphi\circ F^j}e^{-su}\,du\;\Delta_jv.
\]
Hence
$R^j\bV\!_j(s)=
\int_0^{\varphi}e^{-su}\,du\,R^j(\Delta_jv)
=-s^{-1}(e^{-s\varphi}-1)R^j(\Delta_jv)$.
It follows that
\begin{equation} \label{eq:fb}
R^{j+1}\bV\!_j(ib)=iR\big(f(b)R^j(\Delta_jv)\big),
\end{equation}
where $f(b)=b^{-1}(e^{-ib\varphi}-1)$.

Let $d\in\bY$ be a $j$-cylinder and let $y,y'\in d$.
Then the arguments in the proof of Proposition~\ref{prop:DeltaE}(a,b) show that
\[
|\Delta_jv(y)|\ll \gamma_1^j\|v\|_\gamma\, \varphi(F^jy)^\eta, \qquad
|\Delta_jv(y)-\Delta_jv(y')|\ll \gamma_1^{s(y,y')-j}\|v\|_\gamma\, \varphi(F^jy)^\eta.
\]
On the other hand,
$|\Delta_jv(y)-\Delta_jv(y')|\ll \gamma_1^j\|v\|_\gamma\, \varphi(F^jy)^\eta$, so by~\eqref{eq:min},
\[
|\Delta_jv(y)-\Delta_jv(y')|\ll \gamma_1^{j/3}\theta^{s(y,y')}\|v\|_\gamma\,
\varphi(F^jy)^\eta.
\]

Using~\eqref{eq:inf}, it follows that
\[
|1_d(1_{\bY\!_k}\circ \bF^j)\Delta_jv|_\infty\ll
\gamma_1^j\|v\|_\gamma\, \supYk\varphi^\eta\le 2C_1
\gamma_1^j\|v\|_\gamma\, \infYk\varphi^\eta,
\]
and similarly,
\[
|1_d(1_{\bY\!_k}\circ \bF^j)\Delta_jv|_\theta\ll
\gamma_1^{j/3}\|v\|_\gamma\, \infYk\varphi^\eta,
\qquad
\|1_d(1_{\bY\!_k}\circ \bF^j)\Delta_jv\|_\theta\ll
\gamma_1^{j/3}\|v\|_\gamma\, \infYk\varphi^\eta.
\]
By Proposition~\ref{prop:GM},
\[
\|1_{\bY\!_k}R^j(\Delta_jv)\|_\theta=
\|R^j\big((1_{\bY\!_k}\circ \bF^j)\Delta_jv\big)\|_\theta\ll
\gamma_1^{j/3}\|v\|_\gamma\, \infYk\varphi^\eta.
\]
Hence by Proposition~\ref{prop:fb},
\begin{align*}
\|1_{\bY\!_k}f^{(q)}(b)R^j(\Delta_jv)\|_\theta & \ll
\infYk \varphi^{q+\eta} |b|^{-(1-\eta)}\|1_{\bY\!_k}R^j(\Delta_jv)\|_\theta
\\ & \ll \gamma_1^{j/3}\|v\|_\gamma\, \infYk \varphi^{q+2\eta} |b|^{-(1-\eta)}.
\end{align*}
Applying Proposition~\ref{prop:GM} once more and using~\eqref{eq:fb},
\begin{align*}
\|R^{j+1}\bV\!_j^{(q)}(ib)\|_\theta  = \SMALL \|R\big(f^{(q)}(b)R^j(\Delta_jv)\big)\|_\theta
& \ll \sum_k \bmu(\bY\!_k)\|1_{\bY\!_k}f^{(q)}(b)R^j(\Delta_jv)\|_\theta
\\ & \SMALL \ll \gamma_1^{j/3}\|v\|_\gamma
\int_Y \varphi^{q+2\eta}\,d\mu\, |b|^{-(1-\eta)}.
\end{align*}
as required.
\end{proof}

Let $\bV\!_j(t):\bY\to\R$ denote the inverse Laplace transform associated to $\hV_j(s):Y\to\C$.
\begin{prop} \label{prop:VV1}
There is a constant $C$ such that
\[
\SMALL \|R^{j+1}\bV\!_j(t)\|_\theta\le C\gamma_1^{j/3}\|v\|_{\gamma,\eta}\,(t+1)^{-q},
\]
for all $v\in \cH_{\gamma,\eta}(Y^\varphi)$ with $v(y,0)\equiv0$ and all $j\ge0$, $t>0$.
\end{prop}

\begin{proof}
For $j\ge0$,
\[
\hV_j(s)(y)
=\int_0^{\varphi(F^jy)}e^{-s(\varphi(F^jy)-u)}\widetilde{\Delta}_jv(y,u)\,du
=\int_0^{\varphi(F^jy)}e^{-st}\widetilde{\Delta}_jv(y,\varphi(F^jy)-t)\,dt,
\]
so
\[
V_j(t)(y)= 1_{\{\varphi(F^jy)>t\}}
\widetilde{\Delta}_jv(y,\varphi(F^jy)-t).
\]

Recall that $c'=1/(2C_1)$.
Fix a $(j+1)$-cylinder $d$.  By Proposition~\ref{prop:DeltaE}(a),
for $y\in d$,
\begin{align} \label{eq:Vjsup} \nonumber
|V_j(t)(y)| & \le 2C_2\gamma_1^{j-1}\|v\|_\gamma\, \varphi(F^jy)^\eta 1_{\{|1_{\bF^jd}\,\varphi|_\infty>t\}}
\\ & \le 4C_1C_2 \gamma_1^{j-1}\|v\|_\gamma\, \infbFjd\,\varphi^\eta 1_{\{\infbFjd\,\varphi>c't\}}.
\end{align}

For $y,y'\in d$,
\[
|\varphi(F^jy)-\varphi(F^jy')| \le C_1\infbFjd\,\varphi\, \gamma^{s(y,y')-j}.
\]
so by Propositions~\ref{prop:DeltaE}(b,c),
for $t\in[0,\varphi(F^jy)]\cap[0,\varphi(F^jy')]$,
\begin{align} \label{eq:K2} \nonumber
|\widetilde{\Delta}_jv(y,\varphi(F^jy)-t)- & \widetilde{\Delta}_jv(y',\varphi(F^jy')-t)|
 \\ \nonumber & \le 4C_2 \gamma_1^{s(y,y')-j}\|v\|_\gamma\, \infbFjd\,\varphi^\eta
+2|v|_{\infty,\eta}|\varphi(F^jy)-\varphi(F^jy')|^\eta
\\ & \ll \gamma_1^{s(y,y')-j}\|v\|_{\gamma,\eta}\, \infbFjd\,\varphi^\eta.
\end{align}
Similarly, for $t\in[\varphi(F^jy'),\varphi(F^jy)]$,
\begin{align} \label{eq:K3} \nonumber
|\widetilde{\Delta}_j & v(y,\varphi(F^jy)-t)|  =
|\widetilde{\Delta}_jv(y,\varphi(F^jy)-t)- \widetilde{\Delta}_jv(y,0)|
\le 2|v|_{\infty,\eta}|\varphi(F^jy)-t|^\eta
 \\ &  \qquad
\le 2|v|_{\infty,\eta}|\varphi(F^jy)-\varphi(F^jy')|^\eta
\ll \gamma_1^{s(y,y')-j}|v|_{\infty,\eta}\,\infbFjd\,\varphi^\eta.
\end{align}

For $y,y'\in d$ with $\varphi(F^jy)\ge\varphi(F^jy')$,
\begin{align*}
& V_j(t)(y)-V_j(t)(y')  =\begin{cases}
\widetilde{\Delta}_jv(y,\varphi(F^jy)-t)
- \widetilde{\Delta}_jv(y',\varphi(F^jy')-t), \quad  \varphi(F^jy')>t \\
\widetilde{\Delta}_jv(y,\varphi(F^jy)-t), \qquad\qquad\qquad\qquad  \varphi(F^jy)>t\ge \varphi(F^jy') \\
0,
\qquad\qquad \qquad\qquad \qquad\qquad \qquad\qquad
\varphi(F^jy) \le t
\end{cases}
\end{align*}
If $\varphi(F^jy')>t$, then using~\eqref{eq:K2},
\begin{align*}
 |V_j(t)(y)-V_j(t)(y')|  & \ll
\gamma_1^{s(y,y')-j}\|v\|_{\gamma,\eta}1_{\{|1_{\bF^jd}\,\varphi|_\infty>t\}}
\infbFjd\,\varphi^\eta
\\ & \le \gamma_1^{s(y,y')-j}\|v\|_{\gamma,\eta}1_{\{\infbFjd\,\varphi>c't\}}\infbFjd\,\varphi^\eta
.
\end{align*}
If $\varphi(F^jy)>t\ge \varphi(F^jy')$, then using~\eqref{eq:K3},
\[
 |V_j(t)(y)  -V_j(t)(y')|
 \ll \gamma_1^{s(y,y')-j}|v|_{\infty,\eta}1_{\{\infbFjd\,\varphi>c't\}}\infbFjd\varphi^\eta.
\]
Hence in all cases,
\[
 |V_j(t)(y) -V_j(t)(y')|
 \ll \gamma_1^{s(y,y')-j}\|v\|_{\gamma,\eta}1_{\{\infbFjd\,\varphi>c't\}}\infbFjd\,\varphi^\eta.
\]
On the other hand, by~\eqref{eq:Vjsup},
 $|V_j(t)(y) -V_j(t)(y')|
\ll \gamma_1^{j}\|v\|_\gamma\, \infbFjd\,\varphi^\eta 1_{\{\infbFjd\,\varphi>c't\}}$.
Combining these estimates and using~\eqref{eq:min},
\[
 |V_j(t)(y) -V_j(t)(y')|
 \ll \gamma_1^{j/3}\theta^{s(y,y')}\|v\|_{\gamma,\eta}1_{\{\infbFjd\,\varphi>c't\}}\infbFjd\,\varphi^\eta.
\]
Hence
\[
\|1_d V_j(t)\|_\theta\ll \gamma_1^{j/3}\|v\|_{\gamma,\eta}1_{\{\infbFjd\,\varphi>c't\}}\infbFjd\,\varphi^\eta.
\]
By Proposition~\ref{prop:GM},
\begin{align*}
\|R^{j+1}\bV\!_j(t)\|_\theta
& \ll \gamma_1^{j/3}\|v\|_{\gamma,\eta}\sumd\,\bmu(d)1_{\{\infd\,\varphi\circ \bF^j>c't\}}(\infd\,\varphi\circ \bF^j)^\eta
\\ &
\le\gamma_1^{j/3} \|v\|_{\gamma,\eta} \int_\bY 1_{\{\varphi\circ \bF^j>c't\}}(\varphi\circ \bF^j)^\eta\,d\mu=
\gamma_1^{j/3} \|v\|_{\gamma,\eta} \int_Y 1_{\{\varphi>c't\}}\varphi^\eta\,d\mu.
\end{align*}
Now apply Proposition~\ref{prop:varphi}(b).
\end{proof}

\begin{cor} \label{cor:VV}
Let $\kappa:\R\to\R$ be $\C^\infty$ with $|\kappa^{(k)}(b)|=O((b^2+1)^{-1})$ for all $k\in\N$.
Then
$\|\kappa R^{j+1}\bV\!_j\|_{\theta}\in\cR(\gamma_1^{j/3}\|v\|_{\gamma,\eta}\, t^{-q})$
for all $v\in \cH_{\gamma,\eta}(Y^\varphi)$, $j\ge0$.
\end{cor}

\begin{proof}
Write $v(y,u)=v_0(y)+v_1(y,u)$ where $v_0(y)=v(y,0)$. We have the corresponding decomposition $\bV\!_j=\bV\!_{j,0}+\bV\!_{j,1}$.
The function $g(b)=\kappa(b)|b|^{-(1-\eta)}$ is integrable
and
$\|(\kappa R^{j+1}\bV\!_{j,0})^{(q)}(ib)\|_\theta\ll\gamma_1^{j/3}\|v\|_{\gamma,\eta}\,g(b)$ by Proposition~\ref{prop:VV0}, so
$\|\kappa R^{j+1}\bV\!_{j,0}\|_\theta\in\cR(\gamma_1^{j/3}\|v\|_\gamma\, t^{-q})$
by Proposition~\ref{prop:fourier}.
Also, $\kappa\in\cR(t^{-q})$ by Proposition~\ref{prop:fourier}, so
$\|\kappa R^{j+1}\bV\!_{j,1}\|_\theta\in\cR(\gamma_1^{j/3}\|v\|_{\gamma,\eta}\, t^{-q})$
by Propositions~\ref{prop:conv} and~\ref{prop:VV1}.
\end{proof}

\subsection{Truncation}
\label{sec:trunc}

We proceed in a manner analogous to~\cite[Section~8.4]{M18},
replacing $\varphi$ by a bounded roof function.
Given $N\ge1$, let $Y(N)=\bigcup_{j\ge1:\inf_{Y_j}\varphi\ge N}Y_j$.
Define $\varphi(N)=N$ on $Y(N)$ and $\varphi(N)=\varphi$ elsewhere.
(Unlike~\cite{M18}, it is not sufficient to take $\varphi(N)=\min\{\varphi,N\}$.)
Note that $\varphi(N)\le 2C_1N$ by~\eqref{eq:inf}.

Consider the suspension semiflows $F_t$ and $F_{N,t}$ on $Y^\varphi$ and $Y^{\varphi(N)}$ respectively.
(Here, $F_{N,t}$ is computed modulo the identification $(y,\varphi(N)(y))\sim(Fy,0)$ on $Y^{\varphi(N)}$.)
Let $\rho_{v,w}$ and $\rho^{\rm trunc}_{v,w}$ denote the respective correlation functions.
In particular,
$\rho^{\rm trunc}_{v,w}(t)=
\int_{Y^{\varphi(N)}}v\,w\circ F_{N,t}\,d\mu^{\varphi(N)}-
\int_{Y^{\varphi(N)}}v\,\,d\mu^{\varphi(N)}
\int_{Y^{\varphi(N)}}w\,\,d\mu^{\varphi(N)}$
where the observables
$v,w:Y^{\varphi(N)}\to\R$ are the restrictions of $v,w:Y^\varphi\to\R$ to
$Y^{\varphi(N)}$.

\begin{prop}[ {\cite[Proposition~8.19]{M18}} ]
\label{prop:trunc}
There are constants $C,\,t_0>0$, $N_0\ge1$ such that
\[
|\rho_{v,w}(t)-\rho^{\rm trunc}_{v,w}(t)|\le C|v|_\infty|w|_\infty
(tN^{-\beta}+N^{-(\beta-1)}),
\]
for all
$v,w\in L^\infty(Y^\varphi)$, $N\ge N_0$, $t> t_0$.
\qed
\end{prop}

We make the following abuse of notation regarding norms of observables
$v:Y^{\varphi(N)}\to\R$. Define
$\|v\|_{\gamma,\eta}= \|v'\|_{\gamma,\eta}$ where
$v'$ is the extension of $v$ by zero to $Y^\varphi$.
(In other words, the factor of $\varphi$ on the denominator in the definition of $|v|_\gamma$ is {\em not} replaced by $\varphi(N)$.)

With this convention, $v\in \cH_{\gamma,\eta}(Y^\varphi)$ restricts
to $v|_{Y^{\varphi(N)}}\in \cH_{\gamma,\eta}(Y^{\varphi(N)})$ with
$\|v|_{Y^{\varphi(N)}}\|_{\gamma,\eta}\le
\|v\|_{\gamma,\eta}$.
The similar convention applies to observables $w\in\cH_\gamma(Y^{\varphi(N)})$.
However, restricting $w\in\cH_{\gamma,0,m}(Y^\varphi)$ to
$Y^{\varphi(N)}$ need not preserve smoothness in the flow direction.
Below we prove:

\begin{lemma}  \label{lem:trunc}
Assume absence of approximate eigenfunctions.
In particular, there is
a finite union $Z\subset \bY$ of partition elements such that the
corresponding finite subsystem $Z_0$ does not support approximate eigenfunctions.
Choose $N_1\ge |1_Z \varphi|_\infty+3$.

There exist $m\ge1$,  $C>0$ such that
\[
|\rho^{\rm trunc}_{v,w}(t)|\le C\|v\|_{\gamma,\eta}\|w\|_{\gamma,0,m}\, t^{-(\beta-1)},
\]
for all
$v\in \cH_{\gamma,\eta}(Y^{\varphi(N)})$,
$w\in \cH_{\gamma,0,m}(Y^{\varphi(N)})$, $N\ge N_1$, $t>1$.
\end{lemma}

\begin{pfof}{Theorem~\ref{thm:polyflow}}
Let $m\ge1$, $N_1\ge3$ be as in Lemma~\ref{lem:trunc}.
As discussed above, the observable $v:Y^\varphi\to\C$ restricts to an observable \mbox{$v:Y^{\varphi(N)}\to\C$} with no increase in the value of $\|v\|_{\gamma,\eta}$,
but restricting $w\in\cH_{\gamma,0,m}(Y^\varphi)$ to
$Y^{\varphi(N)}$ need not preserve smoothness in the flow direction.
To circumvent this, following~\cite{M09,M18} we define an approximating observable
$w_N:Y^{\varphi(N)}\to\R$, $N\ge N_1$,
\[
w_N(y,u)=\begin{cases} w(y,u) & (y,u)\not\in Y(N)\times[N-2,N] \\
\sum_{j=0}^{2m+1}(u-N+2)^j d_{N,j}(y) & (y,u)\in Y(N)\times[N-2,N-1] \\
w(y,u+\varphi(y)-N) & (y,u)\in Y(N)\times (N-1,N]
\end{cases},
\]
where the $d_{N,j}(y)$ are linear combinations of
$\partial_t^jw(y,N-2)$ and $\partial_t^jw(y,\varphi(y)-1)$, $j=0,\dots,m$, with coefficients independent of $y$ and $N$ uniquely specified by the requirements 
$\partial_t^jw_N(y,N-2)=\partial_t^jw(y,N-2)$
and $\partial_t^jw_N(y,N-1)=\partial_t^jw(y,\varphi(y)-1)$ for $j=0,\dots,m$.
\footnote{In fact $d_{N,j}(y)=(1/j!)\partial_t^jw(y,N-2)$ for $0\le j\le m$ but the remaining formulas are messier.
When $m=1$,
for instance,
$d_{N,2}(y)  =-3w(y,N-2)-2\partial_tw(y,N-2)+3w(y,\varphi(y)-1)-\partial_tw(y,\varphi(y)-1)$,
$d_{N,3}(y)  =2w(y,N-2)+\partial_tw(y,N-2)-2w(y,\varphi(y)-1)+\partial_tw(y,\varphi(y)-1)$.
}

It is immediate from the definitions that $w_N$ is $m$-times differentiable in the flow direction.
We claim that
$\|w_N\|_{\gamma,0,m}\le C' \|w\|_{\gamma,0,m+1}$ for some constant $C'$ independent of~$N$.
By Lemma~\ref{lem:trunc},
\[
|\rho^{\rm trunc}_{v,w_N}(t)| \le CC' \|v\|_{\gamma,\eta}\|w\|_{\gamma,0,m+1}\, t^{-(\beta-1)}.
\]
Also,
\begin{align*}
 |\rho^{\rm trunc}_{v,w} & (t)  -
\rho^{\rm trunc}_{v,w_N}(t)|
 \le |v|_\infty(|w|_\infty+|w_N|_\infty)\mu^{\varphi(N)}(F_{N,t}^{-1}S_N)
 \\ & =|v|_\infty(|w|_\infty+|w_N|_\infty)\mu^{\varphi(N)}(S_N)
\le 2|v|_\infty|w|_\infty\mu(\varphi>N)
\ll |v|_\infty|w|_\infty\, N^{-\beta},
\end{align*}
so
\[
|\rho^{\rm trunc}_{v,w}(t)| \ll  \|v\|_{\gamma,\eta}\|w\|_{\gamma,0,m+1}\, (t^{-(\beta-1)}+N^{-\beta}).
\]
Taking $N=[t]$, the result follows directly from Proposition~\ref{prop:trunc}.

It remains to verify the claim.
Fix $k\in\{0,\dots,m\}$.
Let  $(y,u),\,(y',u)\in Y(N)\times[N-2,N-1]$, where $y,y'$ lie in the same partition element.  Then
\begin{align*}
|\partial_t^k w_N(y,u)| & \SMALL \le (2m+1)! \sum_{j=0}^{2m+1}|d_{N,j}(y)|
\\ & \le \SMALL C\sum_{j=0}^m(|\partial_t^jw(y,N-2)|+|\partial_t^jw(y,\varphi(y)-1)|)
\le 2C\|w\|_{\gamma,0,m},
\end{align*}
where $C$ is a constant independent of $N$.  Also,
by~\eqref{eq:inf}, for $0\le j\le m$
\begin{align*}
|\partial_t^jw(y,\varphi(y)-1)-\partial_t^jw(y,\varphi(y')-1)|
\le |\partial_t^{j+1}w|_\infty |\varphi(y)-\varphi(y')|
\le C_1|\partial_t^{j+1}w|_\infty \varphi(y)\gamma^{s(y,y')}.
\end{align*}
Hence
\begin{align*}
|\partial_t^k  & w_N(y,u)|-\partial_t^k w_N(y',u)|
 \le \SMALL (2m+1)!\sum_{j=0}^{2m+1}|d_{N,j}(y)-d_{N,j}(y')|
\\ & \le \SMALL C\sum_{j=0}^m\big(|\partial_t^jw(y,N-2)-\partial_t^jw(y',N-2)|+|\partial_t^jw(y,\varphi(y)-1)-\partial_t^jw(y',\varphi(y')-1)|\big)
\\ & \SMALL \le 2C\sum_{j=0}^m |\partial_t^jw|_\gamma \varphi(y)\{d(y,y')+\gamma^{s(y,y')}\}
+C\sum_{j=0}^m|\partial_t^jw(y,\varphi(y)-1)-\partial_t^jw(y,\varphi(y')-1)|
\\ & \le 2C\|w|_{\gamma,0,m} \varphi(y)\{d(y,y')+\gamma^{s(y,y')}\}
+CC_1\|w\|_{\gamma,0,m+1} \varphi(y) \gamma^{s(y,y')}
\\ & \le 3CC_1\|w|_{\gamma,0,m+1} \varphi(y)\{d(y,y')+\gamma^{s(y,y')}\}.
\end{align*}
This completes the verification of the claim on the region
$Y(N)\times[N-2,N-1]$ and the other regions are easier to treat.
\end{pfof}

Our strategy for proving Lemma~\ref{lem:trunc} is identical to that for~\cite[Lemma~8.20]{M18}.
The first step is to show that the inverse Laplace transform of $\widehat{\rho^{\rm trunc}_{v,w}}$ can be computed using the imaginary axis as the contour of integration.

\begin{prop} \label{prop:move}
Let $N\ge N_1$,
$v,w\in\cH_{\gamma}(Y^{\varphi(N)})$.
Then
there exists $\eps>0$, $C>0$, $\alpha\ge0$, such that
$\widehat{\rho^{\rm trunc}_{v,w}}$ is continuous on $\{\Re s\in[0,\eps]\}$
and $|\widehat{\rho^{\rm trunc}_{v,w}}(s)|\le C(|b|+1)^\alpha$
for all $s=a+ib$ with $a\in[0,\eps]$.
\end{prop}

\begin{proof}
In this proof, the constant $C$ is not required to be uniform in $N$.
Consequently, the estimates are very straightforward compared to other estimates in this section.

The desired properties for $\widehat{\rho^{\rm trunc}_{v,w}}$ will hold provided they are verified for all the constituent parts
in Lemma~\ref{lem:ABC}.  Note that if $f$ is integrable on $[0,\infty)$, then $\hat f$ satisfies the required properties with $\alpha=0$.
Hence the estimate in Proposition~\ref{prop:Wk} already suffices for $\bW\!_k$.
Also, the proof of Proposition~\ref{prop:Vj} suffices after truncation
since $\varphi^{r+3}$ becomes $(2C_1N)^{r+2}\varphi$.
(Actually, the factor $\varphi^{r+3}$ is easily improved to $\varphi^{r+1+2\eta}$ which is integrable when $r=0$ so truncation is not absolutely necessary for the term $R^{j+1}\bV\!_j$.)

By definition of $N_1$, the truncated roof function $\varphi(N)$ coincides
with $\varphi$ on the subsystem $Z_0$, so absence of approximate eigenfunctions passes over to the truncated flow for each $N\ge N_1$.
Since $\varphi(N)\le 2C_1N$,
all estimates related to $\hR$ and $\hT$ in Section~\ref{sec:rapid}
now hold for $q$ arbitrarily large.
Hence the arguments in Section~\ref{sec:rapid} yield the desired properties for
$\sum_{0\le j,k<\infty}C_{j,k}$.
Also, it is immediate from the proof of~\cite[Proposition~6.3]{M18}
that $|J_0(t)|\ll N\mu(\varphi(N)>t)$ so $J_0(t)=0$ for $t>N$
and hence $J_0$ is integrable.

It remains to consider the terms $A_n$ and $B_n$.
Here, we must take into account that the
factor of $\varphi$ in the definition of $\|\;\|_\gamma$ is not truncated.
Starting from the end of the proof of Proposition~\ref{prop:An}, we obtain
\[
\SMALL |A_n(t)|\le 4C_1C_2N\gamma_1^{n-1}|v|_\infty|w|_\gamma \int_Y\varphi^\eta\circ F^n 1_{\{\varphi_{n+1}>t\}}\,d\mu.
\]
A simplified version of Proposition~\ref{prop:Rvarphi2}
combined with Proposition~\ref{prop:varphi}(b) yields
\begin{align*}
\int_Y\varphi^\eta\circ F^n 1_{\{\varphi_{n+1}>t\}}\,d\mu
& \le \sum_{j=0}^n\int_Y\varphi^\eta R^{n-j} 1_{\{\varphi>t/n\}}\,d\mu
\\ & \le \sum_{j=0}^{n-1}|\varphi|_1 |R1_{\{\varphi>t/n\}}|_\infty+
\int_Y\varphi^\eta 1_{\{\varphi>t/n\}}\,d\mu
\ll n^{\beta+1}t^{-(\beta-\eta)}.
\end{align*}
Hence
$|A_n(t)|\ll n^{\beta+1}\gamma_1^n |v|_\infty|w|_\gamma\, t^{-(\beta-\eta)}$.
Similarly
$|B_{n,k}(t)|\ll n^{\beta+1}\gamma_1^n |v|_\gamma|w|_\infty\, t^{-(\beta-\eta)}$.
Hence $\sum_{n\ge1}A_n$ and
$\sum_{0\le k<n<\infty}B_{n,k}$ are integrable, completing the proof.
\end{proof}

Choose $\psi:\R\to[0,1]$ to be $C^\infty$ and compactly supported such
that $\psi\equiv1$ on a neighbourhood of zero.
Let $\kappa_m(b)=(1-\psi(b))(ib)^{-m}$, $m\ge2$.

\begin{cor} \label{cor:move}
Let $N\ge N_1$, $m\ge\alpha+2$,
$v\in\cH_{\gamma}(Y^{\varphi(N)})$,
$w\in\cH_{\gamma,0,m}(Y^{\varphi(N)})$.
Then
\[
\rho^{\rm trunc}_{v,w}(t)=\int_{-\infty}^\infty \psi(b) e^{ibt}\widehat{\rho^{\rm trunc}_{v,w}}(ib)db
+\int_{-\infty}^\infty \kappa_m(b) e^{ibt}\widehat{\rho^{\rm trunc}_{v,\partial_t^mw}}(ib)db.
\]
\end{cor}

\begin{proof}
As in~\cite[Section~6.1]{M18}, we can suppose without loss that $\rho_{v,w}^{\rm trunc}$ vanishes for $t$ near zero, so that
\begin{equation} \label{eq:parts}
\widehat{\rho^{\rm trunc}_{v,w}}(s)=s^{-m}\widehat{\rho^{\rm trunc}_{v,\partial_t^mw}}(s)
\quad\text{for all $s\in\H$}.
\end{equation}
By Proposition~\ref{prop:move}, it follows as in the proof of~\cite[Lemma~6.2]{M18} that
\[
\SMALL \rho^{\rm trunc}_{v,w}(t)
=\int_{-\infty}^\infty e^{ibt}\widehat{\rho^{\rm trunc}_{v,w}}(ib)db
=\int_{-\infty}^\infty \psi(b) e^{ibt}\widehat{\rho^{\rm trunc}_{v,w}}(ib)db
+\int_{-\infty}^\infty (1-\psi(b)) e^{ibt}\widehat{\rho^{\rm trunc}_{v,w}}(ib)db.
\]
By Proposition~\ref{prop:move}, equation~\eqref{eq:parts} extends to
$\barH\setminus\{0\}$ and the result follows.
\end{proof}

From now on we suppress the superscript ``${\rm trunc}$'' for sake of readability.  Notation $\hR$, $\hT$ and so on refers to the operators obtained using
$\varphi(N)$ instead of $\varphi$.
We end this subsection by recalling some further estimates from~\cite{M18}.
The first is a uniform version of Proposition~\ref{prop:MT}.

\begin{prop}[ { \cite[Proposition~8.27]{M18} }]  \label{prop:Tunif}
Assume absence of approximate eigenfunctions.  Then
there exists $m\ge2$ such that
\[
\|\kappa_m(b)\hT(ib)\|_\theta\in\cR(t^{-q})
\quad\text{uniformly in $N\ge N_1$.}
\]

\vspace{-5ex}
\qed
\end{prop}

The remaining estimates in this subsection
are required when $b$ is close to zero.
By Proposition~\ref{prop:lambda},
for each $N\ge1$ there exists $\delta>0$ such that
\[
\hR(ib)=\lambda(b)P(b)+\hR(ib)Q(b)\quad\text{for $|b|<\delta$,}
\]
where $\lambda,\,P$ and $Q=I-P$ are $C^\infty$ on $(-\delta,\delta)$
and $\lambda(0)=1$, $\lambda'(0)=-i|\varphi(N)|_1$
and $P(0)v=\int_\bY v\,d\bmu$.
In fact, as shown in~\cite[Section~8.5]{M18},
$\delta>0$ can be chosen uniformly in $N$.
Moreover,
$\|\hR^{(q)}(ib)\|_{\theta}$ is bounded uniformly in $N$ on $(-\delta,\delta)$,
so $\lambda,\,P,\,Q$ are $C^q$ uniformly in $N$ on $(-\delta,\delta)$.

Define
\[
\tP(b)=b^{-1}(P(b)-P(0)), \qquad
\tlambda=b^{-1}(1-\lambda(b)).
\]

\begin{prop} \label{prop:tP}
There exists a constant $C>0$, uniform in $N\ge1$, such that
\[
\|(\tlambda^{-1})^{(q)}(ib)\|_{\theta},\;
\|(\tlambda^{-1}\tP)^{(q)}(ib)\|_{\theta} \le
C|b|^{-(1-\eta)}
 \quad\text{for $|b|<\delta$,}
\]
\end{prop}

\begin{proof}
By~\cite[Proposition~8.18]{M18},
$\|\tP^{(q_1)}(b)\|_{\theta}\ll
\begin{cases}
|b|^{-(1-\eta)} & q_1<\beta-2\eta \\
1 & q_1<\beta-1 \end{cases}
$.
The argument in the proof of~\cite[Proposition~8.26]{M18}
gives the same estimates for $\tlambda^{-1}$ completing the estimates
for $\tlambda^{-1}\tP$.
\end{proof}

\subsection{Proof of Lemma~\ref{lem:trunc}}

Let $\psi$ and $\kappa_m$ be as in Corollary~\ref{cor:move}
with the extra property that $\supp\psi\subset (-\delta,\delta)$.
By Proposition~\ref{prop:fourier},
\begin{equation} \label{eq:psikappa}
\psi, \, \kappa_m \in\cR(t^{-p})
\quad\text{for all $p>0$, $m\ge2$}.
\end{equation}
By Corollary~\ref{cor:move},
we need to show that $\psi(b)\hat\rho_{v,w}(ib)\in\cR(\|v\|_{\gamma,\eta}\|w\|_\gamma\,t^{-(\beta-1)})$ and
$\kappa_m(b)\hat\rho_{v,w}(ib)\in\cR(\|v\|_{\gamma,\eta}\|w\|_\gamma\,t^{-(\beta-1)})$
for all $v\in\cH_{\gamma,\eta}(Y^{\varphi(N)})$,
$w\in\cH_{\gamma}(Y^{\varphi(N)})$, uniformly in $N\ge N_1$.

Let $\hA=\sum_{n=1}^\infty \hA_n$,
$\hB=\sum_{n=1}^\infty \sum_{k=0}^{n-1}\hB_{n,k}$,
$\hC=\sum_{j,k=0}^\infty \hC_{j,k}$.
By Lemma~\ref{lem:ABC},
it remains to show that each of the terms
\begin{align*}
  \psi\hJ, \quad
\psi\hA, \quad
\psi\hB, \quad
\psi\hC; \qquad
  \kappa_m\hJ, \quad
\kappa_m\hA, \quad
\kappa_m\hB, \quad
\kappa_m\hC,
\end{align*}
lies in
$\cR(\|v\|_{\gamma,\eta}\|w\|_\gamma\,t^{-(\beta-1)})$ uniformly in $N\ge1$.

By~Propositions~\ref{prop:poll},~\ref{prop:An} and~\ref{prop:Bnk},
$\hJ,\,
\hA,\, \hB \in\cR( \|v\|_\gamma\|w\|_\gamma \, t^{-(\beta-1)})$.
(Estimates such as these that hold even before truncation are clearly independent of $N$.)
By~\eqref{eq:psikappa} and Proposition~\ref{prop:conv}, uniformly in $N\ge1$,
\[
\psi\hJ,\;
\psi\hA,\;
\psi\hB,\;
\kappa_m\hJ,\;
\kappa_m\hA,\;
\kappa_m\hB
\in\cR(\|v\|_\gamma\|w\|_\gamma \, t^{-(\beta-1)}).
\]
Hence it remains to estimate
$\psi\hC$ and
$\kappa_m\hC$.
The next lemma provides the desired estimates and completes the proof of
Lemma~\ref{lem:trunc} (recall that $q>\beta-1$).

\begin{lemma} \label{lem:Cest}
Assume absence of approximate eigenfunctions.
There exists $N_1\ge1$, $m\ge2$, such that
after truncation, uniformly in $N\ge N_1$,
\begin{itemize}
\item[(a)]
$\kappa_m \hC\in
\cR(\|v\|_{\gamma,\eta} \|w\|_\gamma\, t^{-q})$, and
\item[(b)]
$\psi \hC  \in
\cR(\|v\|_{\gamma,\eta} \|w\|_\gamma\, t^{-(\beta-1)})$,
\end{itemize}
for all $t>1$, $v\in \cH_{\gamma,\eta}(Y^\varphi)$,
$w\in \cH_\gamma(Y^\varphi)$.
\end{lemma}

\begin{proof}
(a)
Let $\ell=\max\{j-k-1,0\}$ and recall that
\[
\SMALL \hC_{j,k}=\int_\bY D_{j,\ell}\bW\!_k\,d\bmu, \qquad
D_{j,\ell}=\hR^\ell\hT R^{j+1}\bV\!_j.
\]

By Proposition~\ref{prop:Tunif}, we can choose $m\ge2$ such that
$\|\kappa_{m-5}\hT\|_\theta\in\cR(t^{-q})$
uniformly in $N\ge N_1$.
Write $\kappa_m=\kappa_3\kappa_{m-5}\kappa_2$, where
$\kappa_i$ is $C^\infty$, vanishes in a neighborhood of zero, and is $O(|b|^{-i})$.
Then
\[
|\kappa_m D_{j,\ell}|_\infty
 \le\|\kappa_3\hR^\ell\|_{\theta}\|\kappa_{m-5}\hT\|_{\theta}\|\kappa_2R^{j+1}\bV\!_j\|_{\theta}.
\]
The estimates for $\hR^\ell$ and $R^{j+1}\bV\!_j$ in
Proposition~\ref{prop:Rn} and Corollary~\ref{cor:VV}
hold even before truncation and hence are uniform in $N\ge1$.
Using~\eqref{eq:psikappa} and Propositions~\ref{prop:fourier}
and~\ref{prop:conv},
\[
  \|\kappa_3\hR^\ell\|_{\theta} \in \cR((\ell+1)^{\beta}t^{-q}) ,\quad
\|\kappa_2R^{j+1}\bV\!_j\|_{\theta}  \in \cR\big(\gamma_1^{j/3}\|v\|_{\gamma,\eta}\,t^{-q}\big),
\]
uniformly in $N\ge 1$.
Since $q>1$, it follows from Proposition~\ref{prop:conv} that
uniformly in $N\ge N_1$,
\begin{align*}
|\kappa_m D_{j,\ell}|_\infty
&
\in\cR\big((\ell+1)^{\beta}t^{-q}\star t^{-q} \star
\gamma_1^{j/3}\|v\|_{\gamma,\eta}\,t^{-q}\big)
 \in \cR\big((\ell+1)^{\beta}\gamma_1^{j/3}\|v\|_{\gamma,\eta}\,t^{-q}\big).
\end{align*}
Also,
$|\bW\!_k|_1\in\cR\big((k+1)^{\beta+1}\gamma_1^k\|w\|_{\gamma}\,t^{-q}\big)$
by Proposition~\ref{prop:Wk} and this is uniform in $N\ge1$.
Applying Proposition~\ref{prop:conv} once more, uniformly in $N\ge N_1$,
\[
\kappa_m\hC_{j,k}
\in\cR((j+1)^\beta
\gamma_1^{j/3}
(k+1)^{\beta+1}\gamma_1^k \|v\|_{\gamma,\eta}\|w\|_{\gamma}\,t^{-q}),
\]
and part~(a) follows.
\\[.75ex]
(b)
As in the proof of Proposition~\ref{prop:Dj}, we write
\[
\SMALL D_{j,\ell}=(1-\lambda)^{-1}\int_Y\hV_j\,d\mu +Q_{j,\ell},
\]
where
\[
 Q_{j,\ell}=
\big(-(\lambda^{\ell-1}+\dots+1)P(0)+\lambda^\ell Q_2+\hR^\ell Q_1\big)R^{j+1}\bV\!_j.
\]
Here, $Q_2=(1-\lambda)^{-1}(P-P(0))=\tlambda^{-1}\tP$.

By Proposition~\ref{prop:sumV},
\[
\SMALL \hC=\sum_{j,k}\int_Y D_{j,\ell}\bW\!_k\,d\mu
=\sum_{j,k}\int_Y Q_{j,\ell}\bW\!_k\,d\mu+(1-\lambda)^{-1}I_0\sum_k \int_Y\bW\!_k\,d\mu,
\]
where $I_0(s)=\int_Y\int_0^{\varphi(y)}e^{-s(\varphi(y)-u)}v(y,u)\,du\,d\mu$.

Choose $\psi_1$ to be $C^\infty$ with compact support such that $\psi_1\equiv 1$
on $\supp\psi$.
By Corollary~\ref{cor:VV} and Propositions~\ref{prop:Rn} and~\ref{prop:tP},
uniformly in $N\ge1$,
\[
|\psi\lambda^\ell Q_2R^{j+1}\bV\!_j|_\infty
\le|\psi\lambda^\ell|\|\psi_1\tlambda^{-1}\tP\|_{\theta}\|\psi_1 R^{j+1}\bV\!_j\|_{\theta}
\in \cR\big((\ell+1)^\beta \gamma_1^{j/3}\|v\|_{\gamma,\eta}\, t^{-q}\big).
\]
The other terms in $Q_{j,\ell}$ are simpler and we obtain that
$|\psi Q_{j,\ell}|_\infty\in\cR\big((\ell+1)^\beta \gamma_1^{j/3}\|v\|_{\gamma,\eta}\, t^{-q}\big)$.  Hence by Proposition~\ref{prop:Wk},
uniformly in $N\ge1$,
\[
\psi\sum_{j,k}\int_Y Q_{j,\ell}\bW\!_k\,d\mu\in\cR(\|v\|_{\gamma,\eta}\|w\|_\gamma\, t^{-q}), \qquad
\sum_{k}\int_Y \bW\!_k\,d\mu\in\cR\big(\|w\|_\gamma\, t^{-q}\big).
\]

To complete the proof, it remains to estimate
$\psi(1-\lambda)^{-1}I_0$.
Recall from~\eqref{eq:I0} that $I_0(0)=0$, so
$(1-\lambda)^{-1}I_0=\tlambda^{-1}\widehat I_1$ where
\[
\widehat I_1(s)=s^{-1}(I_0(s)-I_0(0))=s^{-1}\int_Y\int_0^{\varphi(y)}(e^{-s(\varphi(y)-u)}-1)v(y,u)\,du\,d\mu,
\]
with inverse Laplace transform
$I_1(t)=-\int_Y\int_0^{\varphi(y)}1_{\{\varphi(y)>t+u\}}v(y,u)\,du\,d\mu$.
By Proposition~\ref{prop:varphi}(b),
$|I_1(t)|\le |v|_\infty\int_Y \varphi 1_{\{\varphi>t\}}\,d\mu\ll |v|_\infty\,t^{-(\beta-1)}$, uniformly in $N\ge1$,
so $\widehat I_1\in \cR(|v|_\infty\,t^{-(\beta-1)})$.
Combining this with Proposition~\ref{prop:tP}, we obtain that
\[
\psi(1-\lambda)^{-1}I_0=\psi \tlambda^{-1}\widehat I_1
\in\cR(t^{-q}\star |v|_\infty\,t^{-(\beta-1)})
\in\cR(|v|_\infty\,t^{-(\beta-1)}),
\]
uniformly in $N\ge1$.
\end{proof}

\section{General Gibbs-Markov flows}
\label{sec:nonskew}

In this section, we assume the setup from Section~\ref{sec:skew} but we drop the requirement that $\varphi$ is constant along stable leaves.

In Subsection~\ref{sec:H}, we introduce a criterion, condition~(H), that enables us to reduce  to the skew product Gibbs-Markov maps studied in
Sections~\ref{sec:skew},~\ref{sec:rapid} and~\ref{sec:poly}.  This leads to an enlarged class
of Gibbs-Markov flows for which we can prove results on mixing rates
(Theorem~\ref{thm:nonskew} below).
In Subsection~\ref{sec:period}, we recall criteria for absence of approximate eigenfunctions based on periodic data.

\subsection{Condition~(H)}
\label{sec:H}

Let $F:Y\to Y$ be a map as in Section~\ref{sec:skew} with quotient Gibbs-Markov map $\bF:\bY\to\bY$,
and define $\tY\!_j=Y_j\cap\tY$.
Let $\varphi:Y\to\R^+$ be an integrable roof function with $\inf\varphi>1$
and associated suspension flow $F_t:Y^\varphi\to Y^\varphi$.

We no longer assume that $\varphi$ is constant along stable leaves.  Instead of
condition~\eqref{eq:inf} we require that
\begin{equation}
\label{eq:phi}
 |\varphi(y)-\varphi(y')|  \le C_1 \infYj\varphi\,\gamma^{s(y,y')}\quad\text{for all $y,y'\in \tY\!_j$, $j\ge1$.}
\end{equation}
(Clearly, if $\varphi$ is constant along stable leaves, then conditions~\eqref{eq:inf} and~\eqref{eq:phi} are identical.)

Recall that
$\pi:Y\to\tY$ is the projection along stable leaves.
Define
\[
\SMALL\chi(y)=\sum_{n=0}^\infty (\varphi(F^n\pi y)-\varphi(F^ny)),
\]
for all $y\in Y$ such that the series converges absolutely.  We assume

\begin{itemize}
\item[(H)]
\begin{itemize}
\item[(a)]  The series converges almost surely on $Y$ and $\chi\in L^\infty(Y)$.
\item[(b)]
There are constants $C_3\ge1$, $\gamma\in(0,1)$ such that
\[
|\chi(y)-\chi(y')|   \le C_3 (d(y,y')+\gamma^{s(y,y')})
\quad\text{ for all $y,y'\in Y$.}
\]
\end{itemize}
\end{itemize}
When conditions~\eqref{eq:phi} and (H) are satisfied, we call $F_t$ a {\em Gibbs-Markov flow}.
(If $\varphi$ is constant along stable leaves then $\chi=0$, so every skew product Gibbs-Markov flow is a Gibbs-Markov flow.)

Since $\inf\varphi>0$, it follows that $\varphi_n=\sum_{j=0}^{n-1}\varphi\circ F^j\ge 4|\chi|_\infty+1$ for all $n$ sufficiently large.  For simplicity we suppose from now on that $\inf\varphi\ge 4|\chi|_\infty+1$
(otherwise, replace $F$ by $F^n$).

Define
\begin{equation} \label{eq:coh}
\tilde\varphi=\varphi+\chi-\chi\circ F.
\end{equation}
Note that
$\inf\tilde\varphi\ge \inf\varphi-2|\chi|_\infty\ge1$ and
$\int_Y \tilde\varphi\,d\mu = \int_Y \varphi\,d\mu$,
so $\tilde\varphi:Y\to\R^+$ is an integrable roof function.
Hence we can define the suspension flow $\tF_t:Y^{\tilde\varphi}\to Y^{\tilde\varphi}$.
Also, a calculation shows that
$\tilde\varphi(y)=\sum_{n=0}^\infty(\varphi(F^n\pi y)-\varphi(F^n\pi{Fy}))$,
so $\tilde\varphi$ is constant along stable leaves
and we can define the quotient roof function $\bphi:\bY\to\R^+$ with
quotient semiflow $\bF_t:\bY^{\tilde\varphi}\to \bY^{\tilde\varphi}$.

In the remainder of this section, we prove that $\tF_t$ is a skew product Gibbs-Markov flow (and hence $\bF_t$ is a Gibbs-Markov semiflow), and show that (super)polynomial decay of correlations for $\tF_t$ is inherited by $F_t$.

\begin{prop} \label{prop:GMskewGM}
Let $F_t:Y^{\varphi}\to Y^{\varphi}$ be a Gibbs-Markov flow.  Then
$\tF_t:Y^{\tilde\varphi}\to Y^{\tilde\varphi}$
is a skew product Gibbs-Markov flow.
\end{prop}

\begin{proof}
We verify that the setup in Section~\ref{sec:skew} holds.
All the conditions on the map $F:Y\to Y$ are satisfied by assumption.
Hence it suffices to check that $\tilde\varphi$
satisfies condition~\eqref{eq:inf}.

Let $y,y'\in \tY\!_j$ for some $j\ge1$.
By~\eqref{eq:WuF}, $d(y,y')\le C_2\gamma^{s(y,y')}$ and
$d(Fy,Fy')\le C_2\gamma^{s(y,y')-1}$.  By~(H)(b),
$|\chi(y)-\chi(y')| \le 2C_2C_3\gamma^{s(y,y')}$ and
$|\chi(Fy)-\chi(Fy')|  \le 2C_2C_3\gamma^{s(y,y')-1}$.
Hence by~\eqref{eq:phi} and~\eqref{eq:coh},
\[
|\tilde\varphi(y)-\tilde\varphi(y')|  \le
|\varphi(y)-\varphi(y')|+|\chi(y)-\chi(y')|+|\chi(Fy)-\chi(Fy')|
 \ll \infYj\varphi\,\gamma^{s(y,y')}.
\]
Also, $\infYj\varphi\le \infYj\tilde\varphi+2|\chi|_\infty
\le \infYj\tilde\varphi+\frac12\inf \varphi
\le \infYj\tilde\varphi+\frac12\infYj \varphi$.
Hence $\infYj\varphi\le 2\infYj\tilde\varphi$ and
$|\tilde\varphi(y)-\tilde\varphi(y')| \ll \infYj\tilde\varphi\,\gamma^{s(y,y')}$ as required.
\end{proof}

\begin{cor} \label{cor:phionY}
There is a constant $C>0$ such that
\[
|\varphi(y)-\varphi(y')|\le C\infYj\varphi\{d(y,y')+d(Fy,Fy')+\gamma^{s(y,y')}\}
\quad\text{for all $y,y'\in Y_j$, $j\ge1$.}
\]
\end{cor}

\begin{proof}
Let $\tilde y=\tY\cap W^s(y)$,
$\tilde y'=\tY\cap W^s(y')$.
Since $\tilde\varphi$ is constant along stable leaves, it follows as in the proof
of Proposition~\ref{prop:GMskewGM} that
\[
|\tilde\varphi(y)-\tilde\varphi(y')|=
|\tilde\varphi(\tilde y)-\tilde\varphi(\tilde y')|\ll
\infYj\varphi\,\gamma^{s(\tilde y,\tilde y')}
=\infYj\varphi\,\gamma^{s(y,y')}.
\]
Hence by~\eqref{eq:coh} and (H)(b)
\begin{align*}
|\varphi(y)-\varphi(y')| & \le
|\tilde\varphi(y)-\tilde\varphi(y')|+|\chi(Fy)-\chi(Fy')| +|\chi(y)-\chi(y')|
\\ & \ll \infYj\varphi\{\gamma^{s(y,y')}+d(Fy,Fy')+\gamma^{s(Fy,Fy')}+d(y,y')\}.
\end{align*}
The result follows since $\gamma^{s(Fy,Fy')}=\gamma^{-1}\gamma^{s(y,y')}$.
\end{proof}

Next, we relate the two suspension flows
$F_t:Y^\varphi\to Y^\varphi$ and
$\tF_t:Y^{\tilde\varphi}\to Y^{\tilde\varphi}$.
Note that $(y,\varphi(y))$ is identified with $(Fy,0)$ in the first flow
and $(y,\tilde\varphi(y))$ is identified with $(Fy,0)$ in the second flow.
Define
\begin{align*}
& g_+:Y^\varphi\to Y^{\tilde\varphi}, \qquad g_+(y,u)=(y,u+\chi(y)+|\chi|_\infty), \\
& g_-:Y^{\tilde\varphi}\to Y^\varphi, \qquad g_-(y,u)=(y,u-\chi(y)+|\chi|_\infty),
\end{align*}
computed modulo identifications.
Using~\eqref{eq:coh} and the identifications
 on $Y^{\tilde\varphi}$,
\begin{align*}
g_+(y,\varphi(y))=(y,\varphi(y)+\chi(y)+|\chi|_\infty)& =
(y,\tilde\varphi(y)+\chi(Fy)+|\chi|_\infty) \\ &\sim (Fy,\chi(Fy)+|\chi|_\infty)=g_+(Fy,0),
\end{align*}
so $g_+$ respects the identification on $Y^\varphi$ and hence is well-defined.
It follows easily that $g_+:Y^\varphi\to Y^{\tilde\varphi}$
is a measure-preserving semiconjugacy between the two suspension flows.
Similarly, $g_-$ is well-defined and
$g_-\circ g_+=F_{2|\chi|_\infty}:Y^\varphi\to Y^\varphi$.

Given observables $v,w:Y^\varphi\to\R$, let
$\tilde v=v\circ g_-,\,\tilde w=w\circ g_-:Y^{\tilde\varphi}\to\R$.
When speaking of $\cH_\gamma(Y^{\tilde\varphi})$ and so on, we use the metric
$d_1(y,y')=d(y,y')^\eta$ on $Y$ instead of $d$.  Let $\gamma_1=\gamma^\eta$.

Let $\cH_{\gamma,\eta}^*(Y^\varphi)=\{v:Y^\varphi\to\R:\|v\|^*_{\gamma,\eta}<\infty\}$ and
$\cH_{\gamma,0,m}^*(Y^\varphi)=\{w:Y^\varphi\to\R:\|w\|^*_{\gamma,0,m}<\infty\}$
where
\[
\|v\|^*_{\gamma,\eta}= \|v\|_{\gamma,\eta}+
\|v\circ F_{2|\chi|_\infty}\|_{\gamma,\eta}, \qquad
\|w\|^*_{\gamma,0,m}=\|w\|_{\gamma,0,m}+\|w\circ F_{2|\chi|_\infty}\|_{\gamma,0,m}.
\]

\begin{lemma}  \label{lem:nonskew}
Let $v\in\cH^*_{\gamma,\eta}(Y^\varphi)$,
$w\in\cH^*_{\gamma,0,m}(Y^\varphi)$, for some $m\ge1$.
Then
$\tilde v\in\cH_{\gamma_1,\eta}(Y^{\tilde\varphi})$,
$\tilde w\in \cH_{\gamma_1,0,m}(Y^{\tilde\varphi})$, and
$\|\tilde v\|_{\gamma_1,\eta}\le 4C_3\|v\|^*_{\gamma,\eta}$,
$\|\tilde w\|_{\gamma_1,0,m}\le 2C_3\|w\|^*_{\gamma,0,m}$.
\end{lemma}

\begin{proof}
We have $\tilde v(y,u)=v(y,u-\chi(y)+|\chi|_\infty)$.
It is immediate that
$|\tilde v|_\infty\le |v|_\infty$.

Now let $(y,u),\,(y',u)\in Y^{\tilde\varphi}$.
Suppose without loss that $\chi(y)\ge\chi(y')$.
First, we consider the case
$u-\chi(y)+|\chi|_\infty\le\varphi(y)$,
$u-\chi(y')+|\chi|_\infty\le\varphi(y')$.
By (H)(b) and the definition of $\|v\|_{\gamma,\eta}$,
\begin{align*}
 |\tilde v(y,u)-\tilde v(y',u)|
 & \le \big|v(y,u-\chi(y)+|\chi|_\infty)-v(y',u-\chi(y)+|\chi|_\infty)\big|
\\ & \qquad + \big|v(y',u-\chi(y)+|\chi|_\infty)-v(y',u-\chi(y')+|\chi|_\infty)\big|
\\ & \le |v|_\gamma \varphi(y)(d(y,y')+\gamma^{s(y,y')})
+|v|_{\infty,\eta}|\chi(y)-\chi(y')|^\eta
\\ & \le 2|v|_\gamma \tilde\varphi(y)(d(y,y')+\gamma^{s(y,y')})
+|v|_{\infty,\eta}C_3 (d(y,y')+\gamma^{s(y,y')})^\eta
\\ & \le 2C_3\|v\|_{\gamma,\eta}\tilde\varphi(y) (d_1(y,y')+\gamma_1^{s(y,y')}).
\end{align*}

Second, we consider the case $u\ge\chi(y)+|\chi|_\infty\ge \chi(y')+|\chi|_\infty$.
Then we can write
$g_-(y,u)=F_\sigma(y,u-\chi(y)-|\chi|_\infty)$,
$g_-(y',u)=F_\sigma(y',u-\chi(y')-|\chi|_\infty)$ where $\sigma=2|\chi|_\infty$,
so
 \[
|\tilde v(y,u)-\tilde v(y',u)|
 =\big|v\circ F_\sigma(y,u-\chi(y)-|\chi|_\infty)-v\circ F_\sigma(y',u-\chi(y')-|\chi|_\infty)\big|.
\]
Proceeding as in the first case,
\[
|\tilde v(y,u)-\tilde v(y',u)|\le
2C_3\|v\circ F_\sigma\|_{\gamma,\eta}\tilde\varphi(y) (d_1(y,y')+\gamma_1^{s(y,y')}).
\]

This leaves the case $u<\chi(y)+|\chi|_\infty\le 2|\chi|_\infty$
and $u\ge\min\{\varphi(y)+\chi(y)-|\chi|_\infty,
\varphi(y')+\chi(y')-|\chi|_\infty\}\ge \inf\varphi-2|\chi|_\infty$.
This is impossible since $\inf\varphi>4|\chi|_\infty$.
Hence
\[
|\tilde v(y,u)-\tilde v(y',u)|\le
2C_3\|v\|^*_{\gamma,\eta}\tilde\varphi(y) (d_1(y,y')+\gamma_1^{s(y,y')})
\quad\text{for all $(y,u),(y',u)\in Y^{\tilde\varphi}$},
\]
so $\|\tilde v\|_{\gamma_1}\le 2C_3\|v\|^*_{\gamma,\eta}$.

The estimate for $|\tilde v|_{\infty,\eta}$ splits into cases similarly.
Let $0\le u<u'\le\tilde\varphi(y)$.
Then
\[
|\tilde v(y,u)-\tilde v(y,u')|\le \begin{cases}
|v|_{\infty,\eta}|u-u'|^\eta &
u'-\chi(y)+|\chi|_\infty\le\varphi(y) \\
|v\circ F_\sigma|_{\infty,\eta}|u-u'|^\eta &
u\ge \chi(y)+|\chi|_\infty \end{cases}.
\]
This leaves the case $u'-\chi(y)+|\chi|_\infty>\varphi(y)$
and $u<\chi(y)+|\chi|_\infty$.
But then $u'-u>\varphi(y)+2\chi(y)>\varphi(y)-2|\chi|_\infty>\frac12\varphi(y)\ge\frac12$, so we obtain
$
|\tilde v(y,u)-\tilde v(y,u')|\le 2|v|_\infty\le 4|v|_\infty|u-u'|^\eta$.
Hence
$|\tilde v|_{\infty,\eta}\le 4\|v\|^*_{\gamma,\eta}$ completing the estimate
for
$\|\tilde v\|_{\gamma_1,\eta}$.
The calculation for $\tilde w$ is similar.~
\end{proof}

We say that a Gibbs-Markov flow has approximate eigenfunctions if this is the case for $\tF_t$ (equivalently $\bF_t$).

\begin{thm} \label{thm:nonskew}
Suppose that $F_t:Y^\varphi\to Y^\varphi$ is a Gibbs-Markov flow
such that $\mu(\varphi>t)=O(t^{-\beta})$ for some $\beta>1$.
Assume absence of approximate eigenfunctions.
Then there exists $m\ge1$ and $C>0$ such that
\[
|\rho_{v,w}(t)|\le C\|v\|^*_{\gamma,\eta}\|w\|^*_{\gamma,0,m} \,t^{-(\beta-1)}
\quad\text{for all $v\in \cH^*_{\gamma,\eta}(Y^\varphi)$, $w\in \cH^*_{\gamma,0,m}(Y^\varphi)$, $t>1$}.
\]
\end{thm}

\begin{proof}
Since $g_+$ is a measure-preserving semiconjugacy and $g_-\circ g_+=F_{2|\chi|_\infty}$,
\begin{align*}
\int_{Y^\varphi} v\,w\circ F_t\,d\mu^\varphi
& =
\int_{Y^\varphi} v\circ g_-\circ g_+\,w\circ g_-\circ g_+\circ F_t\,d\mu^\varphi
\\
&=
\int_{Y^\varphi} \tilde v\circ g_+\,\tilde w\circ \tF_t \circ g_+\,d\mu^\varphi
=\int_{Y^{\tilde\varphi}}\tilde v\,\tilde w\circ \tF_t\,d\mu^{\tilde\varphi}
\end{align*}
 where $\tF_t$ does not possess approximate eigenfunctions.
Note also that $\mu(\tilde\varphi>t)=O(t^{-\beta})$.
By Lemma~\ref{lem:nonskew},
$\tilde v\in\cH_{\gamma_1,\eta}(Y^{\tilde\varphi})$,
$\tilde w\in \cH_{\gamma_1,0,m}(Y^{\tilde\varphi})$.

By Theorem~\ref{thm:polyflow}, we can choose $m\ge1$ such that
$|\rho_{v,w}(t)|=|\int_{Y^{\tilde\varphi}}\tilde v\,\tilde w\circ \tF_t\,d\mu^{\tilde\varphi}
-\int_{Y^{\tilde\varphi}}\tilde v\,d\mu^{\tilde\varphi}
\int_{Y^{\tilde\varphi}}\tilde w\,d\mu^{\tilde\varphi}|\ll
 \|\tilde v\|_{\gamma_1,\eta}\|\tilde w\|_{\gamma_1,0,m}\,t^{-(\beta-1)}
 \le 8C_3^2 \|v\|^*_{\gamma,\eta}\|w\|^*_{\gamma,0,m}\,t^{-(\beta-1)}$.
\end{proof}

\subsection{Periodic data and absence of approximate eigenfunctions}
\label{sec:period}

In this subsection, we recall the relationship between periodic data and approximate eigenfunctions and review two sufficient conditions to rule out the existence of approximate eigenfunctions.
We continue to assume that $F_t$ is a Gibbs-Markov flow as in Subsection~\ref{sec:H}.

Define $\varphi_n=\sum_{j=0}^{n-1}\varphi\circ F^j$.
Similarly, define $\tilde\varphi_n$ and $\bphi_n$.
If $y$ is a periodic point of period~$p$ for $F$ (that is, $F^py=y$),
then $y$ is periodic of period $\period=\varphi_p(y)$ for $F_t$
(that is, $F_\period y=y$).
Recall that $\bar\pi:Y\to\bY$ is the quotient projection.

\begin{prop} \label{prop:period}
Suppose that there exist approximate eigenfunctions on $Z_0\subset \bY$.
Let $\alpha,C, b_k,n_k$ be as in Definition~\ref{def:approx}.
If $y\in\bar\pi^{-1}Z_0$ is a periodic point with $F^py=y$ and $F_\period y=y$
where $\period=\varphi_p(y)$, then
\begin{equation} \label{eq:period}
\dist(b_kn_k\period-p\psi_k,2\pi\Z)\le C(\inf\varphi)^{-1}\period|b_k|^{-\alpha}
\quad\text{for all $k\ge1$}.
\end{equation}
\end{prop}

\begin{proof}
Define $\bar y=\bar\pi y\in Z_0$ and note that $\bF^p\bar y=\bF^p\bar\pi y=\bar \pi F^py=\bar y$.
By~\eqref{eq:coh},
\begin{align*}
\bphi_p(\bar y)=\tilde\varphi_p(y)=\varphi_p(y)+\chi(y)-\chi(F^py)=\period.
\end{align*}
Now $(M_b^pv)(\bar y)=e^{ib\bphi_p(\bar y)}v(\bF^p\bar y)=e^{ib\period}v(\bar y)$.
Hence substituting $\bar y$ into~\eqref{eq:approx},
we obtain $|e^{ib_kn_k\period}-e^{ip\psi_k}|\le Cp|b_k|^{-\alpha}$.
Also $\period=\varphi_p(y)\ge p\inf\varphi$.
\end{proof}

The following Diophantine condition is based on~\cite[Section~13]{Dolgopyat98b}.
(Unlike in~\cite{Dolgopyat98b}, we have to consider periods corresponding to three periodic points instead of two.)

\begin{prop} \label{prop:tau}
Let $y_1,y_2,y_3\in\bigcup Y_j$ be fixed points for $F$,
and let $\period_i=\varphi(y_i)$, $i=1,2,3$, be the corresponding periods
for $F_t$.
Let $Z_0\subset\bY$ be the finite subsystem corresponding to the three partition elements containing $\bar\pi y_1,\bar\pi y_2,\bar\pi y_3$.

If $(\period_1-\period_3)/(\period_2-\period_3)$ is Diophantine, then
 there do not exist approximate eigenfunctions on $Z_0$.
\end{prop}

\begin{proof}
Using Proposition~\ref{prop:period}, the proof is identical to that of~\cite[Proposition~5.3]{M18}.
\end{proof}

The condition in Proposition~\ref{prop:tau} is satisfied with probability one but is not robust.  Using the notion of {\em good asymptotics}~\cite{FMT07}, we obtain an open and dense condition.

\begin{prop} \label{prop:good}
Let $Z_0\subset\bY$ be a finite subsystem.
Let $y_0\in \bar\pi^{-1}Z_0$ be a fixed point for~$F$ with period $\period_0=\varphi(y_0)$ for the flow.
Let $y_N\in \bar\pi^{-1}Z_0$, $N\ge1$, be a sequence of periodic points with $F^Ny_N=y_N$ such that
their periods $\period_N=\varphi_N(y_N)$ for the flow $F_t$ satisfy
\[
\period_N=N\period_0+\kappa+E_N\gamma^N\cos(N\omega+\omega_N)+o(\gamma^N),
\]
where $\kappa\in\R$, $\gamma\in(0,1)$ are constants,
$E_N\in\R$ is a bounded sequence with \mbox{$\liminf_{N\to\infty}|E_N|>0$},
and either (i) $\omega=0$ and $\omega_N\equiv0$, or (ii)
$\omega\in(0,\pi)$ and $\omega_N\in(\omega_0-\pi/12,\omega_0+\pi/12)$
for some $\omega_0$.
(Such a sequence of periodic points is said to have {\em good asymptotics}.)

Then there do not exist approximate eigenfunctions on $Z_0$.
\end{prop}

\begin{proof}
Using Proposition~\ref{prop:period}, the proof is identical to that of~\cite[Proposition~5.5]{M18}.
\end{proof}

By~\cite{FMT07}, for any finite subsystem $Z_0$, the existence of periodic points with good asymptotics in $\bar\pi^{-1}Z_0$ is a $C^2$-open and $C^\infty$-dense condition.   Although~\cite{FMT07} is set in the uniformly hyperbolic setting, the construction applies directly to the current set up as we now explain.
Assume that $(Y,d)$ is a Riemannian manifold.  Let
$\bar Z_1$ and $\bar Z_2$ be two of the partition elements in $Z$ and
set $Z_j=\Int\bar\pi^{-1}\bar Z_j$ for $j=1,2$.
Assume that $Z_1$, $Z_2$ are submanifolds of $Y$ and that $F$ and
$\varphi$ are $C^r$ when restricted to $Z_1\cup Z_2$ for some $r\ge2$.

Let $y_0\in Z_1$ be a fixed point for $F$ and choose a transverse homoclinic point in $Z_2$.
Following~\cite{FMT07}, we construct a sequence of $N$-periodic points $y_N$, $N\ge1$, for $F$ with orbits lying in $Z_1\cup Z_2$. The sequence automatically has good asymptotics except that in exceptional cases it may be that $\liminf_{N\to\infty}|E_{N}|=0$.  By~\cite{FMT07}, the liminf is positive for a $C^2$ open and $C^r$ dense set of roof functions $\varphi$.

Combining this construction with Proposition~\ref{prop:good}, it follows that nonexistence of approximate eigenfunctions holds for an open and dense set of smooth Gibbs-Markov flows.

\part{Mixing rates for nonuniformly hyperbolic flows}

In this part of the paper, we show how the results for suspension flows in Part~I can be translated into results for nonuniformly hyperbolic flows defined on an ambient manifold.  In Section~\ref{sec:NUHflow}, we show how this is done under the assumption that condition~(H) from Section~\ref{sec:nonskew} is valid.
In Section~\ref{sec:chi}, we describe a number of situations where condition~(H) is satisfied.  This includes all the examples considered here and in~\cite{M18}.  In Section~\ref{sec:Lorentz}, we consider in detail the planar infinite horizon Lorentz gas.

\section{Nonuniformly hyperbolic flows and suspension flows}
\label{sec:NUHflow}

In this section, we describe
a class of nonuniformly hyperbolic flows $T_t:M\to M$ that
have most of the properties required for $T_t$ to be modelled by a Gibbs-Markov flow.  (The remaining property, condition~(H) from Section~\ref{sec:nonskew}, is considered in Section~\ref{sec:chi}.)

In Subsection~\ref{sec:NUH}, we consider a class of
nonuniformly hyperbolic transformations
$f:X\to X$ modelled by a Young tower~\cite{Young98,Young99},
making explicit the conditions from~\cite{Young98} that are needed for this paper.
In Subsection~\ref{sec:holder}, we consider flows that are H\"older suspensions over such a map $f$ and show how to model them, subject to condition~(H), by a Gibbs-Markov flow.
In Subsection~\ref{sec:dyn}, we generalise the H\"older structures in Subsection~\ref{sec:holder} to ones that are dynamically H\"older.

In applications, $f$ is typically a first-hit Poincar\'e map for the flow $T_t$ and hence is invertible.  Invertibility is used in Proposition~\ref{prop:occas} but not elsewhere, so many of our results do not rely on injectivity of $f$.

\subsection{Nonuniformly hyperbolic transformations $f:X\to X$}
\label{sec:NUH}

Let $f:X\to X$ be a measurable transformation defined
on a metric space $(X,d)$ with $\diam X\le1$.
We suppose that $f$ is nonuniformly hyperbolic in the sense that it is modelled by a Young tower~\cite{Young98,Young99}.
We recall the metric parts of the theory; the differential geometry part leading to an SRB or physical measure does not play an important role here.

\vspace{-2ex}
\paragraph{Product structure}
Let $Y$ be a measurable subset of $X$.
Let $\cW^s$ be a collection of disjoint measurable subsets of $X$ (called ``stable leaves'') and let
$\cW^u$ be a collection of disjoint measurable subsets of $X$ (called ``unstable leaves'') such that each collection covers~$Y$.
Given $y\in Y$, let $W^s(y)$ and $W^u(y)$ denote the stable and unstable leaves containing~$y$.

We assume that for all $y,y'\in Y$, the intersection $W^s(y)\cap W^u(y')$ consists of
precisely one point, denoted
$z=W^s(y)\cap W^u(y')$, and that $z\in Y$.
Also we suppose there is a constant $C_4\ge1$ such that
\begin{equation}\label{eq:YoungProductStructure}
d(y,z)\le C_4 d(y,y') \quad\text{for all $y,y'\in Y$, $\,z=W^s(y)\cap W^u(y')$.}
\end{equation}

\vspace{-2ex}
\paragraph{Induced map}
Next, let $\{Y_j\}$ be an at most countable measurable partition  of $Y$ such that $Y_j=\bigcup_{y\in Y_j}W^s(y)\cap Y$ for all $j\ge1$.
Also, fix $\tau:Y\to\Z^+$ constant on partition elements such that $f^{\tau(y)}y\in Y$ for all $y\in Y$.  Define $F:Y\to Y$ by $Fy=f^{\tau(y)}y$.
Let $\mu$ be an ergodic $F$-invariant probability measure on $Y$
and suppose that $\tau$ is integrable.  (It is not assumed that $\tau$ is the first return time to $Y$.)

As in Section~\ref{sec:skew}, we suppose that $F(W^s(y))\subset W^s(Fy)$ for all $y\in Y$.
Let $\bY$ denote the space obtained from $Y$ after quotienting by $\cW^s$,
with natural projection $\bar\pi:Y\to\bY$.
We assume that the quotient map $\bF:\bY\to\bY$ is a Gibbs-Markov map as
in Definition~\ref{def:GM},
with partition $\{\bY\!_j\}$
 and ergodic invariant probability measure $\bmu =\bar\pi_*\mu$.
Let $s(y,y')$ denote the separation time on $\bY$.

\vspace{-2ex}
\paragraph{Contraction/expansion}

Let $Y_j=\bar\pi^{-1}\bY\!_j$; these form a partition of $Y$ and each $Y_j$ is a union of stable leaves.
The separation time extends to $Y$, setting
$s(y,y')=s(\bar\pi y,\bar\pi y')$ for $y,y'\in Y$.

We assume that there are constants $C_2\ge1$, $\gamma\in(0,1)$
such that for all $n\ge0$, $y,y'\in Y$,
\begin{alignat}{2} \label{eq:Ws}
d(f^ny,f^ny') & \le C_2\gamma^{\psi_n(y)}d(y,y') && \quad\text{for all $y'\in W^s(y)$}, \\ \label{eq:Wu}
d(f^ny,f^ny') & \le C_2\gamma^{s(y,y')-\psi_n(y)} && \quad\text{for all $y'\in W^u(y)$},
\end{alignat}
where $\psi_n(y)=\#\{j=1,\dots,n:f^jy\in Y\}$ is the number of returns of $y$ to $Y$ by time~$n$.
Note that conditions~\eqref{eq:WsF} and~\eqref{eq:WuF} are special cases of~\eqref{eq:Ws} and~\eqref{eq:Wu} where $\tY$ can be chosen to be any fixed unstable leaf.
In particular, all the conditions on $F$ in Sections~\ref{sec:skew} and~\ref{sec:nonskew} are satisfied.

In Sections~\ref{sec:dyn},~\ref{sec:D} and~\ref{sec:Lorentz}, we make use of the condition
\begin{equation} \label{eq:occas}
F(W^u(y)\cap Y_j)=W^u(Fy)\cap Y \quad\text{for all $y\in Y_j$, $j\ge1$.}
\end{equation}

\begin{rmk} \label{rmk:further}
  Further hypotheses in~\cite{Young98} ensure the existence of SRB measures on $\bY$, $Y$ and $X$.  These assumptions are not required here and no special properties of $\mu$ and $\bmu$ (other than the properties mentioned above) are used.
\end{rmk}

\begin{rmk}  The abstract setup in~\cite{Young98}  essentially satisfies all of the assumptions above.
However condition~\eqref{eq:Ws} is stated in the slightly weaker form
$d(f^ny,f^ny')\le C_2\gamma^{\psi_n(y)}$.
As pointed out in~\cite{Demers10}, the stronger form~\eqref{eq:Ws}
is satisfied in all known examples
where the weaker form holds.

Condition~\eqref{eq:occas} is not stated explicitly in~\cite{Young98} but is an automatic consequence of the set up therein provided $f:X\to X$ is injective.  We provide the details in Proposition~\ref{prop:occas}.
In the examples considered in this paper and in~\cite{M18},
the map $f$ is a first return map for a flow and hence is injective, so condition~\eqref{eq:occas} is not very restrictive.

Condition~\eqref{eq:occas} is also used in~\cite[Section~5.2]{M18} but is stated there in a slightly different form.  In~\cite{M18}, the subspace $X$ is not needed (and hence not mentioned) and the stable and unstable disks $W^s(y)$, $W^u(y)$ are replaced by their intersections with $Y$.  Hence the condition
$F(W^u(y)\cap Y_j)\supset W^u(Fy)$ for $y\in Y_j$ in~\cite[Section~5.2]{M18}
becomes
$F(W^u(y)\cap Y_j)\supset W^u(Fy)\cap Y$ for $y\in Y_j$ in our present notation and hence holds by~\eqref{eq:occas}.
\end{rmk}

\begin{prop} \label{prop:prod}
$d(f^ny,f^ny') \le C_2C_4(\gamma^{\psi_n(y)}d(y,y')+\gamma^{s(y,y')-\psi_n(y)})$
for all $y,y'\in Y$, \mbox{$n\ge0$}.
\end{prop}

\begin{proof}
Let $z=W^s(y)\cap W^u(y')$.  Note that
$s(z,y')=s(y,y')$ and $\psi_n(z)=\psi_n(y)$.
Hence
\begin{align*}
d(f^ny,f^ny') \le d(f^ny,f^nz)+ d(f^nz,f^ny')
& \le C_2(\gamma^{\psi_n(y)}d(y,z)+\gamma^{s(z,y')-\psi_n(z)})
\\ & \le C_2C_4(\gamma^{\psi_n(y)}d(y,y')+\gamma^{s(y,y')-\psi_n(y)}),
\end{align*}
as required.
\end{proof}

\subsection{H\"older flows and observables}
\label{sec:holder}

Let $T_t:M\to M$ be a flow defined on a metric space $(M,d)$ with $\diam M\le 1$.  Fix $\eta\in(0,1]$.

Given $v:M\to\R$, define
${|v|}_{C^\eta}=\sup_{x\neq x'}|v(x)-v(x')|/d(x,x')^\eta$ and ${\|v\|}_{C^\eta}=|v|_\infty+{|v|}_{C^\eta}$.
Let $C^\eta(M)=\{v:M\to\R:{\|v\|}_{C^\eta}<\infty\}$.
Also, define
${|v|}_{C^{0,\eta}}=\sup_{x\in M,\,t>0}|v(T_tx)-v(x)|/t^\eta$ and
let $C^{0,\eta}(M)=\{v:M\to\R:|v|_\infty+{|v|}_{C^{0,\eta}}<\infty\}$.
(Such observables are H\"older in the flow direction.)

We say that $w:M\to\R$ is {\em differentiable in the flow direction} if the limit
$\partial_tw=\lim_{t\to0}(w\circ T_t-w)/t$ exists pointwise.
Define ${\|w\|}_{C^{\eta,m}}=\sum_{j=0}^m{\|\partial_t^jw\|}_{C^\eta}$ and let
$C^{\eta,m}(M)=\{w:M\to\R:{\|w\|}_{C^{\eta,m}}<\infty\}$.

Let $X\subset M$ be a Borel subset 
and define $C^\eta(X)$ using the metric $d$ restricted to $X$.
We suppose that
$T_{h(x)}x\in X$ for all $x\in X$,
 where $h:X\to\R^+$ lies in $C^\eta(X)$ and $\inf h>0$.
In addition, we suppose that for any $D_1>0$ there exists $D_2>0$ such that
\begin{equation} \label{eq:Tt}
d(T_tx,T_tx')\le D_2d(x,x')^\eta
\quad\text{for all $t\in[0,D_1]$, $x,x'\in M$.}
\end{equation}

Define $f:X\to X$ by $fx=T_{h(x)}x$.  We suppose that $f$ is a nonuniformly hyperbolic transformation as in Subsection~\ref{sec:NUH},
with induced map $F=f^\tau:Y\to Y$ and so on.

Define $h_\ell=\sum_{j=0}^{\ell-1}h\circ f^j$.
We define the induced roof function
\[
\SMALL \varphi=h_\tau:Y\to\R^+, \qquad \varphi(y)=\sum_{\ell=0}^{\tau(y)-1}h(f^\ell y).
\]
Note that $h\le \varphi\le |h|_\infty\tau$ so $\varphi\in L^1(Y)$ and $\inf \varphi>0$.
Define the suspension flow $F_t:Y^\varphi\to Y^\varphi$ as in~\eqref{eq:susp}.

To deduce rates of mixing for nonuniformly hyperbolic flows from the corresponding result for Gibbs-Markov flows, Theorem~\ref{thm:nonskew}, we need to verify that
\begin{itemize}

\parskip = -2pt
\item[(i)] Condition~\eqref{eq:phi} holds.
\item[(ii)] Condition~(H) from Section~\ref{sec:nonskew} holds.
\item[(iii)] Regular observables on $M$ lift to regular observables on $Y^\varphi$.
\end{itemize}
Ingredients~(i) and~(ii) guarantee that the suspension flow $F_t:Y^{\varphi}\to Y^{\varphi}$ is a Gibbs-Markov flow and ingredient~(iii) ensures that Theorem~\ref{thm:nonskew} applies to the appropriate observables on~$M$.

In the remainder of this subsection, we deal with ingredients~(i) and~(iii).
First, we verify that $\varphi$ satisfies condition~\eqref{eq:phi}.
Let $d_1(y,y')=d(y,y')^\eta$ and $\gamma_1=\gamma^\eta$.

\begin{prop} \label{prop:hell}
Let $y,y'\in Y_j$ for some $j\ge1$ and let $\ell=0,\ldots,\tau(y)-1$.
Then
\[
|h_\ell(y) -h_\ell(y')|
\le C_2C_4|h|_\eta\, \ell (d_1(y,y')+\gamma_1^{s(y,y')}).
\]
Moreover,
\[
|\varphi(y)  -\varphi(y')|
\le 2C_2^2C_4(\inf h)^{-1}|h|_\eta\, \infYj\varphi\,\gamma_1^{s(y,y')}
\quad\text{for all $y,y'\in \tY_j$, $j\ge1$.}
\]
\end{prop}

\begin{proof}
 Note that $\psi_\ell(y)=0$, so by Proposition~\ref{prop:prod},
\begin{equation} \label{eq:prodl}
d(f^\ell y,f^\ell y')\le C_2C_4(d(y,y')+\gamma^{s(y,y')}).
\end{equation}
Hence
\begin{align*}
|h_\ell(y)  -h_\ell(y')|   & \le \sum_{j=0}^{\ell-1}|h(f^j y)-h(f^j y')|
\\ &  \le |h|_\eta \sum_{j=0}^{\ell-1}d(f^j y,f^j y')^\eta
\le C_2C_4|h|_\eta\,\ell (d_1(y,y')+\gamma_1^{s(y,y')}),
\end{align*}
establishing the estimate for $h_\ell$.
Also, $\tau(y)\le(\inf h)^{-1}\inf1_{Y_j}\varphi$, so
taking $\ell=\tau(y)$ and using~\eqref{eq:Wu} with $n=0$, we obtain the estimate for $\varphi$.
\end{proof}

Next we deal with ingredient~(iii) assuming~(ii). Define $\pi:Y^{\varphi}\to M$ as $\pi(y,u)=T_uy$.

\begin{prop} \label{prop:obs}
Suppose that the function $\chi:Y\to\R$ satisfies condition~(H).

Then observables $v\in C^\eta(M)\cap C^{0,\eta}(M)$ lift to observables
$\tilde v=v\circ\pi:Y^\varphi\to \R$ that lie in $\cH^*_{\gamma_2,\eta}(Y^\varphi)$ where $\gamma_2=\gamma^{\eta^2}$ and the metric $d$ on $Y$ is replaced by the metric $d_2(y,y')=d(y,y')^{\eta^2}$.

For $m\ge1$, observables $w\in C^{\eta,m}(M)$ lift to observables
$\tilde w=w\circ\pi\in\cH^*_{\gamma_2,0,m}(Y^\varphi)$.

Moreover, there is a constant $C>0$ such that
$\|\tilde v\|^*_{\gamma_2,\eta}\le C({\|v\|}_{C^\eta}+{\|v\|}_{C^{0,\eta}})$ and
$\|\tilde w\|^*_{\gamma_2,0,m}\le C{\|w\|}_{C^{\eta,m}}$.
\end{prop}

\begin{proof}
Let $\sigma=2|\chi|_\infty$.  We show that
$\|\tilde v\circ F_\sigma\|_{\gamma_2,\eta}\ll {\|v\|}_{C^\eta}+{\|v\|}_{C^{0,\eta}}$.
The same calculation with $\sigma=0$ shows that
$\|\tilde v\|_{\gamma_2,\eta}\ll {\|v\|}_{C^\eta}+{\|v\|}_{C^{0,\eta}}$,
so
$\|\tilde v\|^*_{\gamma_2,\eta}\ll {\|v\|}_{C^\eta}+{\|v\|}_{C^{0,\eta}}$,
We take $D_1=|h|_\infty+2|\chi|_\infty$ with corresponding value of $D_2$ in~\eqref{eq:Tt}

Let $(y,u),\,(y',u)\in Y^\varphi$ with $y,y'\in Y_j$ for some $j\ge1$.
There exists $\ell,\ell'\in\{0,\dots,\tau(y)-1\}$ such that
\[
u\in[h_\ell(y),h_{\ell+1}(y)]\cap [h_{\ell'}(y'),h_{\ell'+1}(y')].
\]
Suppose without loss that $\ell\le\ell'$.
Then
\[
u=h_\ell(y)+r=h_\ell(y')+r',
\]
where $r\in[0,|h|_\infty]$ and
 $r'= u-h_\ell(y')\ge u- h_{\ell'}(y')\ge0$.
Note that $T_uy=T_rT_{h_\ell(y)}y=T_rf^\ell y$.
Hence
$\tilde v(y,u)= v(T_rf^\ell y)$ and so
$\tilde v\circ F_\sigma(y,u)=v(T_{\sigma+r}f^\ell y)$.
Similarly, $T_uy'=T_{r'}f^\ell y'$ and
$\tilde v\circ F_\sigma(y',u)=v(T_{\sigma+r'}f^\ell y')$.
Also, $\sigma+r\in[0,D_1]$.
By~\eqref{eq:Tt} and~\eqref{eq:prodl},
\begin{align*}
 |v(T_{\sigma+r}f^\ell y)-v(T_{\sigma+r}f^\ell y')|
& \le {|v|}_{C^\eta} d(T_{\sigma+r}f^\ell y,T_{\sigma+r}f^\ell y')^\eta
 \le D_2^\eta {|v|}_{C^\eta} d(f^\ell y,f^\ell y')^{\eta^2}
\\ & \ll {|v|}_{C^\eta} (d_2(y,y')+\gamma_2^{s(y,y')}).
\end{align*}
Since $u\ge h_\ell(y')\ge\ell\inf h$, it follows from
Proposition~\ref{prop:hell} that
\begin{align*}
 |v(T_{\sigma+r}f^\ell y')- & v(T_{\sigma+r'}f^\ell y')|
 \le {|v|}_{C^{0,\eta}}|r-r'|^\eta
 = {|v|}_{C^{0,\eta}}|h_\ell(y)-h_\ell(y')|^\eta
 \\ & \ll {|v|}_{C^{0,\eta}}\,\ell(d_2(y,y')^\eta+\gamma_2^{s(y,y')})
  \le (\inf h)^{-1}{|v|}_{C^{0,\eta}}\,u(d_2(y,y')^\eta+\gamma_2^{s(y,y')}).
\end{align*}
Hence
\begin{align*}
|\tilde v\circ F_\sigma(y,u)-\tilde v\circ F_\sigma(y',u)|
& = |v(T_{\sigma+r}f^\ell y)- v(T_{\sigma+r'}f^\ell y')|
\\ & \ll  ({|v|}_{C^\eta}+{|v|}_{C^{0,\eta}})
(u+1)(d_2(y,y')+\gamma_2^{s(y,y')})
\end{align*}
whenever $s(y,y')\ge1$.
For $s(y,y')=0$, we have the estimate
$|\tilde v\circ F_\sigma(y,u)-\tilde v\circ F_\sigma(y',u)|\le 2|v|_\infty=2|v|_\infty\gamma_2^{s(y,y')}
\ll |v|_\infty \,\varphi(y)(d_2(y,y')+\gamma_2^{s(y,y')})$, so in all cases we obtain
\begin{align*}
|\tilde v\circ F_\sigma(y,u)-\tilde v\circ F_\sigma(y',u)| & \ll
({\|v\|}_{C^\eta}+{|v|}_{C^{0,\eta}}) (u+1)(d_2(y,y')+\gamma_2^{s(y,y')})
\\ & \le 2({\|v\|}_{C^\eta}+{|v|}_{C^{0,\eta}}) \varphi(y)(d_2(y,y')+\gamma_2^{s(y,y')}).
\end{align*}

Also,
\[
|\tilde v\circ F_\sigma(y,u)-\tilde v\circ F_\sigma(y,u')|
=|v(T_{\sigma+u}y)-v(T_{\sigma+u'}y)\le {|v|}_{C^{0,\eta}} |u-u'|^\eta,
\]
so $\|\tilde v\circ F_\sigma\|_{\gamma_2.\eta}
\ll {|v|}_{C^0}+{|v|}_{C^{0,\eta}}$ as required.
\end{proof}

\subsection{Dynamically H\"older flows and observables}
\label{sec:dyn}

The H\"older assumptions in Subsection~\ref{sec:holder} can be replaced by dynamically H\"older as follows.  We continue to assume that $\inf h>0$.

\begin{defn}\label{def:dyn}
The roof function $h$, the flow $T_t$ and the observable $v$ are
{\em dynamically H\"older} if
$v\in C^{0,\eta}(M)$ for some $\eta\in(0,1]$ and
there is a constant $C\ge1$ such that for all $y,y'\in Y_j$, $j\ge1$,
\begin{itemize}
\item [(a)]
$|h(f^\ell y)-h(f^\ell y')|\le C(d(y,y')^\eta+\gamma^{s(y,y')})$ for all $0\le\ell\le \tau(y)-1$.
\item[(b)]
For every $u\in[0,\varphi(y)]\cap[0,\varphi(y')]$, there exist $t,t'\in \R$ such that $|t-t'|\le C {(u+1)} (d(y,y')^\eta+\gamma^{s(y,y')})$, and
setting $z=W^s(y)\cap W^u(y')$,
\[
\max\big\{|v(T_uy)-v(T_tz)|\,,\,|v(T_uy')-v(T_{t'}z)|\big\}\le C (u+1) (d(y,y')^\eta+\gamma^{s(y,y')}).
\]
\end{itemize}
\end{defn}

Also, we replace the assumption $w\in C^{\eta,m}(M)$ by the condition
that $\partial_t^kw$ lies in $C^{0,\eta}(M)$ and satisfies (b) for all $k=0,\dots,m$.

\begin{rmk}
In the proof of Proposition~\ref{prop:obs}, we showed that
$|v(T_uy)-v(T_uy')|=|\tilde v(y,u)-\tilde v(y',u)|\ll (u+1)(d(y,y')^\eta+\gamma^{s(y,y')})$
(for modified $d$ and $\gamma$) under the old hypotheses.
Hence,
taking $t=t'=u$, we see that
Definition~\ref{def:dyn} is indeed a relaxed version of the conditions
in Subsection~\ref{sec:holder}.
\end{rmk}

It is easily verified that condition~\eqref{eq:phi} remains valid
under the more relaxed assumption on $h$
in Definition~\ref{def:dyn}(a).
Also, it follows as in the proof of Proposition~\ref{prop:obs} that
$|\tilde v(y,u)-\tilde v(y,u')| \le {|v|}_{C^{0,\eta}} |u-u'|^\eta$.

Next we estimate $|\tilde v(y,u)-\tilde v(y',u)|$
and $|\tilde v\circ F_\sigma(y,u)-\tilde v\circ F_\sigma(y',u)|$
for $(y,u),\,(y',u)\in Y^{\varphi}$,
where $\sigma=2|\chi|_\infty$.
If $s(y,y')=0$, then
$|\tilde v(y,u)-\tilde v(y',u)|,\,
|\tilde v\circ F_\sigma(y,u)-\tilde v\circ F_\sigma(y',u)|
\ll |v|_\infty\,\varphi(y)(d_2(y,y)+\gamma_2^{s(y,y')})$ as
in the proof of Proposition~\ref{prop:obs}.
Hence we can suppose that $y,y'\in Y_j$ for some $j\ge1$.  Set $z=W^s(y)\cap W^u(y')$ and choose $t,t'$ as in Definition~\ref{def:dyn}(b).
Then
\begin{align*}
|\tilde{v}(y,u)- & \tilde{v}(y',u)| =|v(T_uy)-v(T_uy')| \\
&\le |v(T_uy)-v(T_{t}z)| + |v(T_{t'}z)-v(T_uy')|+ |v(T_{t}z)-v(T_{t'}z)|
\\
& \le 4C\varphi(y)(d(y,y')^\eta+\gamma^{s(y,y')})+{|v|}_{C^{0,\eta}}|t-t'|^\eta
\ll \varphi(y)(d_2(y,y')+\gamma_2^{s(y,y')}).
\end{align*}
Hence
$|\tilde v(y,u)-\tilde v(y',u)| \ll
\varphi(y)(d_2(y,y')^\eta+\gamma_2^{s(y,y')})$ for all
$(y,u),\,(y',u)\in Y^{\varphi}$, and so
$\tilde v\in \cH_{\gamma_2,\eta}(Y^\varphi)$.

To proceed, we recall that $z=W^s(y)\cap W^u(y')$, so $Fz=W^s(Fy)\cap W^u(Fy')$
by~\eqref{eq:occas}.
Hence
\begin{equation}\label{eq:whyFdoesnotmatter}
d(Fy,Fy')\le d(Fy,Fz)+d(Fz,Fy')\ll d(y,z) +\gamma^{s(y,y')}\le C_4 d(y,y') + \gamma^{s(y,y')}.
\end{equation}

To control $\tilde v \circ F_\sigma (y,u) - \tilde v\circ F_\sigma (y',u)$,
we assume without loss that $\varphi(y)\geq \varphi(y')$, and distinguish three cases.

If $u+\sigma<\varphi(y')$, we argue as in the bound for  $\tilde{v}(y,u)- \tilde{v}(y',u)$.

If $u+\sigma\geq \varphi(y)$, then there exists $0\leq \bar{u}\leq \sigma$
and $\bar{u}'\ge \bar{u}$ such that $T_{u+\sigma} y=T_{\bar{u}} Fy$ and
$T_{u+\sigma} y'=T_{\bar{u}'} Fy'$.
By Corollary~\ref{cor:phionY} and~\eqref{eq:whyFdoesnotmatter},
\[
|\bar{u} - \bar{u}'| = |\varphi(y)-\varphi(y')|\ll \varphi(y) (d(y,y') + \gamma^{s(y,y')})
\]
and so
\[
|v(T_{\bar{u}} Fy')- v(T_{\bar{u}'} Fy')|\ll \varphi(y) (d_2(y,y') + \gamma_2^{s(y,y')}).
\]
On the other hand, choosing  $\bar{t}$ and $\bar{t}'$ for $\bar{u}$ as in Definition~\ref{def:dyn}(b), we get
\begin{align*}
|v(T_{\bar{u}}Fy) & -v(T_{\bar{u}}Fy')| \\
&\le |v(T_{\bar{u}}Fy)-v(T_{\bar{t}}Fz)| + |v(T_{\bar{t}'}Fz)-v(T_{\bar{u}}Fy')|+ |v(T_{\bar{t}}Fz)-v(T_{\bar{t}'}Fz)|
\\
& \le 2C (\bar{u}+1)(d(Fy,Fy')^\eta+\gamma^{s(Fy,Fy')})+{|v|}_{C^{0,\eta}}|\bar{t}-\bar{t}'|^\eta
\ll d_2(y,y')+\gamma_2^{s(y,y')}
\end{align*}
where we have used \eqref{eq:whyFdoesnotmatter} and $\bar{u}\le\sigma$.
Hence
\begin{align*}
|\tilde v\circ  F_\sigma (y,u) - \tilde v \circ F_\sigma (y',u)|
& \le |v(T_{\bar{u}}Fy)-v(T_{\bar{u}}Fy')| + |v(T_{\bar{u}} Fy')- v(T_{\bar{u}'} Fy')|
\\ & \ll \varphi(y)(d_2(y,y')+\gamma_2^{s(y,y')}).
\end{align*}

Finally, if $\varphi(y')\leq u+\sigma < \varphi(y)$, there exist $0< u_1,\,u_2 \leq \varphi(y)-\varphi(y')$ such that
$Fy=T_{u_1} T_{u+\sigma}y$ and $T_{u+\sigma} y'=T_{u_2} Fy'$. Using again
Corollary~\ref{cor:phionY} and~\eqref{eq:whyFdoesnotmatter},
\begin{align*}
|\tilde{v}\circ F_\sigma(y,u)- & \tilde{v}\circ F_\sigma(y',u)| =|v(T_{u+\sigma}y)-v(T_{u+\sigma}y')| \\
& \leq |v(T_{u+\sigma}y)-v(Fy)| +|v(Fy)-v(Fy')| + |v(Fy')-v(T_{u+\sigma}y')| \\
& = |v(T_{u+\sigma}y)-v(T_{u_1+u+\sigma}y)| +|v(Fy)-v(Fy')| + |v(Fy')-v(T_{u_2}Fy')| \\
& \ll \varphi(y) (d_2(y,y') + \gamma_2^{s(y,y')}).
\end{align*}

This completes the verification that $\tilde v \in \cH^*_{\gamma_2,\eta}(Y^\varphi)$.
A similar argument shows that $\tilde w\in\cH^*_{\gamma_2,0,m}(Y^\varphi)$, completing the verification that
Proposition~\ref{prop:obs} holds under the modified assumptions.

\section{Condition~(H) for nonuniformly hyperbolic flows}
\label{sec:chi}

In this section, we consider various classes of nonuniformly hyperbolic flows
for which condition~(H) in Section~\ref{sec:nonskew} can be satisfied.
We are then able to apply Theorem~\ref{thm:nonskew} to obtain results that superpolynomial and polynomial mixing applies to such flows as follows:

\begin{cor} \label{cor:nonskew}
Let $T_t:M\to M$ be a nonuniformly hyperbolic flow as in Section~\ref{sec:holder} and
assume that condition~(H) is satisfied.  Then

\vspace{1ex}\noindent
(a)  $F_t:Y^\varphi\to Y^\varphi$ is a Gibbs-Markov flow.

\vspace{1ex}\noindent
(b)
Suppose that $\mu(\varphi>t)=O(t^{-\beta})$ for some $\beta>1$ and
assume absence of approximate eigenfunctions for $F_t$.  Then there exists $m\ge1$ and $C>0$ such that
\[
|\rho_{v,w}(t)|\le C({\|v\|}_{C^\eta}+{\|v\|}_{C^{0,\eta}}){\|w\|}_{C^{\eta,m}} \,t^{-(\beta-1)},
\]
for all $v\in C^\eta(M)\cap C^{0,\eta}(M)$, $w\in C^{\eta,m}(M)$, $t>1$.
\end{cor}

\begin{proof}
Part~(a) follows from the discussion in Section~\ref{sec:holder} (so ingredient~(i) is automatic and ingredient~(ii) is now assumed).

As described in Section~\ref{sec:H}, there is a measure-preserving conjugacy from $F_t$ to $T_t$, so part~(b) is immediate from Theorem~\ref{thm:nonskew} combined with
Proposition~\ref{prop:obs}.
\end{proof}

The analogous result holds for nonuniformly hyperbolic flows and observables satisfying the dynamically H\"older conditions in Section~\ref{sec:dyn}.

We verify condition~(H) for three classes of flows.
In Subsection~\ref{sec:bounded}, we consider roof functions with bounded H\"older constants.
In Subsection~\ref{sec:exp}, we consider flows for which there is exponential contraction along stable leaves.
In Subsection~\ref{sec:Wss}, we consider flows with an invariant H\"older stable foliation.
These correspond to the situations mentioned
in~\cite[Section 4.2]{M18}.

Also, in Subsection~\ref{sec:D}, we briefly review the temporal distance function and a criterion for absence of approximate eigenfunctions.

\subsection{Roof functions with bounded H\"older constants}
\label{sec:bounded}

We assume a ``bounded H\"older constants" condition
on $\varphi$, namely that for all $y,y'\in Y$,
\begin{alignat}{2} \label{eq:Wsbounded}
|\varphi(y)-\varphi(y')| & \le C_1d(y,y') && \quad\text{for all $y'\in W^s(y)$,}\\
\label{eq:Wubounded}
|\varphi(y)-\varphi(y')| & \le C_1\gamma^{s(y,y')} && \quad\text{for all $y'\in W^u(y)$, $s(y,y')\ge1$.}
\end{alignat}
This leads directly to an enhanced version of~\eqref{eq:phi}:

\begin{prop} \label{prop:prodbounded}
$|\varphi(y)-\varphi(y')|\le C_1C_4(d(y,y')+\gamma^{s(y,y')})$
for all $y,y'\in Y$, $s(y,y')\ge1$.
\end{prop}

\begin{proof}
Let $z=W^s(y)\cap W^s(y')$.  Then
\begin{align*}
|\varphi(y)-\varphi(y')|
& \le |\varphi(y)-\varphi(z)|+ |\varphi(z)-\varphi(y')|
\\ & \le C_1(d(y,z)+\gamma^{s(z,y')})\le C_1C_4(d(y,y')+\gamma^{s(y,y')}),
\end{align*}
as required.
\end{proof}

\begin{lemma} \label{lem:bounded}
If conditions~\eqref{eq:Wsbounded} and
~\eqref{eq:Wubounded} are satisfied,
then condition~(H) holds.
\end{lemma}

\begin{proof}
By~\eqref{eq:Ws} and~\eqref{eq:Wsbounded},
for all $y\in Y$, $n\ge0$,
\[
|\varphi(F^n\pi y)-\varphi(F^ny)|
\le C_1d(F^n\pi y,F^ny)\le C_1C_2\gamma^n d(\pi y,y)\le C_1C_2\gamma^n.
\]
It follows that
\[
\SMALL |\chi(y)|\le\sum_{n=0}^\infty
|\varphi(F^n\pi y)-\varphi(F^ny)|
\le C_1C_2(1-\gamma)^{-1}.
\]
Hence $|\chi|_\infty\le C_1C_2(1-\gamma)^{-1}$ and condition~(H)(a) is satisfied.

Next, let $y,y'\in Y$, and set $N=[\frac12 s(y,y')]$, $\gamma_1=\gamma^{1/2}$.
Write
\[
\chi(y)-\chi(y')=A(\pi y,\pi y')-A(y,y')
+B(y) -B(y'),
\]
where
\[
A(p,q)=\sum_{n=0}^{N-1} (\varphi(F^np)-\varphi(F^nq)), \qquad
B(p)=\sum_{n=N}^\infty (\varphi(F^n\pi p)-\varphi(F^np)).
\]

By the calculation for $|\chi|_\infty$, we obtain
$|B(p)|\le C_1C_2 (1-\gamma)^{-1}\gamma^N$ for all $p\in Y$.   Also,
$\gamma^N\le \gamma^{-1}\gamma^{\frac12 s(y,y')}=\gamma^{-1}\gamma_1^{s(y,y')}$,
so $B(p)=O(\gamma_1^{s(y,y')})$ for $p=y,y'$.

For $n\le N-1$ we have $s(F^ny,F^ny')\ge1$,
so by Propositions~\ref{prop:prod} and~\ref{prop:prodbounded},
\begin{align*}
|\varphi(F^ny)-\varphi(F^ny')|
& \le C_1C_4 (d(F^ny,F^ny')+\gamma^{s(y,y')-n})
 \le C(\gamma^nd(y,y')+\gamma^{s(y,y')-n}),
\end{align*}
where $C=2C_4^2C_1C_2$.  Hence
\begin{align*}
|A(y,y')| & \le \sum_{n=0}^{N-1} |\varphi(F^ny)-\varphi(F^ny')|
\le C \sum_{n=0}^{N-1} (\gamma^nd(y,y')+\gamma^{s(y,y')-n})
\\ & \le C(1-\gamma)^{-1}(d(y,y')+\gamma^{s(y,y')-N})
\le C(1-\gamma)^{-1}(d(y,y')+\gamma_1^{s(y,y')}).
\end{align*}
Similarly for $A(\pi y,\pi y')$.
Hence $|\chi(y)-\chi(y')| \ll d(y,y')+\gamma_1^{s(y,y')}$,
so (H)(b) holds.
\end{proof}

\subsection{Exponential contraction along stable leaves}
\label{sec:exp}

In this subsection, we suppose that $h\in C^\eta(X)$ and that
$f$ is exponentially contracting along stable leaves:
\begin{equation} \label{eq:exp}
d(f^ny,f^ny')  \le C_2\gamma^n d(y,y') \quad\text{for all $n\ge0$ and all $y,y'\in Y$ with
$y'\in W^s(y)$} .
\end{equation}
Note that this strengthens condition~\eqref{eq:Ws}.
Proposition~\ref{prop:prod} becomes
\begin{equation} \label{eq:prodexp}
d(f^ny,f^ny') \le C_2C_4(\gamma^n d(y,y')+\gamma^{s(y,y')-\psi_n(y)})
\quad\text{for all $n\ge0$, $\,y,y'\in Y$.}
\end{equation}

\begin{lemma} \label{lem:exp}
If condition~\eqref{eq:exp} is satisfied, then
condition~(H) holds.
\end{lemma}

\begin{proof}
Let
$\gamma_1=\gamma^\eta$, $\gamma_2=\gamma_1^{1/2}$.
We verify condition~(H) with
$\gamma_2$ and $d_1(y,y')=d(y,y')^\eta$,
using the equivalent definition
for $\chi$,
\[
\SMALL \chi(y)=\sum_{n=0}^\infty (h(f^n\pi y)-h(f^ny)).
\]

By~\eqref{eq:exp},
\[
\SMALL |\chi(y)|\le\sum_{n=0}^\infty
|h|_\eta d(f^n\pi y,f^ny)^\eta\le C_2|h|_\eta\sum_{n=0}^\infty \gamma_1^n d_1(\pi y,y)
\le C_2|h|_\eta(1-\gamma_1)^{-1}.
\]
Hence $|\chi|_\infty\le C_2|h|_\eta(1-\gamma_1)^{-1}$ and condition~(H)(a)
is satisfied.

Next, let $y,y'\in Y$ and set $N=[\frac12 s(y,y')]$.  Write
$\chi(y)-\chi(y')=A(\pi y,\pi y')-A(y,y') +B(y) -B(y')$,
where
\[
A(p,q)=\sum_{n=0}^{N-1} (h(f^np)-h(f^nq)), \qquad
B(p)=\sum_{n=N}^\infty (h(f^n\pi p)-h(f^np)).
\]
By the calculation for $|\chi|_\infty$, we obtain
$|B(p)|\le C_2|h|_\eta (1-\gamma_1)^{-1}\gamma_1^N$ for all $p\in Y$.  Also,
$\gamma_1^N\le \gamma_1^{-1}\gamma_1^{\frac12 s(y,y')}=\gamma_1^{-1}\gamma_2^{s(y,y')}$, so $B(p)=O(\gamma_2^{s(y,y')})$ for $p=y,\,y'$.

Finally, by~\eqref{eq:prodexp} using that $\psi_n\le n$,
\begin{align*}
|A(y,y')| & \le |h|_\eta\sum_{n=0}^{N-1}d(f^ny,f^ny')^\eta
\le C_2C_4|h|_\eta \sum_{n=0}^{N-1}(\gamma_1^nd_1(y,y') + \gamma_1^{s(y,y')-n})
\\ &
\le C_2C_4|h|_\eta(1-\gamma_1)^{-1} (d_1(y,y') + \gamma_1^{s(y,y')-N})
\\ & \le C_2C_4|h|_\eta(1-\gamma_1)^{-1} (d_1(y,y') + \gamma_2^{s(y,y')}).
\end{align*}
Similarly for $A(\pi y,\pi y')$.
Hence $|\chi(y)-\chi(y')|\ll d_1(y,y')+\gamma_2^{s(y,y')}$,
so (H)(b) holds.
\end{proof}

\begin{rmk} \label{rmk:exp}
In cases where $h$ lies in $C^\eta(X)$ and the dynamics on $X$ is modelled by a Young tower with exponential tails
(so $\mu_X(\tau>n)=O(e^{-ct})$ for some $c>0$), it is immediate that
$\varphi\in L^q(Y)$ for all $q$ and that condition~\eqref{eq:exp} is satisfied.
Assuming absence of approximate eigenfunctions, we obtain rapid mixing for such flows.
\end{rmk}

\subsection{Flows with an invariant H\"older stable foliation}
\label{sec:Wss}

Let $T_t:M\to M$ be a H\"older nonuniformly hyperbolic flow as in Section~\ref{sec:holder}.
For simplicity, we suppose that $(M,d)$ is a Riemannian manifold and that $Y$ is a smoothly embedded cross-section for the flow.
We assume that
the flow possesses a $T_t$-invariant H\"older stable foliation $\cW^{ss}$ in a neighbourhood of $\Lambda$.  (A sufficient condition for this to hold is
that $\Lambda$ is a partially hyperbolic attracting set  with a $DT_t$-invariant dominated splitting $T_\Lambda M=E^{ss}\oplus E^{cu}$,
see~\cite{AraujoM17}.)
We also assume that $\diam Y$ can be chosen arbitrarily small.
In this subsection, we show how to use the
stable foliation $\cW^{ss}$ for the flow to show that $\chi$ is H\"older, hence verifying the hypotheses in Section~\ref{sec:H}.

\begin{rmk} \label{rmk:Wss}
As discussed in~\cite[Section~4.2(iii)]{M18}, this framework includes (not necessarily Markovian) intermittent solenoidal flows, and yields polynomial decay $O(t^{-(\beta-1)})$ for any prescribed $\beta>1$.  These results are optimal by~\cite{MT17} in the Markovian case and by~\cite{BMTprep} in general.
\end{rmk}

First, we show that if $W^s(y)$ and $W^{ss}(y)$ coincide for all $y\in Y$, then
$F_t:Y^\varphi\to Y^\varphi$ is already a skew product (so $\chi=0$).

\begin{prop} \label{prop:s=ss}
Suppose that $W^s(y)$ and $W^{ss}(y)$ coincide for all $y\in Y$.
Then $\varphi$ is constant along stable leaves $W^s(y)$, $y\in Y$.
\end{prop}

\begin{proof}
For $y_0\in Y$,
\[
\{T_{\varphi(y)}y:y\in W^{ss}(y_0)\}=
\{Fy:y\in W^{s}(y_0)\}=FW^s(y_0)\subset W^s(Fy_0)=W^{ss}(Fy_0).
\]
But setting $t_0=\varphi(y_0)$,
\[
\{T_{t_0}y:y\in W^{ss}(y_0)\}=
T_{t_0}W^{ss}(y_0)\subset W^{ss}(T_{t_0}y_0) =W^{ss}(Fy_0).
\]
Hence $\varphi|_{W^{ss}(y_0)}\equiv \varphi(y_0)$.
\end{proof}

Let $\tY=W^u(y_0)$ for some fixed $y_0\in Y$ and
define the new cross-section to the flow $Y^*=\bigcup_{y\in \tY}W^{ss}(y)$.
Shrinking $Y$ if necessary, there exists a unique continuous function $r:Y\to \R$ with $|r|\le \frac12\inf\varphi$ such that
$r|_{\tY}\equiv0$ and \(\{ T_{r(y)}(y) : y\in Y  \} \subset Y^*\). Moreover, $r$ is H\"older since $Y$ is smoothly embedded in \(M\) and
$Y^*$ is H\"older by the assumption on the regularity of the stable foliation
$\cW^{ss}$.
Define the new roof function
\[
\varphi^*:Y^*\to\R^+, \qquad \varphi^*(T_{r(y)}y)=\varphi(y)+r(Fy)-r(y).
\]
We observe that \(\varphi^*\) is the return time for the flow \(T_t\) to the cross-section \(Y^*\).

\begin{lemma} \label{lem:Wss}
  Under the above assumption on $\cW^{ss}$, condition~(H) holds.
\end{lemma}

\begin{proof}
We show that $\chi=-r$.  The result follows since $r$ is H\"older.

Let $n\ge0$, $y\in Y$.
By Proposition~\ref{prop:s=ss} applied to $\varphi^*:Y^*\to\R^+$,
we have that $\varphi^*(T_{r(F^n\pi y)}F^n\pi y)=\varphi^*(T_{r(F^ny)}F^ny)$.
Hence by definition of $\varphi^*$,
\[
 \varphi(F^n\pi y) - \varphi(F^ny)
 = r(F^n\pi y) -  r(F^n y) +r(F^{n+1}y)-r(F^{n+1}\pi y) .
\]
Let $\eta$ be the H\"older exponent of $r$.
By~\eqref{eq:Ws},
$|\varphi(F^n\pi y) - \varphi(F^ny)|\le 2C_2|r|_\eta (\gamma^{\eta})^n$
so the series
 \(\chi(y) = \sum_{n=0}^\infty  (\varphi(F^n\pi y) - \varphi(F^ny))\) converges
absolutely.
Moreover,
\[
\begin{aligned}
\chi(y) & =\lim_{N\to\infty}\sum_{n=0}^{N-1}  (\varphi(F^n\pi y) - \varphi(F^ny))
\\ & =\lim_{N\to\infty}\big(r(\pi y)-r(y)+r(F^{N}y)-r(F^{N}\pi y)\big)=r(\pi y)-r(y).
\end{aligned}
\]
Finally, \(r(\pi y) = 0\) since $r|_{\tY}\equiv0$.
\end{proof}

\subsection{Temporal distance function}
\label{sec:D}

Dolgopyat~\cite[Appendix]{Dolgopyat98b} showed that for Axiom~A flows a sufficient condition for absence of approximate eigenfunctions is that the range of the temporal distance function has positive lower box dimension.  This was extended to nonuniformly hyperbolic flows in~\cite{M09,M18}.  Here we recall the main definitions and result.

We assume that condition~(H) holds, so that the suspension flow $Y^\varphi\to Y^\varphi$ is a Gibbs-Markov flow (and hence conjugate to a skew product flow).
We also assume
the dynamically H\"older setup from Section~\ref{sec:dyn}.
In particular, the Poincar\'e map $f:X\to X$ is nonuniformly hyperbolic as
in Section~\ref{sec:NUH} and $Y$ has a local product structure.
Also we assume that the roof function $\varphi$ has bounded H\"older constants along unstable leaves, so condition~\eqref{eq:Wubounded} is satisfied.

Let $y_1,y_4\in Y$ and set
$y_2=W^s(y_1)\cap W^u(y_4)$,
$y_3=W^u(y_1)\cap W^s(y_4)$.
Define the {\em temporal distance function} $D:Y\times Y\to\R$,
\[
D(y_1,y_4)   =\sum_{n=-\infty}^\infty \Big(\varphi(F^ny_1)-\varphi(F^ny_2)-\varphi(F^ny_3)+\varphi(F^ny_4)\Big).
\]
It follows from the construction in~\cite[Section~5.3]{M18} (which uses~\eqref{eq:occas} and~\eqref{eq:Wubounded}) that inverse branches $F^ny_i$ for $n\le -1$ can be chosen so that $D$ is well-defined.

\begin{lemma}[ {\cite[Theorem~5.6]{M18}}] \label{lem:D}
Let $Z_0=\bigcap_{n=0}^\infty F^{-n}Z$ where $Z$ is a union of finitely many elements of the partition $\{Y_j\}$.  Let $\bZ_0$ denote the corresponding finite subsystem of $\bY$.  If
the lower box dimension of $D(Z_0\times Z_0)$ is positive,
then there do not exist approximate eigenfunctions on $\bZ_0$.
\qed
\end{lemma}

\begin{rmk} \label{rmk:contact}
For Axiom~A attractors, $Z_0$ can be taken to be connected and $D$ is continuous, so absence of approximate eigenfunctions is ensured whenever
$D$ is not identically zero.
For nonuniformly hyperbolic flows, where the partition $\{Y_j\}$ is countably infinite, $Z_0$ is a Cantor set of positive Hausdorff dimension~\cite[Example~5.7]{M09}.  In general it is not clear how to use this property since $D$ is generally at best H\"older.   However for flows with a contact structure, a formula for $D$ in~\cite[Lemma~3.2]{KatokBurns94} can be exploited and the lower box dimension of $D(Z_0\times Z_0)$ is indeed positive, see~\cite[Example~5.7]{M09}.
The arguments in~\cite[Example~5.7]{M09} apply to general Gibbs-Markov flows with a contact structure.  A special case of this is the Lorentz gas examples considered in Section~\ref{sec:Lorentz}.
\end{rmk}

\section{Billiard flows associated to infinite horizon Lorentz gases}
\label{sec:Lorentz}

In this section we show that billiard flows associated to planar infinite horizon Lorentz gases satisfy the assumptions of Section~\ref{sec:bounded}.
In particular, we prove decay of correlations with decay rate $O(t^{-1})$.

Background material on infinite horizon Lorentz gases is recalled in Subsection~\ref{sec:BackLorentz} and the decay rate $O(t^{-1})$ is proved in Subsection~\ref{sec:tail}.  In Subsection~\ref{sec:Stadia}, we show that the same decay rate holds for semidispersing Lorentz flows and stadia.
In Subsection~\ref{sec:lower}, we show that the decay rate is optimal for the examples considered in this section.

\subsection{Background on the infinite horizon Lorentz gas}
\label{sec:BackLorentz}

We begin by recalling some background on billiard flows; for further details we refer to the monograph \cite{ChernovMarkarian06}.

Let $\T^2$ denote the two dimensional flat torus, and let us fix finitely many disjoint convex scatterers $S_k\subset \T^2$
with $C^3$ boundaries of nonvanishing curvature.
The complement $Q=\T^2 \setminus \bigcup S_k$ is the billiard domain, and the billiard dynamics are that of a point particle that performs uniform motion with unit speed inside $Q$, and specular reflections  --- angle of reflection equals angle of incidence --- off the scatterers, that is, at the
boundary $\partial Q$. The resulting billiard flow is $T_t:M\to M$, where the phase space $M=Q\times \bS^1$ is a Riemannian manifold, and $T_t$ preserves the (normalized) Lebesgue measure $\mu_M$  (often called Liouville measure in the literature).

There is a natural Poincar\'e section $X=\partial Q\times[-\pi/2,\pi/2]\subset M$ corresponding to collisions (with outgoing velocities), which gives rise to the billiard map denoted by $f:X\to X$, with absolute continuous invariant probability measure $\mu_X$. The time until the next collision, the
free flight function $h:X\to \R^+$, is
defined to be $h(x)=\inf\{t>0:T_tx\in X\}$.
The Lorentz gas has {\em finite horizon} if $h\in L^\infty(X)$ and
{\em infinite horizon} if $h$ is unbounded.

In the finite horizon case,~\cite{BaladiDemersLiverani18} recently proved exponential decay of correlations.  In this section, we prove
\begin{thm} \label{thm:Lorentz}
Let $\eta\in(0,1]$.
In the infinite horizon case, there exists $m\ge1$ such that
$\rho_{v,w}(t)=O(t^{-1})$
for all $v\in C^\eta(M)\cap C^{0,\eta}(M)$ and $w\in C^{\eta,m}(M)$ (and more generally for the class of observables defined in Corollary~\ref{cor:dyn} below).
\end{thm}

Let us fix some terminology and notations. The billiard map $f:X\to X$ is discontinuous, with singularity set $\cS$ corresponding to the preimages of grazing collisions.
Here, $\cS$ is the closure of a countable union of smooth curves,
$X\setminus\cS$ consists of countably many connected components $X_m$, $m\ge1$,
and $f|_{X_m}$ is $C^2$.
If $x,x'\in X_m$ for some $m\ge 1$, then, in particular,
$x,x'$ and $fx,fx'$ lie on the same scatterer (even when the configuration is unfolded to the plane). Throughout our exposition, $d(x,x')$ denotes the Euclidean distance of the two points, i.e.~the distance that is generated by the Riemannian metric on $X$ (or $M$).

It follows from geometric considerations in the infinite horizon case
that $\mu_X(h>t)=O(t^{-2})$. Moreover, as the trajectories are straight lines, we have
\begin{align}\label{eq:bilroof}
&|h(x)-h(x')|\le d(x,x') +d(f x,f x')
\quad\text{for all $x,x'\in X_m$, $m\ge 1$; and} \\ \label{eq:bilroofv}
&d(T_tx,T_{t'}x)\le |t-t'| \quad\text{for all $x\in X$ and $t,t'\in [0,h(x))$}.
\end{align}

The billiard maps considered here (both finite and infinite horizon) have uniform contraction and expansion even for $f$.
There exist stable and unstable manifolds of positive length for almost every $x\in X$, which we denote by
$W^s(x)$ and $W^u(x)$ respectively, and there exist constants $C_2\ge1$, $\gamma\in(0,1)$ such that for all $x,x'\in X$, $n\ge0$,
\begin{align} \label{eq:Wsf}
& d(f^nx,f^nx')\le C_2\gamma^n d(x,x')
\quad\text{for $x'\in W^s(x)$.} \\
\label{eq:Wuf}
& d(x,x')\le C_2\gamma^n d(f^nx,f^nx')
\quad\text{for $f^nx'\in W^u(f^nx)$.}
\end{align}
This follows from the uniform hyperbolicity properties of $f$, see in particular \cite[Formula~(4.19)]{ChernovMarkarian06}.

Furthermore, there is a constant $C_5\ge1$ such that for $x,x'\in X$,
\begin{align}\label{eq:bilroof2s}
 d(T_t x,T_t x') & \le C_5 d(x,x')
\quad\text{for
$x'\in W^s(x)$, $t\in[0,h(x)]\cap[0,h(x')]$.}
\\
\label{eq:bilroof2u}
 d(T_{-t} x,T_{-t} x') & \le C_5 d(x,x')
\quad\text{for
$x'\in W^u(x)$, $t\in[0,h(f^{-1}x)]\cap[0,h(f^{-1}x')]$.}
\end{align}
To verify \eqref{eq:bilroof2s}, note that $d(x,x')$ consists of a position and a velocity component. In course of the free flight, the velocities do not change,
while for $x'\in W^s(x)$, the position component can only shrink as stable manifolds correspond to converging wavefronts. A similar argument applies to
\eqref{eq:bilroof2u}.

\begin{rmk}
(a) In the remainder of the section -- and in particular in the proof of Proposition~\ref{prop:dyn} below -- we apply \eqref{eq:bilroof} repeatedly, but always in the case when either $x'\in W^s(x)$, or $fx'\in W^u(fx)$. As all iterates $f^n, n\ge 0$ are smooth on local stable manifolds (while all iterates $f^{-n}, n\ge 0$ are smooth on local unstable manifolds), both of these conditions imply $x,x'\in X_m$ for some $m\ge 1$.
\\[.75ex]
(b)
For larger values of $t$ than those in~\eqref{eq:bilroof2s}, we note that $d(T_t x,T_t x')$ may grow large \emph{temporarily}: it can happen that one of the trajectories has already collided with some scatterer, while the other has not, hence even though the two points are close in position, the velocities differ substantially.
Similar comments apply to~\eqref{eq:bilroof2u}. This phenomenon is the main reason why we require the notion of dynamically H\"older flows $T_t$ in Definition~\ref{def:dyn}.
\end{rmk}

In \cite{Young98}, Young constructs a subset $Y\subset X$ and an induced map $F=f^\tau:Y\to Y$ that possesses the properties discussed in Section~\ref{sec:NUH}
including~\eqref{eq:occas}.
The tails of the return time $\tau:Y\to \Z^+$ are exponential, i.e.\ $\mu(\tau>n)=O(e^{-cn})$ for some $c>0$.
Moreover, the construction can be carried out so that
$\diam Y$ is as small as desired. This is proved in \cite{Young98} for the finite horizon, and in \cite{Chernov99} for the infinite horizon case.
We mention that \eqref{eq:Ws} and \eqref{eq:Wu} follow from \eqref{eq:Wsf}
and \eqref{eq:Wuf}, respectively, while \eqref{eq:YoungProductStructure} holds as the stable and the unstable manifolds are uniformly transversal, see \cite[Formulas (4.13) and (4.21)]{ChernovMarkarian06}.

\begin{prop} \label{prop:hell2}
For all $y,y'\in Y_j$, $j\ge1$, and all $0\le\ell\le \tau(y)-1$,
\[
|h(f^\ell y)-h(f^\ell y')|\le 2C_2^2C_4\gamma^{-1}(\gamma^\ell d(y,y')+\gamma^{\tau(y)-\ell}\gamma^{s(y,y')}).
\]
\end{prop}

\begin{proof}
Let $z=W^s(y)\cap W^u(y')$.  By~\eqref{eq:occas}, $Fz\in W^u(Fy')$.
By~\eqref{eq:Wsf} and~\eqref{eq:Wuf}, for $0\le\ell\le\tau(y)$,
\[
d(f^\ell y,f^\ell y')\le d(f^\ell y,f^\ell z)+d(f^\ell z,f^\ell y')\le C_2(\gamma^\ell d(y,z)+\gamma^{\tau(y)-\ell}d(Fz,Fy')).
\]
Using also~\eqref{eq:YoungProductStructure} and~\eqref{eq:Wu},
\[
d(f^\ell y,f^\ell y')\le C_2(\gamma^\ell C_4 d(y,y')+\gamma^{\tau(y)-\ell}C_2\gamma^{s(y,y')-1}).
\]
Hence by~\eqref{eq:bilroof}, for $\ell\le\tau(y)-1$,
\[
|h(f^\ell y)-h(f^\ell y')|\le
d(f^\ell y,f^\ell y')+d(f^{\ell+1}y,f^{\ell+1}y')
\ll \gamma^\ell d(y,y')+\gamma^{\tau(y)-\ell}\gamma^{s(y,y')},
\]
as required.
\end{proof}

Define the induced roof function $\varphi=\sum_{\ell=0}^{\tau-1} h\circ f^\ell$.
Using~\eqref{eq:Wu}, it is immediate from Proposition~\ref{prop:hell2}
that $\varphi$ has bounded H\"older constants in the sense of Section~\ref{sec:bounded}:

\begin{cor} \label{cor:Lorentz}
Conditions~\eqref{eq:Wsbounded} and~\eqref{eq:Wubounded} hold.
\end{cor}

\begin{proof}
If $y'\in W^s(y)$, then $s(y,y')=\infty$ so
$|\varphi(y)-\varphi(y')|\ll d(y,y')$ by Proposition~\ref{prop:hell2}.
If $y'\in W^u(y)$, then $d(y,y')\le C_2\gamma^{s(y,y')}$ by~\eqref{eq:Wu}, so
$|\varphi(y)-\varphi(y')|\ll \gamma^{s(y,y')}$ by Proposition~\ref{prop:hell2}.~
\end{proof}

\begin{prop}\label{prop:dyn}
For $\diam Y$ sufficiently small, there exist an integer $n_0\ge 1$ and a constant $C>0$ such that for all $y,y'\in Y$, $s(y,y')\ge n_0$,
and all $u\in[0,\varphi(y)]\cap[0,\varphi(y')]$, there exist $t,t' \in \R$ such that
\begin{alignat*}{2}
 |t-u| & \le Cd(y,y'), & \qquad d(T_uy, T_{t}z) & \le Cd(y,y'),
 \\
|t'-u| & \le  C\gamma^{s(y,y')}, & \qquad d(T_uy', T_{t'}z) & \le C\gamma^{s(y,y')},
\end{alignat*}
where $z=W^s(y)\cap W^u(y')$.
\end{prop}

\begin{proof}
Define $h_\ell(y)=\sum_{j=0}^{\ell-1}h(f^jy)$ for $y\in Y$, $0\le\ell\le\tau(y)$.
By Proposition~\ref{prop:hell2}, there is a constant $C>0$ such that
\begin{equation} \label{eq:hell}
|h_\ell(y)-h_\ell(y')|\le \sum_{j=0}^{\tau(y)-1}|h(f^jy)-h(f^jy')|\le C(d(y,y')+\gamma^{s(y,y')}),
\end{equation}
for all $y,y'\in Y_j$, $j\ge1$ (which is equivalent to $s(y,y')\ge 1$) and all $0\le \ell\le\tau(y)$.

Now consider $y,y'\in Y$ with $s(y,y')\ge n_0$, and
$u\in[0,\varphi(y)]\cap[0,\varphi(y')]$.
Let $z=W^s(y)\cap W^u(y')$.

\noindent{\bf Choosing $t$.}
By~\eqref{eq:YoungProductStructure}, $d(y,z)\le C_4d(y,y')$.
Also, $s(y,z)=\infty$.
We can shrink $Y$ if necessary so that $C\diam Y\le\inf h$.

Write $T_uy=T_r f^\ell y$ where
$0\le\ell\le\tau(y)-1$ and $r\in[0,h(f^\ell y))$. (When $u=\varphi(y)$, we take $\ell=\tau(y)-1$, $r=h(f^\ell y)$.)
Similarly, write $T_uz= T_{r'}f^{\ell'} z$.
Note that $u=h_\ell(y)+r=h_{\ell'}(z)+r'$.

First we show that $|\ell-\ell'|\le1$.
By~\eqref{eq:hell},
\begin{align*}
(\ell-\ell'-1)\inf h & \le
h_\ell(z)-h_{\ell'+1}(z)  \le
 h_\ell(y)-h_{\ell'+1}(z)+h_\ell(z)-h_\ell(y)
 \\ & \le h_\ell(y)-h_{\ell'}(z)-h(f^{\ell'}z)+C\diam Y
 \\ & =r'-r-h(f^{\ell'}z)+C\diam Y\le C\diam Y\le
\inf h.
\end{align*}
Hence $\ell\le\ell'+1$.
Similarly,
$(\ell'-\ell-1)\inf h  \le
h_{\ell'}(y)-h_{\ell+1}(y)  \le \inf h$,
so $|\ell-\ell'|\le1$.

If $\ell=\ell'$, then we take $t=u$. By~\eqref{eq:hell},
\[
|r-r'|=|h_\ell(y)-h_\ell(z)|\le Cd(y,z)\le CC_4d(y,y').
\]
By~\eqref{eq:Wsf},
$d(f^\ell y,f^\ell z)\le C_2d(y,z)\le C_2C_4d(y,y')$.
Without loss, $r\le r'$, so
by~\eqref{eq:bilroofv} and~\eqref{eq:bilroof2s}
\begin{align*}
d(T_uy,T_tz)  =d(T_rf^\ell y,T_{r'}f^\ell z)
& \le  d(T_rf^\ell y,T_rf^\ell z) + d(T_rf^\ell z,T_{r'}f^\ell z)
\\ &
 \le C_5 d(f^\ell y,f^\ell z)+|r-r'| \ll d(y,y').
\end{align*}

If $\ell'=\ell-1$, then we take $t=u+r+s$ where $s=h(f^{\ell-1}z)-r'\ge0$.
Then
$T_uy=T_r f^\ell y$ and
$T_tz=T_{r+s}T_{r'}f^{\ell-1}z=T_{r+h(f^{\ell-1}z)}f^{\ell-1}z=T_rf^\ell z$.

Note that $u=h_\ell (y)+r=h_\ell (z)-s$, hence
$r+s=h_\ell (z)-h_\ell (y)\le Cd(y,z)$ by~\eqref{eq:hell}.
In particular, $|t-u|=r+s\le CC_4d(y,y')$.
Also $0\le r\le r+s\le C\diam Y\le \inf h$.
Hence by~\eqref{eq:Wsf} and~\eqref{eq:bilroof2s},
\[
d(T_uy,T_{t}z)=d(T_r f^{\ell} y,T_r f^{\ell} z)\le C_5d(f^{\ell} y,f^{\ell} z)\le
C_2C_5 d(y,z)\le C_2C_4C_5d(y,y').
\]
The argument for $\ell'=\ell+1$ is analogous.

\noindent{\bf Choosing $t'$.}
This goes along similar lines.
We can shrink $\diam Y$ and increase $n_0$ so  that $C(C_2+1)(\diam Y+\gamma^{n_0})\le\inf h$.
Note that $s(z,y')=s(y,y')\ge n_0\ge1$.

Since $s(z,y')\ge1$, it follows from~\eqref{eq:occas} that $Fz\in W^u(Fy')$.
Write $T_uz=T_{-r}f^{-\ell} Fz$ where $0\le \ell\le \tau(y)-1$ and
$r\in [0,h(f^{-(\ell+1)}Fz))$.  Similarly write
$T_uy'=T_{-r'}f^{-\ell'}Fy'$.
Note that $u=h_{\tau(y)-\ell}(z)-r
=h_{\tau(y)-\ell'}(y')-r'$.

Again, we show that $|\ell-\ell'|\le1$.
By~\eqref{eq:hell},
\begin{align*}
(\ell-\ell'-1)\inf h & \le
h_{\tau(y)-\ell'-1}(y')-h_{\tau(y)-\ell}(y')
 \\ & \le h_{\tau(y)-\ell'-1}(y')-h_{\tau(y)-\ell}(z)+C(\diam Y+\gamma^{n_0})
 \\ & =r'-r-h(f^{\tau(y)-\ell'-1}y')+C(\diam Y+\gamma^{n_0})
 \le C(\diam Y+\gamma^{n_0})\le \inf h.
\end{align*}
Hence $\ell\le\ell'+1$.
Similarly,
$(\ell'-\ell-1)\inf h  \le
h_{\tau(y)-\ell-1}(z)-h_{\tau(y)-\ell'}(z)  \le\inf h$ so
$|\ell-\ell'|\le1$.

If $\ell=\ell'$, then we take $t'=u$.
 It follows from~\eqref{eq:Wu} and~\eqref{eq:hell} that
\[
|r-r'|= |h_{\tau(y)-\ell}(y')-h_{\tau(y)-\ell}(z)|
\le C(d(y',z)+\gamma^{s(y',z)})
\le C(C_2+1)\gamma^{s(y',z)}.
\]
Also, by~\eqref{eq:Wu} and~\eqref{eq:Wuf},
\[
d(f^{-\ell}Fy',f^{-\ell}Fz)\le C_2 d(Fy',Fz)\le C_2^2\gamma^{-1} \gamma^{s(y',z)}.
\]
Without loss, $r'\le r$, so by~\eqref{eq:Wu},~\eqref{eq:bilroofv} and~\eqref{eq:bilroof2u},
\begin{align*}
d(T_uy',T_uz) & =d(T_{-r'}f^{-\ell}Fy',T_{-r}f^{-\ell}Fz)
\\ & \le d(T_{-r'}f^{-\ell}Fy',T_{-r'}f^{-\ell}Fz)+
d(T_{-r'}f^{-\ell}Fz,T_{-r}f^{-\ell}Fz)
\\ & \le C_5 d(f^{-\ell}Fy',f^{-\ell}Fz)+|r-r'|\ll  \gamma^{s(y',z)}=\gamma^{s(y,y')}.
\end{align*}

If  $\ell=\ell'-1$, then we take $t'=u-r'-s$ where $s=h(f^{-(\ell-1)}Fz)-r\ge0$.
Then $T_uy'=T_{-r'} f^{-\ell'}Fy'$ and $T_{t'}z=T_{-r'-s}T_uz=T_{-r'} f^{-\ell'}Fz$.

Note that $u=h_{\tau(y)-\ell'}(y')-r'=h_{\tau(y)-\ell'}(z)+s$, hence
$r'+s=h_{\tau(y)-\ell'}(y')-h_{\tau(y)-\ell'}(z)\le C(C_2+1)\gamma^{s(y,y')}$
by~\eqref{eq:hell}.
In particular, $|t'-u|=r'+s\ll \gamma^{s(y,y')}$.
Also, $0\le r'\le r'+s\le C(C_2+1)\gamma^{n_0}\le \inf h$.
Hence by~\eqref{eq:Wu},~\eqref{eq:Wuf} and~\eqref{eq:bilroof2u},
\[
d(T_uy',T_{t'}z)=d(T_{-r'} f^{-\ell'}Fy',T_{-r'} f^{-\ell'}Fz)\le C_5d(f^{-\ell'}Fy',f^{-\ell'}Fz)\ll \gamma^{s(y,y')}.
\]
The argument for $\ell=\ell'+1$ is analogous.
\end{proof}

\begin{cor}  \label{cor:dyn}
Let $v\in C^{0,\eta}(M)$, $w\in C^{0,m}(M)$ such that $\partial_t^kw\in C^{0,\eta}(M)$, for all $k=0,\dots,m$.
Suppose also that there is a constant $C>0$ such that
$|v(x)-v(x')|\le Cd(x,x')^\eta$ and
$|\partial_t^kw(x)-\partial_t^kw(x')|\le Cd(x,x')^\eta$ for all $x,x'\in M$ of the form
$x=T_uy$, $x'=T_uy'$ where
$y,y'\in Y_j$ for some $j\ge1$, $u\in[0,\varphi(y)]$, $u'\in[0,\varphi(y')]$,
and for all $k=0,\dots,m$.
Then $h$, $T_t$, $v$ and $w$ are dynamically H\"older in the sense of Definition~\ref{def:dyn}.
\end{cor}

\begin{proof}
Condition (a) of Definition~\ref{def:dyn} follows from Proposition~\ref{prop:hell2}.  To check condition~(b), we distinguish two cases. If $s(y,y')< n_0$, we may take $t=t'=u$
and use that $|v(x)-v(x')|\le 2|v|_{\infty}\ll \gamma^{n_0}$ for any $x,x'\in M$. If $s(y,y')\ge n_0$, Proposition~\ref{prop:dyn} applies and, along with Formulas~\eqref{eq:bilroof}--\eqref{eq:bilroof2u}, implies Definition~\ref{def:dyn}(b) .
\end{proof}

\subsection{Tail estimate for $\varphi$ and completion of the proof of Theorem~\ref{thm:Lorentz}}
\label{sec:tail}

Since
\begin{align} \label{eq:htail}
\mu_X(x\in X:h(x)>t) & =O(t^{-2}) \\
\label{eq:tautail}
\mu(y\in Y:\tau(y)>n) & =O(e^{-cn})\quad\mbox{for some $c>0$},
\end{align}
a standard argument shows that $\mu(\varphi>t)=O((\log t)^2t^{-2})$.
In fact, we have

\begin{prop}   \label{prop:gas}
$\mu(\varphi>t)=O(t^{-2})$.
\end{prop}

The crucial ingredient for proving Proposition~\ref{prop:gas} is due to
Sz\'asz \& Varj\'u~\cite{SzaszVarju07}.

\begin{lemma}[ {\cite[Lemma 16]{SzaszVarju07}, \cite[Lemma 5.1]{ChernovZhang08}} ]
\label{lem:relative}
There are constants $p,q>0$ with the following property.
Define
\[
X_b(m)=\big\{x\in X:[h(x)]=m\enspace \text{and}\enspace h(T^jx)>m^{1-q}\enspace \text{for some $j\in\{1,\dots,b\log m\}$}\big\}.
\]
Then for any $b$ sufficiently large there is a constant $C=C(b)>0$ such that
\[
 \mu_X(X_b(m)) \le Cm^{-p}\mu_X(x\in X:[h(x)]=m)
\quad\text{for all $m\ge1$.}
\]

\vspace{-5ex}
\qed
\end{lemma}

For $b>0$, define
\[
Y_b(n)=\{y\in Y: \tau(y)\le b\log n\enspace\text{and}\enspace \max_{0\le\ell<\tau(y)}h(T^\ell y)\le {\textstyle\frac12 n}\enspace\text{and}\enspace
\varphi(y)\ge n\}.
\]

\begin{cor} \label{cor:Yb}
For $b$ sufficiently large, $\mu(Y_b(n))=o(n^{-2})$.
\end{cor}

\begin{proof}
Fix $p$ and $q$ as in Lemma~\ref{lem:relative}.
Also fix $b$ sufficiently large.

Let $y\in Y_b(n)$.
Define $h_1(y)=\max_{0\le\ell<\tau(y)}h(f^\ell y)$ and choose $\ell_1(y)\in\{0,\dots,\tau(y)-1\}$ such that $h_1(y)=h(f^{\ell_1(y)}y)$.
Define $h_2(y)=
\max_{0\le\ell<\tau(y),\,\ell\neq\ell_1(y)}h(f^{\ell} y)$.
Then $h_1(y)$ and $h_2(y)$ are the two largest free flights $h\circ f^\ell$ during the iterates $\ell=0,\dots,\tau(y)-1$.

We begin by showing that these two flight times have comparable length.
 Indeed, let $m_i=[h_i]$, $i=1,2$.
Then $n\le \varphi\le h_1+(\tau-1)h_2\le n/2+(b\log n)h_2$.
Hence
\begin{equation} \label{eq:relative}
\frac{n}{2b\log n}-1\le m_2\le m_1\le \frac{n}{2}.
\end{equation}
In particular, $m_1>m_2^{1-q}$ and $m_2> m_1^{1-q}$ for large $n$.

Choose $\ell_2(y)\in\{0,\dots,\tau(y)-1\}$ such that $\ell_2(y)\neq\ell_1(y)$ and $h_2(y)=h(f^{\ell_2(y)}y)$.
We can suppose without loss that $\ell_1(y)<\ell_2(y)$.
For large $n$, it follows from~\eqref{eq:relative} that
$f^{\ell_1(y)}y\in X_b(m_1(y))$.
Hence
\[
Y_b(n)\subset f^{-\ell}X_b(m)
\quad\text{for some $\ell<b\log n$, $m\ge n/(2b\log n)-1$,}
\]
and so
\[
\mu(Y_b(n))\ll \mu_X(Y_b(n)\times0)\le b\log n\sum_{m\ge n/(2b\log n)-1}\mu_X(X_b(m)).
\]
By Lemma~\ref{lem:relative} and~\eqref{eq:htail},
\begin{align*}
\mu(Y_b(n)) & \ll \log n \sum_{m\ge n/(2b\log n)-1}m^{-p}\mu_X(x\in X:[h(x)]=m)
\\ & \ll \log n (n/\log n)^{-(2+p)}=o(n^{-2}),
\end{align*}
as required.
\end{proof}

\begin{pfof}{Proposition~\ref{prop:gas}}
Define the tower $\Delta=\{(y,\ell)\in Y\times\Z:0\le \ell\le\tau(y)-1\}$
with probability measure $\mu_\Delta=\mu\times{\rm counting}/\bar\tau$
where $\bar\tau=\int_Y\tau\,d\mu$.
Recall that $\mu_X=\pi_*\mu_\Delta$ where $\pi(y,\ell)=f^\ell y$.

Write
$\max_{0\le\ell<\tau(y)}h(f^\ell y)=h(f^{\ell_1(y)}y)$ where $\ell_1(y)\in\{0,\dots,\tau(y)-1\}$.
Then
\begin{align*}
\mu\{y\in &  Y:  \max_{0\le \ell<\tau(y)}h(f^\ell y)>n/2\} =
\bar \tau\mu_\Delta\{(y,0)\in \Delta:h(f^{\ell_1(y)}y)>n/2\} \\ &
=
\bar \tau\mu_\Delta\{(y,\ell_1(y)):h(f^{\ell_1(y)}y)>n/2\} =
\bar \tau\mu_\Delta\{(y,\ell_1(y)):h\circ\pi(y,\ell_1(y))>n/2\}
\\ & \le
\bar \tau\mu_\Delta\{p\in\Delta:h\circ\pi(p)>n/2\}
=\bar \tau\mu_X\{x\in X:h(x)>n/2\},
\end{align*}
and so $\mu\{y\in Y:\max_{0\le \ell<\tau(y)}h(T^\ell y)>n/2\}
=O(n^{-2})$ by~\eqref{eq:htail}.
Hence it follows from Corollary~\ref{cor:Yb} that
\[
\mu\{y\in Y: \tau(y)\le b\log n\enspace\text{and}\enspace
\varphi(y)\ge n\}=O(n^{-2}).
\]

Finally, by~\eqref{eq:tautail},
$\mu(\tau >b\log n)=O(n^{-bc})=o(n^{-2})$ for any $b>2/c$
and so $\mu(\varphi\ge n)=O(n^{-2})$ as required.
\end{pfof}

It follows from Lemma~\ref{lem:bounded} and Corollary~\ref{cor:Lorentz} that
condition~(H) is satisfied.  Hence by Corollary~\ref{cor:nonskew}(a),
the
suspension flow $F_t:Y^\varphi\to Y^\varphi$ is a Gibbs-Markov flow as defined in Section~\ref{sec:nonskew}.
By Proposition~\ref{prop:gas}, $\mu(\varphi>t)=O(t^{-2})$.
By Corollary~\ref{cor:dyn}, the flows and observables are dynamically H\"older (Definition~\ref{def:dyn}).  
Hence it follows from Corollary~\ref{cor:nonskew}(b)
that absence of approximate eigenfunctions implies decay rate $O(t^{-1})$.

Finally, we exclude approximate eigenfunctions.  By Corollary~\ref{cor:Lorentz},
condition~\eqref{eq:Wubounded} holds and hence the temporal distortion function $D:Y\times Y\to\R$ is defined as in Section~\ref{sec:D}.
Let $\bZ_0\subset\bY$ be a finite subsystem and let $Z_0=\bar\pi^{-1}\bZ_0$.
The presence of a contact structure implies by Remark~\ref{rmk:contact} that
the lower box dimension of $D(Z_0\times Z_0)$ is positive.
Hence absence of approximate eigenfunctions follows from Lemma~\ref{lem:D}.

\subsection{Semi-dispersing Lorentz flows and stadia}
\label{sec:Stadia}

In this subsection we discuss two further classes of billiard flows and show that the scheme presented above
can be adapted to cover these examples, resulting in Theorem~\ref{thm:stadia}.

\textit{Semi-dispersing Lorentz flows} are billiard flows in the planar domain obtained as $R \setminus \bigcup S_k$ where $R$ is a rectangle and the
$S_k\subset R$ are finitely many disjoint convex scatterers with $C^3$ boundaries of nonvanishing curvature. By the unfolding process  -- tiling the plane with identical copies of $R$, and reflecting
the scatterers $S_k$ across the sides of all these rectangles -- an infinite periodic configuration is obtained, which can be regarded as an infinite horizon
Lorentz gas.

\textit{Bunimovich stadia} are convex billiard domains enclosed by two semicircular arcs (of equal radii) connected by two parallel line segments.
An unfolding process could reduce the bounces on the parallel line segments to long flights in an unbounded domain, however, there is another quasi-integrable
effect here corresponding to sequences of consecutive collisions on the same semi-circular arc.

Both of these examples have been extensively studied in the literature, see for instance \cite{Bunimovich79,ChernovMarkarian06,ChernovZhang05,M09,BalintMelbourne08}, and references therein. A common feature of the two examples is that the billiard map itself is not uniformly hyperbolic; however, there is a geometrically defined
first return map which has uniform expansion rates. As before, the billiard
domain is denoted by $Q$, and the billiard flow is $T_t:M\to M$ where $M=Q\times \bS^1$.  However, this time
we prefer to denote the natural Poincar\'e section  $\partial Q\times[-\pi/2,\pi/2]\subset M$ by $\tX$,
the corresponding billiard map as $\tf:\tX\to\tX$, and the free flight function as $\th:\tX\to\R^+$ where $\th(\tx)=\inf\{t>0:T_t \tx\in \tX\}$.
Then, as mentioned above, there is a subset
$X\subset \tX$ such that the first return map of $\tf$ to $X$ has good hyperbolic properties. We denote this first return map by $f:X\to X$.
The corresponding
free flight function $h:X\to\R^+$ is given by $h(x)=\inf\{t>0:T_tx\in X\}$. Let us, furthermore, introduce the discrete return time $\tr:X\to\Z^+$ given by $\tr(x)=\min\{n\ge 1: \tf^n x\in X\}$.

In the case of the semi-dispersing Lorentz flow, $X$ corresponds to collisions on the scatterers  $S_k$. In the case of the stadium,
$X$ corresponds to first bounces on semi-circular arcs, that is, $x\in X$ if $x$ is on one of the semi-circular arcs, but $\tf^{-1} x$ is on another boundary component
(on the other semi-circular arc, or on one of the line segments).

The following properties hold. Unless otherwise stated, standard references are \cite[Chapter 8]{ChernovMarkarian06} and \cite{ChernovZhang05}. As in section~\ref{sec:BackLorentz}, $d(x,x')$ always denotes the Euclidean distance of the two points, generated by the Riemannian metric.

\begin{itemize}
\item There is a countable partition $X\setminus\cS=\bigcup_{m=1}^{\infty} X_m$ such that $f|_{X_m}$ is $C^2$ and $\tr|_{X_m}$ is constant for any $m\ge 1$.
We refer to the partition elements $X_m$ with $\tr|_{X_m}\ge 2$ as {\em cells}; these are of two different types:
\begin{itemize}
\item \textit{Bouncing cells} are present both in the semi-dispersing billiard examples and in stadia. For these, one iteration of $f|_{X_m}$  consists of several consecutive reflections
on the flat boundary components, that is, the line segments. By the above mentioned unfolding process, these reflections reduce to trajectories along straight lines in the associated unbounded table.
\item  \textit{Sliding cells} are present only in stadia. For these, one iteration of $f|_{X_m}$  consists of  several consecutive collisions on
the same semi-circular arc.
\end{itemize}
\item $\inf h>0$, and $\sup \th<\infty$, however, there is no uniform upper bound on $h$, and no uniform lower bound for $\th$.
\item $f:X\to X$ is uniformly hyperbolic in the sense that stable and unstable manifolds exist for almost every $x$, and Formulas \eqref{eq:Wsf} and \eqref{eq:Wuf} hold. This follows from the uniform expansion rates of $f$, see \cite[Formula (8.22)]{ChernovMarkarian06}.
\item If $x,x'\in X_m$ where $X_m$ is a bouncing cell, in the associated unfolded table the flow trajectories until the first return to $X$ are
straight lines, hence \eqref{eq:bilroof} follows. If $x,x'\in X_m$ and $X_m$ is a sliding cell, the induced roof function is uniformly H\"older continuous with exponent $1/4$, as established in the proof of \cite[Theorem 3.1]{BalintMelbourne08}.
The same geometric reasoning applies to $\th_k(x)=\th(x)+\th(\tf x)+\dots+\th(\tf^{k-1}x)$ as long as $k\le\tr(x)$. Summarizing, we have
\begin{equation}
\label{eq:bilroof_stad}
 |\th_k(x)-\th_k(x')|\ll d(x,x')^{1/4} +d(f x,f x')^{1/4}
\end{equation}
 for $x,x'\in X_m,\ m\ge 1$ and $k\le\tr(x)-1$.
In particular, $|h(x)-h(x')|\ll d(x,x')^{1/4} +d(f x,f x')^{1/4}$.
\item \eqref{eq:bilroofv} has to be relaxed to
\begin{equation}\label{eq:bilroofv_stad}
d(T_t\tx,T_{t'}\tx)\le |t-t'| \quad\text{for\ all}\ \tx\in \tX\ \text{and\ } t,t'\in [0,\th(\tx)).
\end{equation}
\item \eqref{eq:bilroof2s} has to be relaxed to the following two formulas:
\begin{align}\label{eq:bilroof2s_stad}
d(T_t \tx,T_t \tx')  \ll d(\tx,\tx')
\quad\text{for}\
& \tx\in\tX, \tx'\in W^s(\tx),\  t\in[0,\th(\tx))\cap[0,\th(\tx'));\\
\label{eq:bilroof2s_stadv}
d(\tf^k x,\tf^k x') \ll d(x,x')
\quad\text{for}\  &
x\in X, x'\in W^s(x),\  0\le k.
\end{align}
Similarly, \eqref{eq:bilroof2u} has to be relaxed to
\begin{align}
\label{eq:bilroof2u_stad}
 d(T_{-t} \tx,T_{-t} \tx')  \ll d(\tx,\tx')
\quad\text{for}\
 & \tx\in\tX, \tx'\in W^u(\tx), \nonumber \\
 & t\in[0,\th(\tf^{-1}\tx))\cap[0,\th(\tf^{-1}\tx'));\\
\label{eq:bilroof2u_stadv}
d(\tf^{-k} x,\tf^{-k} x') \ll d(x,x')
\quad\text{for}\
& x\in X, x'\in W^u(x),\ 0\le k.
\end{align}
To verify \eqref{eq:bilroof2u_stadv}, let us note first that $d(x,x')$ consists of a position and a velocity component, and in course of a free flight velocities do not change. Now the mechanism of hyperbolicity for stadia is defocusing, see, for instance, \cite[Figure 8.1]{ChernovMarkarian06}, which guarantees that for $x'\in W^u(x)$, the position component of $d(x,x')$ in course of the free flight is dominated by the position component at the end of the free flight. \eqref{eq:bilroof2s_stadv} holds for analogous reasons. To verify \eqref{eq:bilroof2u_stad}, by uniform hyperbolicity of $f$ (in particular Formula \eqref{eq:Wuf}, see above), it is enough to consider how $\tf$ evolves unstable vectors between two consecutive applications of $f$, ie.~within a series of sliding or bouncing collisions. On the one hand, again by the defocusing mechanism, $\tf$ does not contract the p-length of unstable vectors, see \cite[Section 8.2]{ChernovMarkarian06}. On the other hand, for an unstable vector, the ratio of the Euclidean and the p-length is $\sqrt{1+\mathcal{V}^2}/\cos\varphi$, where $\mathcal{V}$ is the slope of the unstable vector in the standard billiard coordinates, and $\varphi$ is the collision angle, see \cite[Formula (8.21)]{ChernovMarkarian06}. Now $|\mathcal{V}|$ is uniformly bounded away from $\infty$, see Formula
\cite[Formula (8.18)]{ChernovMarkarian06}, while $\cos\varphi$ is constant in course of a sequence of consecutive sliding or bouncing collisions. \eqref{eq:bilroof2s_stad} holds by an analogous argument.
 
\item The map $f:X\to X$ can be modeled by a Young tower with exponential tails. In particular, there exists a subset $Y\subset X$ and an induced map
$F=f^\tau:Y\to Y$ that possesses the properties discussed in Section~\ref{sec:NUH} including~\eqref{eq:occas}.
The tails of the return time $\tau:Y\to \Z^+$ are exponential, i.e.\ $\mu(\tau>n)=O(e^{-cn})$ for some $c>0$.\footnote{It is important to note that here $\tau$ is the return time to $Y$ in terms of $f$; the return time in terms of $\tf$ has polynomial tails.}
Moreover, the construction can be carried out so that $\diam Y$ is as small as desired. The existence of the Young tower satisfying these properties is established in \cite{ChernovZhang05}.
As in subsection~\ref{sec:BackLorentz}, we introduce the induced roof function $\varphi=\sum_{\ell=0}^{\tau-1} h\circ f^\ell$.
\item By construction, for $y,y'\in Y_j$, $j\ge 1$ and $\ell\le \tau$ fixed, $f^{\ell}y$ and $f^{\ell}y'$ always belong to the same cell of $X$.
\end{itemize}

Let us introduce $\hat{\gamma}=\gamma^{1/4}$ and $\bdy=d(y,y')^{1/4}$. The following version of Proposition~\ref{prop:hell2} holds.

\begin{prop}\label{prop:hell2_stad}
For all $y,y'\in Y_j$, $j\ge1$, and all $0\le\ell\le \tau(y)-1$,
\[
|h(f^\ell y)-h(f^\ell y')|\ll \hat{\gamma}^\ell (\bdy+\hat{\gamma}^{\tau(y)-\ell}\hat{\gamma}^{s(y,y')}.
\]
\end{prop}

\begin{proof}
The proof of Proposition~\ref{prop:hell2} applies, using \eqref{eq:bilroof_stad} instead of \eqref{eq:bilroof}.
\end{proof}

This readily implies

\begin{cor} \label{cor:Lorentz_stad}
Conditions~\eqref{eq:Wsbounded} and~\eqref{eq:Wubounded} hold, with $\gamma$ replaced by $\hat{\gamma}$, and $d(y,y')$ replaced by $\bdy$.  \qed
\end{cor}

The adapted version of Proposition~\ref{prop:dyn} reads as follows.

\begin{prop}\label{prop:dyn_stad}
For $\diam Y$ sufficiently small, there exist an integer $n_0\ge 1$ and a constant $C>0$ such that for all $y,y'\in Y$, $s(y,y_0)\ge n_0$,
and all $u\in[0,\varphi(y)]\cap[0,\varphi(y')]$, there exist $t,t' \in \R$ such that
\begin{alignat*}{2}
 |t-u| & \le C\bdy, & \qquad d(T_uy, T_{t}z) & \le C\bdy,
 \\
|t'-u| & \le  C\hat{\gamma}^{s(y,y')}, & \qquad d(T_uy', T_{t'}z) & \le C\hat{\gamma}^{s(y,y')},
\end{alignat*}
where $z=W^s(y)\cap W^u(y')$.
\end{prop}

\begin{proof}
First, \eqref{eq:hell} can be updated as
\begin{equation} \label{eq:hell_stad}
|h_\ell(y)-h_\ell(y')|\le \sum_{j=0}^{\tau(y)-1}|h(f^jy)-h(f^jy')|\ll \bdy+\hat{\gamma}^{s(y,y')},
\end{equation}
for $0\le \ell\le\tau(y)$.

Fix $y,y'\in Y_j$ for some  $j\ge 1$, and $u\in[0,\varphi(y)]\cap[0,\varphi(y')]$.
We will focus on choosing the appropriate $t$ and obtaining the relevant estimates. The choice of $t'$ is analogous.
Recall the notation $\bd=d(y,z)^{1/4}$ and note that $\bd\ll \bdy$.

\textit{First adjustment.} As in the proof of Proposition~\ref{prop:dyn}, we arrive at $T_uy=T_rf^{\ell}y$ and
$T_{t_1}z=T_{r_1}f^{\ell}z$ for the same $0\le\ell\le\tau(y)-1$, and such that $|u-t_1|\ll \bd$ and $|r-r_1|\ll \bd$.
Indeed, a priori we have $T_uy=T_rf^{\ell}y$ and $T_uz=T_{r'}f^{\ell'}z$, where, as $\inf h>0$, shrinking $\diam Y$ if needed,
\eqref{eq:hell_stad} implies $|\ell-\ell'|\le 1$. If $\ell=\ell'$, then let $t_1=u$, $r_1=r'$, and $|r-r_1|\ll \bd$ follows from
\eqref{eq:hell_stad}. If $\ell'=\ell-1$, then  $T_uz=T_{-r^*}f^{\ell}z$, where $r^*=h(f^{\ell-1}z)-r'\in[0,h(f^{\ell-1}z)]$.
Note that $u=h_\ell (y)+r=h_\ell (z)-r^*$, hence $r+r^*=h_\ell (z)-h_\ell (y)\ll \bd$.
Let $t_1=u+r+r^*$, so that $|t_1-u|\ll \bd$ and $r_1=r$ as $T_{t_1}z=T_r f^{\ell} z$.
Note that we do not claim anything about $d(T_uy,T_{t_1}z)$ at this point.

\textit{Second adjustment.}  For brevity, introduce $\hy=f^{\ell}y$ and $\hz=f^{\ell}z$. We have
\[
T_u y= T_r \hy = T_s \tf^k \hy,\qquad T_{t_1} z= T_{r_1} \hz = T_{s'} \tf^{k'} \hz,
\]
for some  $0\le k,k'\le \tr(\hy)-1$ (note that $\tr(\hy)=\tr(\hz)$), $s\in[0,\th(\tf^k\hy))$ and $s'\in[0,\th(\tf^{k'}\hz))$.
Note that by \eqref{eq:bilroof2s_stad}, \eqref{eq:bilroof2s_stadv} and \eqref{eq:Wsf}, for any $0\le k\le \tr(\hy)-1$, we have
\begin{equation}
\label{eq:bilroofd_stad}
d(\tf^k\hy,\tf^k\hz)\ll d(\hy,\hz)\ll d(y,z), \quad \text{hence} \quad |\th_k(\hy)-\th_k(\hz)|\ll \bd,
\end{equation}
where we have used \eqref{eq:bilroof_stad}.
We distinguish three cases: $k=k'$, $k>k'$ and $k<k'$.

If $k=k'$, 
\eqref{eq:bilroofd_stad} along with $|r-r_1|\ll \bd$ implies $|s-s'|\ll\bd$. But then, again by \eqref{eq:bilroofd_stad}, \eqref{eq:bilroof2s_stad} and \eqref{eq:bilroof2s_stadv}, we have
\[
d(T_u y,T_{t_1}z)=d(T_s \tf^k\hy,T_{s'} \tf^k\hz)\ll \bd.
\]
As $|u-t_1|\ll \bd$, we can fix $t=t_1$.

If $k>k'$, we prefer to represent our points as
\[
T_u y= T_r \hy = T_s \tf^k \hy,\qquad T_{t_1} z= T_{r_1} \hz = T_{-s_1} \tf^{k} \hz
\]
for some $s_1>0$. Now by \eqref{eq:bilroofd_stad} and as $|r-r_1|\ll \bd$, we have $s+s_1\ll \bd$. Define
\[
s_2=\min(s,\th(\tf^k\hz)/2,\th(\tf^k\hy)/2),\quad r_2=s_2+s_1+r_1,\quad t=s_2+s_1+t_1.
\]
Then $T_t z=T_{s_2} \tf^{k} \hz$, where $s_2\in [0,\th(\tf^k\hy))\cap [0,\th(\tf^k\hz))$ and
\[
|s-s_2|\le s \le s +s_1 \ll \bd.
\]
Hence
\[
d (T_u y, T_t z)= d(T_s \tf^k \hy, T_{s_2} \tf^{k} \hz)\le d(T_{s_2} \tf^k \hy,T_{s_2} \tf^{k} \hz) + d(T_s \tf^k \hy, T_{s_2} \tf^{k} \hy),
\]
where $d(T_s \tf^k \hy, T_{s_2} \tf^{k} \hy) \ll \bd$ by \eqref{eq:bilroofv_stad}, while
$d(T_{s_2} \tf^k \hy,T_{s_2} \tf^{k} \hz)\le \bd$ by \eqref{eq:bilroof2s_stad}, \eqref{eq:bilroof2s_stadv} and \eqref{eq:bilroofd_stad}. Hence $d (T_u y, T_t z)\ll \bd$, as desired.
On the other hand $|t-t_1|=s_1+s_2\le s_1+s\ll \bd$, and as we have already controlled $|t_1-u|$, we have $|t-u|\ll \bd$.

The case when  $k<k'$ can be treated analogously. The choice of $t'$ goes along similar lines, so we omit the details.
\end{proof}

\begin{thm} \label{thm:stadia}
Consider a semi-dispersing Lorentz flow or the billiard flow in a Bunimovich stadium. Let $\eta\in(0,1]$.
There exists $m\ge1$ such that
$\rho_{v,w}(t)=O(t^{-1})$
for all $v\in C^\eta(M)\cap C^{0,\eta}(M)$ and $w\in C^{\eta,m}(M)$ (and more generally for the class of observables defined in Corollary~\ref{cor:dyn}).
\end{thm}

\begin{proof}
It follows from Lemma~\ref{lem:bounded} and Corollary~\ref{cor:Lorentz_stad} that
condition~(H) is satisfied.  Hence by Corollary~\ref{cor:nonskew}(a),
the
suspension flow $F_t:Y^\varphi\to Y^\varphi$ is a Gibbs-Markov flow as defined in Section~\ref{sec:nonskew}.
The conclusions of Corollary~\ref{cor:dyn} follow from Propositions~\ref{prop:hell2_stad} and~\ref{prop:dyn_stad}.
Hence the flows and observables are dynamically H\"older (Definition~\ref{def:dyn}).  

For the tail estimate on $\varphi$,  introduce $\ttau:Y\to\Z^{+}$, $\ttau(y)=\min\{n\ge 1: \tf^ny\in Y\}$.  Note that $\sup \th<\infty$, and
$\varphi(y)=\sum_{k=0}^{\ttau(y)-1} \th(\tf^ky)\le \ttau(y) \sup \th$ .
Also it is shown in
\cite{ChernovZhang08} (both for the semi-dispersing examples and for stadia)
that
$\mu(\ttau>n)=O(n^{-2})$.
Hence $\mu(\varphi>t)\le \mu(\ttau \sup\th >t ) = O(t^{-2})$.

Finally, to exclude approximate eigenfunctions,
we may appeal as at the end of Section~\ref{sec:tail}
to the contact structure which the billiard examples have in common.
The result now follows from Corollary~\ref{cor:nonskew}(b).
\end{proof}

\subsection{Lower bounds}
\label{sec:lower}

In this subsection, we show that it is impossible to improve on the error rate
$O(t^{-1})$ for infinite horizon Lorentz gases, semidispersing Lorentz flows, and Bunimovich stadia.
The following result is based on~\cite[Corollary~1.3]{BalintGouezel06}.

\begin{prop} \label{prop:lower}
Let $v\in L^2(M)$ with $\int_M v\,d\mu_M=0$.  Suppose that $\rho_{v,v}(t)= o(t^{-1})$.
Then $|\int_0^t v\circ T_s\,ds|_2=o((t\log t)^{1/2})$.
\end{prop}

\begin{proof}
Let $v_t=\int_0^t v\circ T_s\,ds$.  Then
\begin{align*} \int_M v_t^2\,d\mu_M & = \int_0^t\int_0^t \int_M v\circ T_r\,v\circ T_s\,d\mu_M\,dr\,ds = 2\int_0^t\int_0^s \int_M v\,v\circ T_{s-r}\,d\mu_M\,dr\,ds
\\ & =2 \int_0^t\int_0^s \rho_{v,v}(r)\,dr\,ds
=2 \int_0^t\int_r^t \rho_{v,v}(r)\,ds\,dr
\le 2t \int_0^t \rho_{v,v}(r)\,dr.
\end{align*}
By the assumption on $\rho_{v,v}$, we obtain
$|v_t|_2^2=o(t\log t)$.
\end{proof}

In the case of the planar infinite horizon Lorentz gas,
Sz\'asz \& Varj\'u~\cite{SzaszVarju07} showed
that $(t\log t)^{-1/2}\int_0^t v\circ T_s\,ds$ converges in distribution to
a nondegenerate normal distribution for typical H\"older mean zero observables $v$.
The result applies equally to semidispersing Lorentz flows.
Similarly, in the case of Bunimovich stadia by
B\'alint \& Gou\"ezel~\cite[Corollary~1.6]{BalintGouezel06}.
In particular, $(t\log t)^{-1/2}|\int_0^t v\circ T_s\,ds|_2\not\to0$.
Hence by Proposition~\ref{prop:lower}, an upper bound of the type $o(t^{-1})$ is impossible and so the upper bound in Theorems~\ref{thm:Lorentz} and~\ref{thm:stadia} is optimal.

\begin{rmk}  There is also the possibility of obtaining an asymptotic expression of the form
\begin{equation} \label{eq:MT}
\rho_{v,w}(t)=ct^{-1}+O(t^{-(2-\eps)}),
\end{equation}
($\eps>0$ arbitrarily small, $c>0$) for certain classes of observables $v,w$.  Such results are obtained in~\cite{MT17} in cases where there is a first return to a uniformly hyperbolic map $f:X\to X$.  The first return map in the examples considered here is nonuniformly hyperbolic, modelled by a Young tower with exponential tails, so~\cite{MT17} does not apply directly.
In a recent preprint,~\cite{CWZ} have announced the existence of a uniformly hyperbolic first return.  This combined with~\cite{MT17} may yield the asymptotic~\eqref{eq:MT}.
(Interestingly, the class of observables in~\eqref{eq:MT} would be disjoint from the class of observables covered by Proposition~\ref{prop:lower}.)
\end{rmk}

\appendix

\section{Condition~\eqref{eq:occas}}

In this appendix, we verify that condition~\eqref{eq:occas} holds in the abstract framework of~\cite{Young98}.
For this purpose, we switch to the notation of~\cite{Young98}.

\begin{prop} \label{prop:occas}
Let $f:\Lambda\to\Lambda$ be an injective transformation
satisfying the abstract set up in~\cite[Section~1]{Young98}:
specifically (P1), the second part of (P2), property~(iii) of the separation time $s_0$, and (P4)(a).

Let $x\in \Lambda_i$, $i\ge1$.
Then
$f^{R_i}(\gamma^u(x)\cap \Lambda_i)=\gamma^u(f^{R_i}x)\cap \Lambda$.
\end{prop}

\begin{proof}
It follows from injectivity of $f$ and hence $f^{R_i}$, as well as (P2), that
\begin{equation} \label{eq:1-1}
f^{R_i}(\gamma^u(x)\cap \Lambda_i)=
f^{R_i}\gamma^u(x)\cap f^{R_i}\Lambda_i
\supset \gamma^u(f^{R_i}x)\cap f^{R_i}\Lambda_i.
\end{equation}

Recall from (P1) that we have the local product structure
$\Lambda=\big(\bigcup_{k\in K^u}\gamma_k^u\big)
\cap\big(\bigcup_{\ell\in K^s}\gamma_\ell^s\big)$.
By (P2), $f^{R_i}\Lambda_i$ is a $u$-subset of $\Lambda$ which means that
$f^{R_i}\Lambda_i=\big(\bigcup_{k\in K_i^u}\gamma_k^u\big)
\cap\big(\bigcup_{\ell\in K^s}\gamma_\ell^s\big)$
for some subset $K_i^u\subset K^u$.
Hence
$\gamma_k^u\cap \Lambda=\gamma_k^u\cap\big(\bigcup_{\ell\in K^s}\gamma_\ell^s\big)=\gamma_k^u\cap f^{R_i}\Lambda_i$
for all $k\in K_i^u$.
Also, $\gamma_k^u\cap f^{R_i}\Lambda_i=\emptyset$ for all $k\not\in K_i^u$.

Now, $\gamma^u(f^{R_i}x)\cap f^{R_i}\Lambda_i\neq\emptyset$ (it contains
$f^{R_i}x$) so
it follows from the above considerations that
$\gamma^u(f^{R_i}x)\cap \Lambda=\gamma^u(f^{R_i}x)\cap f^{R_i}\Lambda_i$.
Combining this with~\eqref{eq:1-1},
\begin{equation} \label{eq:supset}
f^{R_i}(\gamma^u(x)\cap \Lambda_i)
\supset \gamma^u(f^{R_i}x)\cap\Lambda.
\end{equation}

It remains to prove the reverse inclusion, so suppose that
$y\in \gamma^u(x)\cap\Lambda_i$.
By (P1), there exists $z^*\in \gamma^u(f^{R_i}x)\cap\gamma^s(f^{R_i}y)\subset\Lambda$.
By~\eqref{eq:supset}, $z^*=f^{R_i}z$ for some $z\in \gamma^u(x)\cap \Lambda_i$.

Since $z^*$ and $f^{R_i}y$ lie in the same stable disk it follows
from property~(iii) of the separation time that
$s_0(z^*,f^{R_i}y)=\infty$.  Using property~(iii) once more,
\mbox{$s_0(z,y)\ge s_0(z^*,f^{R_i}y)=\infty$}.
But $z\in\gamma_u(x)=\gamma_u(y)$ so (P4)(a) implies that
$d(z,y)\le C\alpha^{s_0(z,y)}=0$.  Hence $f^{R_i}y=f^{R_i}z=z^*\in\gamma^u(f^{R_i}x)$.
This shows that
$f^{R_i}(\gamma^u(x)\cap \Lambda_i)
\subset \gamma^u(f^{R_i}x)\cap\Lambda$ completing the proof.
\end{proof}

\paragraph{Acknowledgements}
The research of PB was supported in part by
Hungarian National Foundation for Scientific Research (NKFIH OTKA)
grants K104745 and K123782. OB was supported in part by EU Marie-Curie
IRSES Brazilian-European partnership in Dynamical Systems (FP7-PEOPLE-2012-IRSES 318999 BREUDS).
The research of IM was supported in part by a
European Advanced Grant {\em StochExtHomog} (ERC AdG 320977).

We are grateful to the referees for very helpful comments which led to many clarifications and corrections.

\def\polhk#1{\setbox0=\hbox{#1}{\ooalign{\hidewidth
  \lower1.5ex\hbox{`}\hidewidth\crcr\unhbox0}}}

\end{document}